\documentclass[11pt,twoside,a4paper]{article}  
\usepackage{amssymb,amsfonts,amsmath,amsbsy,theorem,amscd,dina42e}
\usepackage[T1]{fontenc}
\usepackage{color}
\newcommand{\dd}{\text{\it \dj\hspace{1pt}}}
\newcommand{\ul}[1]{\underline{#1}}
\newcommand{\ol}[1]{\overline{#1}}

\numberwithin{equation}{section}
\newcommand{\C}{\ensuremath{\mathbb{C}}}
\newcommand{\R}{\ensuremath{\mathbb{R}}}

\newcommand{\Rn}{\ensuremath{\mathbb{R}^n}}

\newcommand{\D}{\ensuremath{\mathcal{D}}}
\newcommand{\N}{\ensuremath{\mathbb{N}}}
\newcommand{\Z}{\ensuremath{\mathbb{Z}}}

\newcommand{\F}{\ensuremath{\mathcal{F}}}

\newcommand{\HS}{\ensuremath{\mathcal{H}}}
\newcommand{\SD}{\ensuremath{\mathcal{S}}}
\newcommand{\E}{\ensuremath{\mathcal{E}}}
\newcommand{\dist}{\operatorname{dist}}

\newcommand{\sd}{\, d}
\newcommand{\supp}{\operatorname{supp}}
\newcommand{\eps}{\ensuremath{\varepsilon}}
\newcommand{\weight}[1]{\langle #1\rangle}

\newcommand{\OP}{\operatorname{OP}}
\newcommand{\Os}{\operatorname{Os--}}

\newcommand{\habil}[1]{}
\newtheorem{thm}{Theorem}[section]
\newtheorem{cor}[thm]{Corollary}
\newtheorem{corollary}[thm]{Corollary}
\newtheorem{lem}[thm]{Lemma}
\newtheorem{lemma}[thm]{Lemma}

\newtheorem{definition}[thm]{Definition}
\newtheorem{theorem}[thm]{Theorem}
\newtheorem{prop}[thm]{Proposition}
\newtheorem{proposition}[thm]{Proposition}

\newtheorem{claim*}{Claim}

\theorembodyfont{\rmfamily}

\newtheorem{remark}[thm]{Remark}

\newtheorem{example}[thm]{Example}
\newenvironment{proof*}[1]{\noindent{\bf Proof
#1:}}{\hspace*{\fill}\rule{1.2ex}{1.2ex}\\ }
\newenvironment{proof}{\noindent{\bf
Proof:\,}}{\hspace*{\fill}\rule{1.2ex}{1.2ex}\\ }

\newcommand{\crp}{\overline{\mathbb R}_+}
\newcommand{\crm}{\overline{\mathbb R}_-}

\newcommand{\rn}{{\mathbb R}^n}
\newcommand{\rnp}{{\mathbb R}^n_+}

\newcommand{\rnpm}{\mathbb R^n_\pm}
\newcommand{\crnp}{\overline{\mathbb R}^n_+}
\newcommand{\crnm}{\overline{\mathbb R}^n_-}
\newcommand{\crnpm}{\overline{\mathbb R}^n_\pm}
\newcommand{\comega}{\overline\Omega }

\newcommand{\ang}[1]{\langle {#1} \rangle}

\newcommand{\simto}{\overset\sim\rightarrow}

\newcommand{\rp}{ \mathbb R_+}

\newcommand{\op}{\operatorname{OP}}
\newcommand{\OPK}{\operatorname{OPK}}

\newcommand{\stimes}{\!\times\!}

\begin{document}
\begin{titlepage}
\title{Fractional-Order Operators on Nonsmooth Domains} 
\author{Helmut~Abels\footnote{Fakult\"at f\"ur Mathematik,  
Universit\"at Regensburg,  
93040 Regensburg, Germany, E-mail {\tt helmut.abels@ur.de}}~ and Gerd~Grubb\footnote{Department of Mathematical Sciences,
Copenhagen University,
Universitetsparken 5,
 DK-2100 Copenhagen, Denmark.
E-mail {\tt grubb@math.ku.dk} 
}
}
\end{titlepage}
\date{}
\maketitle

\begin{abstract}
The fractional Laplacian
$(-\Delta )^a$, $a\in(0,1)$, and its generalizations to
variable-coefficient $2a$-order pseudodifferential
operators $P$, are studied in $L_q$-Sobolev spaces of
Bessel-potential type $H^s_q$. For a bounded open set $\Omega \subset
\mathbb R^n$, consider the homogeneous Dirichlet problem: $Pu
=f$ in $\Omega $, $u=0$ in $ \mathbb R^n\setminus\Omega $. We find the
regularity of solutions and determine the exact Dirichlet
domain $D_{a,s,q}$ (the space of solutions $u$ with $f\in
H_q^s(\overline\Omega )$) in cases where $\Omega $ has
limited smoothness $C^{1+\tau }$, for $2a<\tau <\infty $, $0\le s<\tau
-2a$. Earlier, the regularity and  Dirichlet
domains were determined for smooth $\Omega$  by the second author, and
the regularity was found in low-order H\"older spaces for $\tau =1$ by
Ros-Oton and Serra. The $H_q^s$-results obtained now when
$\tau <\infty $ are new, even for $(-\Delta )^a$. In detail, the spaces $D_{a,s,q}$ are
identified as $a$-transmission spaces $H_q^{a(s+2a)}(\overline\Omega
)$, exhibiting estimates in terms of
$\operatorname{dist}(x,\partial\Omega )^a$ near the boundary.


The result has required a new development of
methods to handle nonsmooth coordinate changes for pseudodifferential
operators, which have not been available before; this constitutes another
main contribution of the paper.
 \end{abstract}


 \section{Introduction}

 The present work has two main purposes. One is to solve the regularity
question and determine the domain for the Dirichlet problem for  $(-\Delta )^a$ and its fractional-order
 generalizations, in $L_q$-Sobolev spaces over domains of finite smoothness
 degrees between $C^{1,\sigma }$ and $C^\infty $. The other is to
 develop a tool that has been missing in the theory of
 pseudodifferential operators: Nonsmooth coordinate changes. It plays
 an important role in the solution of the regularity question.
\medskip

The fractional Laplacian $(-\Delta )^a$, $0<a<1$, and its
generalizations $P$ of the same order $2a$, have been much studied in
recent years, with applications in probability,
finance, differential geometry and mathematical physics. To mention some of the studies through the times:  
Blumenthal and Getoor \cite{BG59},  Hoh and Jacob
\cite{HJ96},  Kulczycki
\cite{K97}, Chen and Song \cite{CS98},
Jakubowski \cite{J02}, Bogdan, Burdzy and Chen \cite{BBC03}, Cont and
Tankov \cite{CT04}, 
Caffarelli and Silvestre \cite{CS07}, Gonzales, Mazzeo and Sire \cite{GMS12},
Ros-Oton and Serra \cite{RS14,RS16}, Abatangelo \cite{A15}, Felsinger, Kassmann and Voigt
\cite{FKV15}, Bonforte, Sire and Vazquez \cite{BSV15}, Dipierro, Ros-Oton and Valdinoci \cite{DRV17},  Abatangelo, Jarohs and Saldana
\cite{AJS18}. They refer to many more works, also with applications to
nonlinear problems.
 From its action on $\rn$ one defines
the {\it homogeneous Dirichlet problem:}
\begin{equation}
  \label{eq:1.4}
Pu=f\text{ on }\Omega ,\quad \supp u\subset \comega,
\end{equation}
on bounded open subset $\Omega $ of $\rn$ (with some boundary
regularity). 
For operators $P$ with $\operatorname{Re}\int_\Omega Pu\, \bar u\,dx>0$,
there is unique solvability
for $f\in L_2(\Omega )$ by a variational argument. One of the
fundamental  questions in then: How do the solutions look?
 The variational theory gives
that the solution belongs to $H^a$-functions supported in $\comega$,
but is $u$ in fact more regular? And will higher regularity of $f$
increase that of $u$? This is often called {\it the
  regularity question.}
Early results of Vishik and Eskin (see \cite{E81})
imply  e.g.\ for $a\ge \frac12$ that $u$ is $H^{a+\frac12-\varepsilon
}$. 
More precisely, one can ask: What is the {\it Dirichlet domain} for
$P$; the space of functions $u$ solving (1.1), when $f$ runs through a Sobolev space $H^s(\comega)$?

It was an important step forward when  Ros-Oton and Serra
\cite{RS14} showed that for  $f\in L_\infty (\Omega ) $, $u$
is H\"older-continuous with a singularity $d(x)^a$ at the boundary
(where $d(x)\sim \operatorname{dist}(x,\partial\Omega )$ near
 $\partial\Omega $, extended smoothly to a positive function on $\Omega $),
 \begin{equation}
   \label{eq:1.5}
   f\in L_\infty (\Omega )\implies u/d^a\in C^\alpha (\comega),
 \end{equation}
for small $\alpha >0$. $\Omega $ was assumed to be $C^{1,1}$; a later
study lifted $\alpha $ up to $a$. The methods were delicate
potential-theoretic arguments, based on the representation of
$(-\Delta )^a$ as a real singular integral operator; they were later
extended to other real translation-invariant singular integral operators.

A very different method was introduced by one of the present authors
\cite{G15}: Fourier analysis in the form of pseudodifferential
operator ($\psi $do) theory. This theory (necessarily for complex
functions) is designed to allow $x$-dependent operators (not
translation-invariant), taking care of the composition rules that arise,
which make the theory quite technical. Here it was shown when
$\Omega $ is a $C^\infty $-domain that if, say, $u\in \dot
H_q^a(\comega)$ and $s\ge 0$,
\begin{align}
  f&\in C^\infty (\comega)\iff u\in d^aC^\infty (\comega)\label{eq:1.6}\\
f&\in  H^s_q(\comega )\iff u\in D_{a,s,q},\label{eq:1.7}
\end{align}
where $D_{a,s,q}$ is a certain space contained in $\dot H_q^{s+2a}(\comega)+d^a
H_q^{s+a}(\comega )$
for $s+a-\frac1{q}>0$ (and $\notin \Z$), $1<q<\infty $; here $H^s_q$
denotes the $L_q$-Sobolev space of Bessel-potential type. For $0\le
s<\frac1q-a$, $D_{a,s,q}= \dot H_q^{s+2a}(\comega)$ (the $
H_q^{s+2a}(\rn)$-functions supported in $\comega$).

In detail, the space 
$D_{a,s,q}$  equals the so-called $a$-trans\-mis\-sion space
$H_q^{a(s+2a)}(\comega)$ introduced in \cite{G15}. \eqref{eq:1.7} was in \cite{G14} extended to a wealth of other
scales of function spaces, including H\"older-Zygmund spaces $C^s_*$; here $f\in
C^s(\comega)$ implies $u\in \dot
C^{s+2a}(\comega)+d^aC^{s+a}(\comega)$ (when $s,s+a,s+2a\notin \N$), the component in $d^aC^{s+a}(\comega)$
described more precisely in \cite{G19}. 

There is a large gap between the results \eqref{eq:1.5} and \eqref{eq:1.7}, in that the
former allows low smoothness of the domain $\Omega $ and
correspondingly shows smoothness of $u$ in a low range, whereas the latter, under a very high smoothness assumption
on $\Omega $, shows results in the full scale $s\ge 0$.

This gap remained open until recently. Ros-Oton with coauthor Abatangelo
presented a study \cite{ARO20} treating the regularity question for
translation-invariant real singular integral operators, 
lifting \eqref{eq:1.5} to $\alpha =a+s$ when $f\in C^{s}(\comega)$, $\Omega $
is $C^{1+\tau }$ with  $\tau >a$, for
$0<s\le\tau -a$ with $s,a+s, 2a+s\notin\N$.
This allows
 $\tau $ to be a step $a$ lower than we assume below, but does not exhibit
a more precise domain, and does not treat Sobolev spaces. \cite{ARO20} gives
important consequences for regularity questions for the obstacle problem.

One may also compare with the standard knowledge for  second order
strongly elliptic {\it differential} operators (cf.\ Gilbarg-Trudinger
\cite[Theorem 8.13]{GT77}), where $a=1$: When $\Omega $ is a $C^{2+s}$-domain, $s\in\N_0$, the Dirichlet
solutions for $f\in \ol H^s(\Omega )$ lie in $\ol H^{s+2}(\Omega )\cap
\dot H^1(\comega)$; this allows a first step lower regularity of $\Omega
$ than we do.

We shall in the present paper fill the gap in the category of
$L_q$-Sobolev spaces, showing that the characterizations 
\eqref{eq:1.7}
can be obtained also when $\Omega $ has
finite smoothness $C^{1+\tau }$, for $s$ in a finite interval whose
upper bound follows $\tau $
linearly. This is the first treatment of the fractional-order
operators acting in  $L_q$-Sobolev spaces over domains with limited but high smoothness, giving
correspondingly high regularity of solutions.

In this connection there  enters a condition on the behavior of the
operator $P$ at $\partial\Omega $, the $a$-transmission condition,
known from \cite{G15} to be necessary for \eqref{eq:1.6} on smooth
domains. Part of our work consists in finding the appropriate
generalization of this condition to nonsmooth domains, as well as
generalizing the
0-transmission condition of Boutet de Monvel
\cite{B66,B71} and of Grubb and H\"ormander \cite{GH90}. The
$a$-transmission condition is
satisfied in all directions  when $P$ of order $2a$ is {\it even},
i.e., its symbol $p(x,\xi )$ has a graded symmetry property in $\xi $,
cf.\ Section~\ref{subsec:PsDOs} below. This holds for $(-\Delta )^a$. We can now show a result that is new even for $(-\Delta )^a$:

\begin{theorem}\label{thm:main}
  Let $1<q<\infty $,  $0<a<1$, $\tau > 2a$, 
  let $\Omega \subset\rn$ be a bounded
$C^{1+\tau }$-domain, and let $P$ be an {\bf even} 
pseudodifferential
operator of order $2a$ with symbol depending $C^\tau $ on $x$ (i.e.,
$p\in C^\tau S^{2a}(\Rn\times\Rn)$, cf.\ Section~{\rm \ref{subsec:PsDOs}} below).

There is a family of  spaces $D_{a,s,q}$, equal to the
$a$-transmission spaces $ H_q^{a(s+2a)}(\comega)$ {(cf.~Definition~{\rm \ref{def:roughTransmSpaces} below)}},
such that the following holds:

$1^\circ$ $r^+P\colon u\mapsto (Pu)|_\Omega $ maps $D_{a,s,q}$
continuously into $H_q^s(\comega)$ for $0\le s<\tau -2a$.

$2^\circ$ Let $P$ moreover be strongly elliptic. If  $u\in 
  \dot{H}^a_q(\comega)$ is a solution of the homogeneous Dirichlet problem \eqref{eq:1.4}, then
  \eqref{eq:1.7} holds for $0\le s<\tau -2a$. This shows that $D_{a,s,q}$ is
  {\bf the
    Dirichlet domain} for $P$ with data in $H_q^s(\comega)$:
  \begin{equation}
D_{a,s,q}=\{u\in \dot H_q^a(\comega)\mid (Pu)|_\Omega \in H_q^s(\comega)\}.\label{eq:1.7a}
    \end{equation}
\end{theorem}

The definition of the $a$-transmission spaces
 $H_q^{a(t)}(\comega)$ over $C^{1+\tau }$-domains $\Omega $ ($t<\tau +1$) involves the order-reducing operators $\Xi _+^a$ connected with the $a$-transmission condition.
 Here  $ H_q^{a(s+2a)}(\comega)$ equals $\dot H_q^{s+2a}(\comega)$ if
 $s<\frac1q-a$, and is, if $\tau \ge 1$, included in $\dot H_q^{s+2a}(\comega)+d^a
 H_q^{s+a}(\comega )$ for larger $s$, with $d\in
 C^{1+\tau}(\overline{\Omega},\R_+)$ equal to the distance to
 $\partial\Omega$ in a neighborhood of $\partial\Omega$.
 The condition $\tau \ge 1$ assures that
 $d$ is differentiable in a neighborhood of $\partial\Omega $ and
 normal coordinates exist. In cases $2a<\tau <1$, which can only occur
 if $a<\frac12$, there is a local
 interpretation of the inclusion,
 cf. Remark~\ref{rem:TransmissionSpace}.
 
 We show furthermore that for $u\in H_q^{a(s+2a)}(\comega)$, the
 function $u/d^a$  has boundary value
 in $ B_q^{s+a -\frac1q}(\partial \Omega )$ when
 $s+a-\frac1q>0$, cf. Theorem~\ref{thm:TransmissionSpace} below.

It is a new result that
the exact Dirichlet domains $D_{a,s,q}$ have been found in cases where $\Omega $ has
limited smoothness. A remarkable fact is that these Dirichlet domains are
{\it universal}, depending on  $a$, $s$ and  $q$, but
not on the symbol of $P$ within the class of even, strongly elliptic
$\psi $do's with $C^\tau $-smooth symbols.
 
Our analysis has a number of other new applications for nonsmooth
domains, e.g.: \linebreak 1) Solution of evolution
problems by functional analysis as in \cite{G18JFA}, where the
determination of the Dirichlet domain leads to precise results. 2) Solution of
nonhomogeneous boundary value problems with local Dirichlet  conditions
as in \cite{G14}. They are  worked out in \cite{G22}. 
 
At the end, we include some consequences in H\"older and H\"older-Zygmund spaces that follow by use of embedding theorems by letting $q \to \infty$; this expands the scope of \cite{ARO20} to our classes of pseudodifferential operators.
  
  It would of course be interesting to extend the general principles to other scales of function spaces, namely the Triebel-Lizorkin spaces $F^s_{p,q}$ and the Besov spaces $B^s_{p,q}$ (including Hölder-Zygmund spaces $B^s_{\infty,\infty}=C^s_*$), as it was done in smooth cases in \cite{G14} after the treatment in Bessel potential spaces $H^s_q$ in \cite{G15}. But even for integer-order operators, the basic results \cite{A05} on pseudodifferential boundary problems with nonsmooth $x$-dependence have so far only been established in $H^s_q$-spaces (plus immediate consequences by interpolation), so a substantial additional work would be required.

\medskip

 The new results are based on a development of
 pseudodifferential theory that makes nonsmooth coordinate changes
 possible beyond the principal symbol level; this is the other main
contribution from the present work.

  Recall that pseudodifferential operators  were
 originally developed from singular integral operators as a systematic
 calculus  (containing differential operators) that could handle
 compositions of $x$-dependent operators 
and constructions of
 inverses in elliptic cases, by use of the mechanisms of Fourier
 transformation $\F$ (Seeley \cite{S65}, Kohn and Nirenberg \cite{KN65},
 H\"ormander \cite{H65}, and others). From a symbol $p(x,\xi )$
 depending smoothly on $x,\xi \in\rn$ (except possibly at $\xi =0$)
 one defines the operator $P=\op(p(x,\xi ))$ by
 \begin{equation}
   \label{eq:1.9}
  Pu(x)= (2\pi )^{-n}\int_{{\R} ^{n}}e^{ix\cdot\xi }p(x,\xi
)\F u(\xi )\,d\xi,
 \end{equation}
for nice functions $u$, extended to general $u$ as so-called oscillatory
integrals.

Theories for boundary value
problems for $\psi $do's,  resembling those for differential operators,
were soon set up in Boutet de Monvel  \cite{B66,B71} and further
developed in e.g.\
Rempel--Schulze \cite{RS82}, Grubb \cite{G84,G96}, acting on $C^\infty $-domains $\Omega \subset \rn$. The
definitions of appropriate symbol
classes were focused at first on the behavior with respect to $\xi $,
but after some years, nonsmooth behavior in $x$ was also introduced
(cf.\ e.g.\ Kumano-go and Nagase \cite{KN78}, Beals and Reed \cite{BealsReed},  Marschall
\cite{Marschall87,Marschall88}, Witt \cite{Witt98},  or Taylor \cite{TaylorNonlinPDE,ToolsForPDE}). Presently, we focus on symbols $p(x,\xi )$ in classes $C^\tau
S^m_{1,0}(\rn\times\rn)$, with $C^\tau $-smoothness in $x$ and standard
estimates in $\xi $. For boundary value problems, a nonsmooth generalization of the
calculus of Boutet de Monvel for integer-order $\psi $do's  was worked out in
Abels \cite{A05}.

In the study of boundary value problems, one needs not only the
$x$-dependence in the symbol to be less smooth than $C^\infty $; one
also needs to be able to apply the theory to domains $\Omega $ with nonsmooth
boundary.
When localization techniques are used, this means that one needs to
perform changes-of-variables with nonsmooth transition functions.
For some questions it is sufficient to do this only at the principal symbol
level, with estimates for the remainder operators
(since the result is possibly  sought in low-regularity spaces
anyway); this has been done e.g.\ in applications to the Navier-Stokes
problem (cf.\ Abels \cite{A05b}, with Terasawa \cite{AT09}) and spectral
theory (cf.\ Abels, Grubb, and Wood \cite{AGW14}). Invariance under smooth coordinate changes for nonsmooth pseudodifferential operators and boundary value problems were discussed by Jim\'{e}nez and the first author in \cite{AJ18}.

Full symbol rules for nonsmooth changes of variables have not to our
knowledge been established  anywhere (e.g., \cite{Marschall88} leaves out
this point, \cite{AbelsPfeufferCharacterization} works with a restrictive symbol condition that avoids
the issue.).
Changes of variables are of course also interesting for interior
problems, e.g.\ if one wants to let one coordinate play a special
role.

Since the behavior of a symbol under a coordinate change has a nice
 exact expression when one allows symbols ``in $(x,y)$-form'' (also
called amplitude functions), which
define operators by formulas
\begin{equation}
  \label{eq:1.10}
Au(x)=(2\pi )^{-n}\int_{{\R} ^{2n}}e^{i(x-y)\cdot\xi }a(x,y,\xi
)u(y)\,d\xi dy=\op(a(x,y,\xi ))u, 
\end{equation} 
the question of how to reduce an operator $\op(a(x,y,\xi ))$ to the
form $\op(p(x,\xi ))$ is intimately related to the
change-of-variables question. For this question one has to establish
the validity 
of \eqref{eq:1.10},  as well as set up the formula for the reduction with
appropriate
remainder terms, and to our knowledge
neither of these points has been treated before when the
$y$-dependence is nonsmooth. We shall show:

\begin{theorem}\label{thm:main2}
   Let $\tau >0$ and $m<\tau $.

 $1^\circ$ Let $a\in C^\tau S^m_{1,0}(\R^{2n}\times
  \Rn)$. Then \eqref{eq:1.10} has a meaning as
  an oscillatory integral, cf.\ Theorem~{\rm \ref{thm:OscIntegrals}} below for the details. It defines an operator $ \op(a(x,y,\xi))$
  mapping continuously
$$
 \op(a(x,y,\xi))\colon H_q^{s+m}(\rn)\to H_q^{s}(\rn),
$$
when $|m|$, $|s|$ and $|s+m|$ are $<\tau $, and $1<q<\infty $.
Moreover, for any nonnegative integer
  $l<\tau $,
 \begin{equation}\label{eq:1.12}
    \op(a(x,y,\xi))u(x) = \sum_{|\alpha|\leq l}  \op (p_\alpha (x,\xi)) u(x) +\op(r(x,y,\xi))u(x),
 \end{equation}
  where $p_\alpha(x,\xi)=\frac1{\alpha!}\partial_y^\alpha D_\xi^{\alpha} a(x,y,\xi)|_{y=x} \in C^{\tau-|\alpha|}S^{m-|\alpha|}_{1,0}(\Rn\times \Rn)$ and
  $r(x,y,\xi )$ \linebreak
  $\in C^{\tau-l}S^{m-l}_{1,0}(\R^{2n}\times
  \Rn)$, with $ r(x,x,\xi )=0$. Here the operators map continuously
  \begin{alignat*}{2}
    \op (p_\alpha (x,\xi))&\colon   H^{s+m-|\alpha |}_q (\R^n)\to
    H^s_q(\R^n)&\quad &\text{for }|s|<\tau-|\alpha |,\\
    \op(r(x,y,\xi))&\colon
H_q^{(s+m-l)_+}(\rn)\to H^s_q(\R^n)&&\text{for }0\le s<\min\{\tau-l,\tau-m\}.
  \end{alignat*}

  $2^\circ$ Let $p\in C^\tau S^m_{1,0}(\R^{n}\times
  \Rn)$. Under a $C^{1+\tau }$-diffeomorphism $F\colon \rn\to \rn$
  such that $c_0\le |\det(\nabla F(x))|\le C_0$ with $c_0,C_0>0$, $P$
  transforms to an operator $\underline P$,
$$
      \underline P = F^{\ast}P F^{\ast,-1} = \op(q(x,y,\xi))  + R_1 ,
$$
 where $q\in  C^\tau S^m_{1,0}(\R^{2n}\times
  \Rn)$, $R_1\colon L_q(\rn)\to H^s_q(\R^n)$ for $ s<\min\{\tau ,
 \tau -m\}$, and $\op(q(x,y,\xi))$ is as under $1^\circ$; in
 particular reducing to $x$-form as in {\rm \eqref{eq:1.12}}.
 \end{theorem}

We address the problem on subsets of $\R^n$, because that is what is asked for in the probabilistic and financial applications. Based on the work of Jim\'enez and the first author~\cite{AJ18} for nonsmooth operators on smooth manifolds, the various results are likely to carry over without trouble to suitable nonsmooth domains in smooth manifolds. The case where the basis manifold itself is nonsmooth would demand a larger effort because of various conditions between the regularities of the manifolds, the symbols and the order of the function spaces.

 The structure of the manuscript is as follows: In
Section~\ref{subsec:prelim} we summarize necessary preliminaries on
function spaces and (nonsmooth) pseudodifferential operators. The
first main results are established in
Section~\ref{sec:CoordinateChange}, where pseudodifferential operators
with non\-smooth symbols in $(x,y)$-form, oscillatory integrals for
them, and their reduction to $x$-form operators, are discussed. These
results are applied to treat nonsmooth coordinate changes for
nonsmooth pseudodifferential operators, and Theorem~\ref{thm:main2} is
proved. In Section~\ref{sec:FurtherPrelim}, the description of
$\mu$-transmission spaces is generalized from the smooth case to nonsmooth domains. In
Section~\ref{sec:TransmissionCond}, we introduce general versions of a
$\mu$-transmission condition for nonsmooth pseudodifferential
operators, and show their mapping properties in relation to the
$\mu$-transmission spaces, 
as a key to the regularity results.  Finally, in
Section~\ref{sec:Regularity} the results of the preceding sections are
applied to show regularity results for the homogeneous Dirichlet
problem for even, strongly elliptic pseudodifferential operators on
nonsmooth domains, proving Theorem \ref{thm:main}.

\section{Preliminaries}\label{subsec:prelim}

\subsection{Function spaces}

Recall that the standard Sobolev spaces $W_q^s({\mathbb R}^n)$, $1<q<\infty $ and
$s\ge 0$, have a different character according to whether $s$ is
integer or not. Namely, for $s$ integer, they consist of
$L_q$-functions with derivatives in $L_q$ up to order $s$, hence
coincide with the Bessel-potential spaces $H^s_q({\mathbb R}^n)$, defined
for $s\in{\mathbb R}$ by 
\begin{equation}
  \label{eq:2.1}
H_q^s(\R^n)=\{u\in \SD'({\mathbb R}^n)\mid \F^{-1}(\ang{\xi }^s\hat u)\in
L_q(\R^n)\}.
\end{equation}
Here $\mathcal F$ is the Fourier transform  $\hat
u(\xi )=\mathcal F
u(\xi )= \int_{{\mathbb R}^n}e^{-ix\cdot \xi }u(x)\, dx$, and the function $\ang\xi $ equals $(|\xi |^2+1)^{\frac12}$.
For noninteger $s$, the $W^s_q$-spaces coincide with the Besov
spaces $B^s_q(\rn)= B^s_{q,q}(\rn)$, defined e.g.\ 
as follows: For $0<s<2$ and measurable $f\colon \R^n\to \C$,
\begin{equation}
  \label{eq:2.2}
f\in
B^s_q({\mathbb R}^{n})\iff \|f\|_{L_q}^q+ \int_{\R^{2n}}\frac{|f(x)+f(y)-2f((x+y)/2)|^q}{|x-y|^{n+qs}}\,dxdy<\infty ;
\end{equation}
and $B^{s+t}_q(\R^n)=(1-\Delta )^{-t/2}B^s_q({\mathbb R}^n)$ for all $t\in{\mathbb R}$.
The Bessel-potential spaces are important because they are most
directly related to $L_q(\R^n)$; the Besov spaces have other
convenient properties, and are
needed for boundary value problems in an $H^s_q$-context,
 because they are the correct range spaces for
trace maps $\gamma _ju=(\partial_n^ju)|_{x_n=0}$:
\begin{equation}
  \label{eq:2.3}
\gamma _j\colon \ol H^s_q(\rnp), \ol B^s_q(\rnp) \to
B_q^{s-j-\frac1q}({\mathbb R}^{n-1}), \text{ for }s-j-\tfrac1q >0,
\end{equation}
surjectively and with a continuous right inverse; see e.g.\ the overview in
the introduction to \cite{G90}. For $q=2$, the two scales are
identical, but for $q\ne 2$ they are related by strict inclusions: $
H^s_q\subset B^s_q\text{ when }q>2$, $H^s_q\supset B^s_q\text{ when
}q<2$. When $q=2$, the index $q$ is usually omitted. We will always use $B^s_q$ as abbreviation of $B^s_{q,q}$.

We shall also use the spaces $C^k(\rn)\equiv C^k_b(\rn)$ of $k$-times differentiable
functions with uniform norms $\|u\|_{C^k}=\sup_{|\alpha |\le k,x\in\rn}|D^\alpha
u(x)|$ ($k\in{\mathbb N}_0$), and the H\"older
spaces $C^\tau   (\rn)$, $\tau  =k+\sigma $ with $k\in{\mathbb N}_0$,
$0<\sigma <1$, also denoted $C^{k,\sigma } (\rn)$, with norms
$\|u\|_{C^\tau  }=\|u\|_{C^k}+\sup_{|\alpha |= k,x\ne y}|D^\alpha
u(x)-D^\alpha u(y)|/|x-y|^\sigma $. The latter definition extends to
Lipschitz spaces $C^{k,1} (\rn)$. There are similar spaces over subsets
of $\rn$. Finally, we denote $C^\infty_b(\Rn)= \bigcap_{k\in\N} C^k_b(\Rn)$.

The halfspaces $\rnpm$ are defined by
 $\rnpm=\{x\in
{\mathbb R}^n\mid x_n\gtrless 0\}$, with points denoted  $x=(x',x_n)$,
$x'=(x_1,\dots, x_{n-1})$. When $\gamma\in C^{1+\tau}
(\R^{n-1})$ for some $\tau >0$, we define the curved halfspace
$\rn_\gamma $ by $\Rn_\gamma= \{x\in\R^n\mid x_n>\gamma(x')\}$.

Also bounded $C^{1+\tau }$-domains $\Omega $ will be considered. By this we mean
that $\Omega \subset\rn$ is open and bounded, and every boundary point
$x_0$ has an open neighborhood $U$ such that, after a translation of
$x_0$ to $0$ and a suitable rotation, $U\cap \Omega =U\cap \rn_\gamma
$ for a function $\gamma \in C^{1+\tau }(\R^{n-1})$ with $\gamma
(0)=0$. (In some texts, a hypothesis on connectedness of $\Omega $ is
included here, but we do not need this.) 

Restriction from $\R^n$ to $\rnpm$ (or from
${\mathbb R}^n$ to $\Omega $ resp.\ $\complement\comega= \Rn \setminus \comega$) is denoted $r^\pm$,
 extension by zero from $\rnpm$ to $\R^n$ (or from $\Omega $ resp.\
 $\complement\comega$ to ${\mathbb R}^n$) is denoted $e^\pm$. (The
 notation is also used for $\Omega =\rn_\gamma $). Restriction
 from $\crnp$ or $\comega$ to $\partial\rnp$ resp.\ $\partial\Omega $
 is denoted $\gamma _0$.

By $d(x)$ we denote (as in \cite[Definition~2.1]{G15} for the $C^\infty
$-case) a function  that is $C^{1+\tau } $ on $\comega$,
positive on $\Omega $ and vanishes only to the first order on
$\partial\Omega $ (i.e., $d(x)=0$ and $\nabla d(x)\neq 0$ for $x\in
\partial\Omega$).  On bounded sets it satisfies near $\partial\Omega $:
\begin{equation}\label{eq:2.3a}
  C^{-1}d_0(x)\le d(x)\le Cd_0(x)
\end{equation}
with $C>0$, where $d_0(x)$ equals
$\dist(x,\partial\Omega )$ on a  neighborhood of
$\partial\Omega $ and is extended as a correspondingly smooth positive function on
$\Omega $.

When $\tau \ge 1$,
$d_0$ itself can be taken $C^{1+\tau}$. This holds since there is then a tubular
neighborhood of $\partial\Omega $ where $d_0(x)$ plays the role of a
normal coordinate; its gradient equals the interior unit normal vector
$\nu (x)$ (cf.\ e.g.\ Pr\"uss and Simonett \cite[p.\
65-66]{pruessbuch}). Since $\nu $ is $C^\tau $, $d_0$ is $C^{1+\tau }$. 
 Then moreover, 
 $d/d_0$ is a positive $C^{\tau }$-function on $\comega$.
   
We take $d_0(x)=x_n$ in the
case of $\rnp$. For $\rn_\gamma $, the function  $d(x)=x_n-\gamma (x')$ satisfies \eqref{eq:2.3a} near
$\partial\rn_\gamma $, and does so globally on $\rn_\gamma $ if we choose
the extension of $d_0(x)$ further away from $\partial\rn_\gamma $ to equal
$x_n+C_1$, where $C_1>\sup_{x'}|\gamma (x')|$. Then when  $\tau \ge
1$,  $d/d_0$ is a positive function in $C^{\tau }(\ol\R^n_\gamma $).

Along with the spaces $H^s_q({\mathbb R}^n)$ defined in \eqref{eq:2.1}, there are
the two scales of spaces associated with $\Omega $ for $s\in{\mathbb R}$:
\begin{equation}
  \label{eq:2.4}
  \begin{split}
\ol H_q^{s}(\Omega)&=\{u\in \D'(\Omega )\mid u=r^+U \text{ for some }U\in
H_q^{s}(\R^n)\}, \text{ the {\it restricted} space},\\
\dot H_q^{s}(\comega)&=\{u\in H_q^{s}({\mathbb R}^n)\mid \supp u\subset
\comega \},\text{ the {\it supported} space;}
  \end{split}
\end{equation}
here $\operatorname{supp}u$ denotes the support of $u$. 
 $\ol H_q^s(\Omega )$ is in other texts often denoted  $H_q^s(\Omega )$  or
$H_q^s(\comega )$, and $\dot H_q^{s}(\comega)$ may be indicated with a
ring, zero or twiddle;
the current notation stems from H\"ormander \cite[Appendix B2]{H85}.
There are similar spaces with $B^s_q$.

Besides for the $H^s_q$ and $B^s_q$-spaces, there are in \cite{G14} for
$C^\infty $-domains 
established the relevant results in many other scales of spaces,
namely Besov spaces $B^s_{p,q}$ for $1\le p,q \le \infty $
and Triebel-Lizorkin spaces $F^s_{p,q}$ (for the same $p,q$ but with $p<\infty $). We shall
not pursue this in the present work, except that we want to refer to
the H\"older-Zygmund scale $B^s_{\infty ,\infty }$, also denoted
$C^s_*$. Here $C^s_*$ identifies with the
H\"older space $C^s$ when $s\in
\rp\setminus {\mathbb N}$, and for positive integer $k$ satisfies
$ C^{k-\varepsilon }\supset C^k_*\supset 
C^{k-1,1}\supset  C^k_b$ for small $\varepsilon >0$; moreover,
$C^0_*\supset L_\infty \supset C^0_b$ (with strict inclusions everywhere). Similarly to \eqref{eq:2.4} we denote the
spaces of restricted, resp.\ supported distributions
\begin{align*}
\ol C_*^{s}(\Omega)&=\{u\in \D'(\Omega )\mid u=r^+U \text{ for some }U\in
C_*^{s}(\R^n)\},\\
\dot C_*^{s}(\comega)&=\{u\in C_*^{s}({\mathbb R}^n)\mid \supp u\subset
\comega \};
\end{align*}
 the star can be omitted when $s\in \rp\setminus \N$.

\subsection{Pseudodifferential operators}\label{subsec:PsDOs}

A {\it pseudodifferential operator} ($\psi $do) $P$ on ${\mathbb R}^n$ is
defined from a symbol $p(x,\xi )$ on ${\mathbb R}^n\times{\mathbb R}^n$ by 
\begin{equation}
  \label{eq:2.5}
Pu=\op (p(x,\xi ))u 
=\int_{\Rn} e^{ix\cdot\xi
}p(x,\xi )\hat u(\xi)\, \dd\xi =\mathcal F^{-1}_{\xi \to x}(p(x,\xi )\F u(\xi
)),
\end{equation}
using the Fourier transform $\F$ and the notation $\dd\xi= (2\pi)^{-n}\, d\xi$. 
We refer to
textbooks such as Kumano-go~\cite{KumanoGo}, H\"ormander \cite{H85}, Taylor \cite{T81},  Grubb
\cite{G09}, Abels \cite{PsDOBuch} for the rules of
calculus, in particular the definition by oscillatory integrals in
\cite{H85}, \cite{PsDOBuch}. For precision, the notation $\Os\int$ is often
used when an integral is
understood as an oscillatory integral.

The symbols
$p$ of order $m\in{\mathbb R}$ were originally taken to lie in the symbol
space $S^m_{1,0}({\mathbb R}^n\times{\mathbb R}^n)$, consisting of complex
$C^\infty $-functions $p(x,\xi )$
such that $\partial_x^\beta \partial_\xi ^\alpha p(x,\xi
)$ is $O(\ang\xi ^{m-|\alpha |})$ for all $\alpha ,\beta $, for some
$m\in{\mathbb R}$, with global estimates for $x\in{\mathbb R}^n$ (as in
\cite[start of Section~18.1]{H85} and \cite{G96}).
$P$ (of order $m$)
then  maps $H^s_q({\mathbb R}^n)$ continuously into
$H^{s-m}_q ({\mathbb R}^n)$ for all $s\in{\mathbb R}$, cf.\ \eqref{eq:2.1}.
$P$ is said to be {\it classical} when 
$p$  
has an asymptotic expansion $p(x,\xi )\sim \sum_{j\in{\mathbb
N}_0}p_j(x,\xi )$ with $p_j$ homogeneous in $\xi $ of degree $m-j$ for
all $|\xi |\ge 1$ and $j\in\N_0$, such that
\begin{equation}
  \label{eq:2.6}
\partial_x^\beta \partial_\xi ^\alpha \bigl(p(x,\xi )-
{\sum}_{j<J}p_j(x,\xi )\bigr) \text{ is }O(\ang\xi ^{m-\alpha -J})\text{ for
all }\alpha ,\beta \in{\mathbb N}_0^n, J\in{\mathbb N}_0.  
\end{equation}
 For a 
complete theory one adds to these operators 
the smoothing operators (mapping any  $H^s_q({\mathbb R}^n)$ into
$\bigcap_tH^t_q ({\mathbb R}^n)$), regarded as operators of order
$-\infty $. (For example, $(-\Delta )^a$ fits into the calculus when
it is 
written as $\op((1-\zeta (\xi ))|\xi |^{2a})+\op(\zeta (\xi
)|\xi |^{2a})$, where $\zeta (\xi )$ is a $C^\infty $-function that
equals $1$ for $|\xi |\le \frac12$ and 0 for $|\xi |\ge 1$; the 
second term is smoothing.)

In the present study we moreover consider symbols with limited
smoothness in $x$. For later purposes we here replace $x\in \rn$ by
$X\in {\mathbb R}^{n'}$, where $n'$ will usually be $n$ or $2n$.

The space $C^\tau S^m_{1,0}(\R^{n'}\times \Rn)$ for $\tau>0$, $m\in\R$, $n',n\in\N$ consists of functions
$p \colon  \R^{n'}\times \R^n\to \C$ that are continuous w.r.t.\
$(X,\xi)\in  \R^{n'} \times \R^n$ and smooth with respect to $\xi\in
\R^n$, such that for every $\alpha\in\N_0^n$  we have:
$\partial_\xi^\alpha p(X,\xi)$ is in $C^\tau (\R^{n'})$ with respect to
$X$ and satisfies for all $\xi\in\R^n$, $\alpha \in {\mathbb N}_0^n$, 
\begin{equation}
  \label{eq:2.6'}
  \|\partial_\xi^\alpha p(\cdot,\xi)\|_{C^\tau  (\R^{n'})}\leq C_\alpha
  \ang{\xi}^{m-|\alpha|} ,
\end{equation}
with $C_\alpha>0$. We equip the symbol space with the semi-norms
\begin{equation}\label{eq:seminorm}
  |p|_{k,C^\tau S^m_{1,0}(\R^{n'}\times \R^n)}:= \max_{|\alpha|\leq k} \sup_{\xi\in\R^n} \ang{\xi}^{-m+|\alpha|}\|\partial_\xi^\alpha p(\cdot,\xi)\|_{C^\tau  (\R^{{n'}})}\quad \text{for }k\in\N_0.
\end{equation}

The following theorem is well-known:

\begin{theorem}\label{thm:bd-compos}
  Let $\tau>0$, $1<q<\infty$, $m\in\R$ and $p\in C^\tau S^m_{1,0}(\R^n\times \R^n)$. Then $\op(p)$ extends to a bounded linear operator
  \begin{equation*}
    \op(p)\colon H^{s+m}_q (\R^n)\to H^s_q(\R^n)\quad \text{for all }|s|<\tau.
  \end{equation*}
  Moreover, for every $s\in (-\tau,\tau)$ there is some $k\in\N$ and $C>0$ such that
  \begin{equation*}
    \|\op(p)\|_{\mathcal{L}(H^{s+m}_q (\R^n), H^s_q(\R^n))}\leq C|p|_{k,C^\tau S^m_{1,0}(\R^n\times \R^n)}\quad \text{for all }p\in  S^m_{1,0}(\R^n\times \R^n).
  \end{equation*}
\end{theorem}
The mapping properties follow from \cite[Theorem~2.7]{Marschall87}.
 The boundedness of the operator norm by a symbol semi-norm is a consequence of the closed graph theorem or the Hahn-Banach theorem, cf.\ e.g.\ \cite[Theorem~3.7]{AbelsPfeufferSpectralInvariance}.

The subspace of classical symbols $C^\tau
S^m(\R^{n'}\times \Rn)$ consists of those functions that moreover have
expansions into terms  $p_j$ homogeneous in $\xi $ of degree $m-j$ for
$|\xi |\ge 1$, all $j$, such that for all $\xi\in\R^n$, $\alpha
\in{\mathbb N}_0^n$, $ J\in{\mathbb N}_0$,
\begin{equation}
  \label{eq:2.7}
\|\partial_\xi ^\alpha \bigl(p(\cdot,\xi )-
{\sum}_{j<J}p_j(\cdot,\xi )\bigr)\|_{C^\tau  (\R^{n'})}\leq C_{\alpha
,J}\ang{\xi}^{m-J-|\alpha|}.
\end{equation}

A classical symbol $p(x,\xi )$ (and the associated operator $P$) is said to be   {\it strongly elliptic}
when $\operatorname{Re}p_0(x,\xi )\ge c|\xi |^m $ for $|\xi |\ge 1$,
with $c>0$.
Moreover, a classical $\psi $do $P=\op(p(x,\xi ))$ of
order $m\in\R$ is said
to be 
{\it even}, when the terms in the symbol expansion $p\sim\sum_{j\in\N_0}p_j$ satisfy
\begin{equation}\label{eq:even}
p_j(x,-\xi )=(-1)^jp_j(x,\xi )\quad \text{ for all }x\in \rn,|\xi |\ge 1, j\in \mathbb N_0.
\end{equation}
(The word ``even'' is short for {\it even-to-even
  parity}, meaning that  the
terms with even $j$ are even in $\xi $,  the
terms with odd $j$ are odd in $\xi $.)
Similarly, $p$ is \emph{odd} (short for odd-to-even parity) if
$p\sim\sum_{j\in\N_0}p_j$, where
\begin{equation*}
p_j(x,-\xi )=(-1)^{j+1}p_j(x,\xi )\quad \text{ for all }x\in \rn,|\xi |\ge 1, j\in \mathbb N_0.
\end{equation*}
Note that when $p$ is even of order $m$, $p-p_0$ is odd of order $m-1$.

\medskip
In part of the present paper, we consider symbols that are moreover
assumed  to satisfy a  $\mu $-transmission
condition, as introduced in the smooth case by H\"ormander
\cite{H66,H85} and Grubb \cite{G15}, see Section \ref{subsec:muTransmissionSpaces}. 
To handle operators with such properties, we must introduce
 {\it order-reducing
operators}. There is a simple definition of operators $\Xi _\pm^t $ on
${\mathbb R}^n$, $t\in{\mathbb R}$,
\begin{equation}
  \label{eq:2.8}
\Xi _\pm^t =\operatorname{OP}(\chi _\pm^t),\quad \chi _\pm^t(\xi )=(\ang{\xi '}\pm i\xi _n)^t ;  
\end{equation}
 they preserve support
in $\crnpm$, respectively, because the symbols extend as holomorphic
functions of $\xi _n$ into ${\mathbb C}_\mp$, respectively; ${\mathbb
C}_\pm=\{z\in{\mathbb C}: \operatorname{Im}z\gtrless 0\}$. (The functions
$(\ang{\xi '}\pm i\xi _n)^t $  satisfy only part of the estimates
\eqref{eq:2.6} with $m=t$, but the $\psi $do definition can be applied anyway.)
 There is a more refined choice $\Lambda _\pm^t $, cf.\
\cite{G90, G15}, with
symbols $\lambda _\pm^t (\xi )$ that do
satisfy all the required estimates, and where $\overline{\lambda _+^t }=\lambda _-^{t }$.
These symbols likewise have holomorphic extensions in $\xi _n$ to the complex
halfspaces ${\mathbb C}_{\mp}$, so that  
the operators preserve
support in $\crnpm$, respectively. Operators with that property are
called ``plus'' resp.\ ``minus'' operators.

There is also a pseudodifferential definition $\Lambda
_\pm^{(t )}$ adapted to the situation of a smooth domain $\Omega
$, by \cite{G90,G15}.

It follows from the Lizorkin multiplier theorem and the definition of the spaces $H_q^s(\R^n)$
in terms of Fourier transformation that the operators define homeomorphisms 
$
\Xi^t _\pm\colon H_q^s(\R^n) \simto H_q^{s- t
}(\R^n)$, for all $s\in \R$. 
The special
interest is that the ``plus''/``minus'' operators also 
 define
homeomorphisms related to $\crnp$ and $\comega$, for all $s\in{\mathbb R}$: 
$$
\aligned
\Xi ^{t }_+\colon \dot H_q^s(\crnp )&\simto
\dot H_q^{s- t }(\crnp),\quad
r^+\Xi ^{t }_{-}e^+\colon \ol H_q^s(\rnp )\simto
\ol H_q^{s- t } (\rnp ),\\
\Lambda  ^{(t) }_+\colon \dot H_q^s(\comega )&\simto
\dot H_q^{s- t }(\comega),\quad
r^+\Lambda  ^{(t) }_{-}e^+\colon \ol H_q^s(\Omega  )\simto
\ol H_q^{s- t } (\Omega  ),
\endaligned
$$
with similar rules for $\Lambda ^t_\pm$.
Moreover, the operators $\Xi ^t _{+}$ and $r^+\Xi ^{t }_{-}e^+$ identify with each other's adjoints
over $\crnp$, because of the support preserving properties.
There is a
similar statement for $\Lambda ^t_+$ and  $r^+\Lambda ^t_-e^+$, and for  $\Lambda ^{(t )}_+$ and $r^+\Lambda ^{(
t )}_{-}e^+$ relative to the set $\Omega $.

\section{Nonsmooth Coordinate Changes}
\label{sec:CoordinateChange}

Since some of the proofs of the results in this chapter are quite technical,
we have moved them to the Appendix, in order to keep the flow of the
presentation leading to results on fractional-order boundary problems. 

Basic rules
of calculus for pseudodifferential operators with nonsmooth symbols
were set up by Kumano-go and Nagase \cite{KN78}, Beals and Reed \cite{BealsReed},  Marschall
\cite{Marschall87,Marschall88}, Witt \cite{Witt98},
however leaving out the question of nonsmooth coordinate changes. To our
knowledge, this question has been open since then, and it is this that
we want to work out now. A basic step is the reduction of operators
defined by symbols in $(x,y)$-form to symbols in $x$-form; in
particular to handle the remainders arising from this. The definition
of operators from symbols $a(x,y,\xi )$ (also called amplitude
functions) follows  from known
facts when $a$ is $C^\tau $ in $x$ and $C^\infty $ in $y$, see e.g.\  [AP18],
but a definition when $a$ is merely $C^\tau $ in $y$ has not to our
knowledge been
established in detail; this is the first thing we undertake here.  

Symbol classes and their seminorms were defined in the preliminaries section. 
In the following let $\tau>0$, $m\in\R$, $a\in C^\tau
S^m_{1,0}(\R^{2n}\times \Rn)$.

The first task is to
show the existence of oscillatory integrals for nonsmooth 
$(x,y)$-form symbols, and to find derived $x$-form operators.

Assuming $\tau >m$, we have to give a meaning to, 
\begin{alignat*}{1}
  \op(a)u(x)&=  \Os\int_{\R^{2n}} e^{i(x-y)\cdot \xi} a(x,y,\xi) u(y)\sd y\dd\xi\\
  &:= \lim_{\eps\to 0}\int_{\R^{2n}}\chi(\eps y,\eps \xi) e^{i(x-y)\cdot \xi} a(x,y,\xi) u(y)\sd y\dd\xi
\end{alignat*}
for every $u\in \SD(\Rn)$, $x\in\R^n$, where $\chi\in \SD(\R^{2n})$
with $\chi(0,0)=1$.
To this end (and for other purposes later on) we shall use a Taylor expansion
\begin{equation}\label{eq:Taylor}
    a(x,y,\xi)= \sum_{|\alpha|\leq l} \frac1{\alpha!}\partial_y^\alpha a(x,y,\xi)|_{y=x} (y-x)^\alpha+\sum_{|\alpha|=l}(y-x)^\alpha r_\alpha(x,y,\xi),
  \end{equation}
where $l\in\N_0$ with $l<\tau$ and 
  \begin{equation}\label{eq:Taylor1}
    r_\alpha(x,y,\xi) = \frac{|\alpha|}{\alpha!}\int_0^1(1-t)^{|\alpha|-1} (\partial_y^\alpha a)(x,(1-t)x+ty,\xi)\sd t -\frac1{\alpha!} \partial_y^\alpha a(x,y,\xi)|_{y=x}.
  \end{equation}
  We have here subtracted the $\alpha $'th precise term from the usual
  remainder term, to achieve that
  \begin{equation*}
    r_\alpha(x,x,\xi) = 0\quad \text{for all }x,\xi\in\R^n. 
  \end{equation*} 
Note that  $r_\alpha\in C^{\tau-|\alpha|}S^m_{1,0}(\R^{2n}\times \Rn)$. Hence for every $\beta\in \N_0^n$
  \begin{equation}\label{eq:ralphaEstim}
    |\partial_\xi^\beta r_\alpha(x,y,\xi)|\leq C_\beta|x-y|^{\min\{\tau-|\alpha|,1\}}\weight{\xi}^{m-|\beta|}\quad \text{for all }x,y,\xi\in\Rn.
  \end{equation}

  We will use a dyadic partition of unity $(\varphi_j)_{j\in\N_0}$ of $\R^n$ in the proof. More precisely, let $\varphi_j\in C_0^\infty(\Rn)$, $j\in\N_0$, be a partition of unity such that
  \begin{equation}\label{eq:dyadic}
  \operatorname{supp} \varphi_j\subset \{\xi\in\Rn\mid 2^{j-1}\leq |\xi|\leq 2^{j+1}\}  
  \end{equation}
  and $\varphi_j(\xi)= \varphi_1(2^{1-j}\xi)$ for all $j\geq 1$, $\xi\in \Rn$.
  
  For later purposes we show existence of the oscillatory integral in
  a more general form. We denote by $[\sigma ]$ the largest integer
  $\le \sigma $.
  
  \begin{thm}\label{thm:OscIntegrals}
    Let $a\in C^\tau S^m_{1,0}(\R^{2n}\times \Rn)$, $\tau >0$, $m\in\R$, $\gamma\in\N_0^n$, and assume that $m<\tau + |\gamma|$. Then 
    for every $x\in \Rn$ and $u\in\SD (\Rn)$ the limit
    \begin{equation*}
    \op((y-x)^\gamma a(x,y,\xi))u(x)
    := \lim_{\eps\to 0} \int_{\R^{2n}}\chi(\eps y,\eps \xi) e^{i(x-y)\cdot \xi} (y-x)^\gamma a(x,y,\xi) u(y)\sd y\dd\xi
  \end{equation*}
  exists and coincides with
  \begin{alignat*}{1}
    \sum_{|\alpha|\leq l}  \frac1{\alpha!}\op(\partial_y^\alpha D_\xi^{\alpha+\gamma} a(x,y,\xi)|_{y=x})u(x)
    + \sum_{|\alpha|=l}\op((y-x)^{\alpha+\gamma} r_\alpha(x,y,\xi))u(x),
  \end{alignat*}
  where $l=\max\{[m-|\gamma|],0\} <\tau$, and
  \begin{equation*}
    \op((y-x)^\alpha r_\alpha(x,y,\xi))u(x) = \int_{\Rn} k_{\alpha,\gamma}(x,y,x-y)u(y)\, dy,
  \end{equation*}
  $|k_{\alpha,\gamma}(x,y,x-y)|\leq g(x-y)$ for some nonnegative $g
  \in L_1(\Rn)$, and 
  \begin{alignat*}{1}
    k_{\alpha,\gamma}(x,y,z)&= {\sum}_{j\in\N_0} k_{\alpha,\gamma,j}(x,y,z),\quad\\
    k_{\alpha,\gamma,j}(x,y,z)&= \int_{\R^n} e^{iz\cdot \xi} (x-y)^{\alpha+\gamma}r_\alpha (x,y,\xi)\varphi_j(\xi)\dd\xi 
  \end{alignat*}
  for all $x,y,z\in\Rn$ with $z\neq 0$.
\end{thm}

The proof of Theorem \ref{thm:OscIntegrals} is given in the Appendix.

It will be convenient to observe:

\begin{lem}\label{lem:PartIntOscInt}
  Let $a\in C^\tau S^m_{1,0}(\R^{2n}\times \Rn)$, $\tau >0$, $m\in\R$, $\gamma\in\N_0^n$, and assume that $m<\tau + |\gamma|$. Then 
    for every $x\in \Rn$ and $u\in\SD (\Rn)$ and $\beta \in\N_0^n$ with $\beta\leq \gamma$: 
  \begin{alignat*}{1}
    &\op((y-x)^\gamma a(x,y,\xi))u(x) = \op((y-x)^{\gamma-\beta} D_\xi^\beta a(x,y,\xi))u(x).
  \end{alignat*}
\end{lem}

The proof of Lemma \ref{lem:PartIntOscInt} is given in the Appendix.

    We also have:
    
\begin{cor}\label{cor:xyFormReduction}
  Let $a\in C^\tau S^m_{1,0}(\R^{2n}\times \Rn)$, $\tau >0$, $m\in\R$,
  and assume that $m<\tau$. Setting
  \begin{equation*}
p_\alpha(x,\xi)=\frac1{\alpha!}\partial_y^\alpha D_\xi^{\alpha} a(x,y,\xi)|_{y=x},
    \end{equation*}
we have for every $l\in\N_0$ with $l<\tau$:
  \begin{alignat*}{1}
    \op(a(x,y,\xi))u(x) = \sum_{|\alpha|\leq l}  \op (p_\alpha (x,\xi)) u(x) + \sum_{|\alpha|=l}\op(D_\xi^\alpha r_\alpha(x,y,\xi))u(x),
  \end{alignat*}
  where $p_\alpha(x,\xi) \in C^{\tau-|\alpha|}S^{m-|\alpha|}_{1,0}(\Rn\times \Rn)$ and
  $D_\xi^\alpha r_\alpha\in C^{\tau-l}S^{m-l}_{1,0}(\R^{2n}\times
  \Rn)$ for all $|\alpha|=l$ (as  defined in 
  \eqref{eq:Taylor1}), with $D_\xi^\alpha r_\alpha(x,x,\xi )=0$.
\end{cor}
\begin{proof}
  The equality follows directly from Theorem~\ref{thm:OscIntegrals}
  with $\gamma=0$ and Lemma~\ref{lem:PartIntOscInt} applied to
  $a=r_\alpha$ and $\gamma=\beta=\alpha$. The statements on $D_\xi
  ^\alpha r_\alpha $ follow from  \eqref{eq:Taylor1}ff.
\end{proof}

The next task is to determine the mapping properties of $(x,y)$-form operators.

The following result will be the basis for all further results on
mapping properties of $(x,y)$-form operators. It applies not only to
remainders, but also to full operators, without a graded expansion that
would reduce the regularity with respect to $x$.

\begin{thm}\label{thm:Bddness}
  Let $a\in C^\tau S^m_{1,0}(\R^{2n}\times \R^n)$ with $|m| < \tau$,
  and let $1<q<\infty$. Then
  \begin{equation*}
    \op(a(x,y,\xi))\colon H^{s+m}_q(\R^n)\to H^s_q(\R^n)
  \end{equation*}
  is a bounded linear operator, provided that
   $ |s| <\tau$, $|s+m| < \tau$.
 \end{thm}

 The proof of Theorem \ref{thm:Bddness} is given in the Appendix.

We observe a particular conclusion for the remainder terms
\eqref{eq:Taylor1}, in case $s\ge 0$. 

\begin{cor}\label{cor:Bddnessr} Let $r_\alpha$  be as in
  \eqref{eq:Taylor1} and Corollary~{\rm \ref{cor:xyFormReduction}}, let $m<\tau$,
  $l\in\N_0$ with $l<\tau$,
   $1<q<\infty$, and let $s\in\R$ such that $0\leq s<\tau-l$, $ {s+m<\tau}$. Then
 \begin{equation}\label{eq:Bddnessb}
   \op(D_\xi^\alpha r_\alpha(x,y,\xi ))\colon H^{{(s+m-l)_+}}_q(\Rn)\to H^s_q(\Rn)
 \end{equation}
 is bounded. 
  Moreover, there is some $k\in\N$ and $C_{s,q}>0$ such that
 \begin{equation*}
   \|\op(D_\xi^\alpha r_\alpha(x,y,\xi ))\|_{\mathcal{L}(H^{(s+m-l)_+}_q(\Rn), H^s_q(\Rn))}\leq C_{s,q}|a|_{k, C^{\tau} S^{m}_{1,0}(\R^{2n}\times \R^n)}.
 \end{equation*}
\end{cor}
\begin{proof} If $m-l> -(\tau -l)$, we
can apply Theorem~\ref{thm:Bddness} directly with $\tau $ and $m$
replaced by $\tau '=\tau -l$ and $m'=m-l$; this gives
\begin{equation*}
   \op(D_\xi^\alpha r_\alpha(x,y,\xi ))\colon H^{{s+m-l}}_q(\Rn)\to H^s_q(\Rn),
 \end{equation*}
and \eqref{eq:Bddnessb} holds \`a{} fortiori. If $m-l
  \leq -(\tau-l)$, we note that $-(\tau -l)<-s$, so that $D_\xi^\alpha r_\alpha \in
  C^{\tau-l}S^{m-l}_{1,0}(\R^{2n}\times \Rn)\subset
  C^{\tau-l}S^{-s}_{1,0}(\R^{2n}\times \Rn)$; here
  Theorem~\ref{thm:Bddness} can be applied with $\tau $ and $m$
  replaced by $\tau '=\tau -l$ and $m''=-s$, giving
  \begin{equation*}
   \op(D_\xi^\alpha r_\alpha(x,y,\xi ))\colon H^{0}_q(\Rn)\to H^s_q(\Rn).
 \end{equation*}
 This is as desired since in this case $(s+m-l)_+=0$.

 Finally,  the boundedness of the operator norm by a symbol semi-norm is a consequence of the closed graph theorem. It can be shown in the same way as in the proof of Theorem~\ref{thm:BddnessTruncatedPsDOs} below.
\end{proof}

The next result will help improve remainder estimates; it is related
to statements given in Taylor 
\cite[Proposition~9.5]{ToolsForPDE}. 

\begin{thm}
  \label{thm:Bddness2}
  Let $a\in C^\tau S^m_{1,0}(\R^{2n}\times \R^n)$ with $m < \tau$ and
  \begin{equation*}
  \partial_x^\alpha \partial_y^\beta a(x,y,\xi)|_{y=x}=0\quad \text{for all }x,\xi\in \Rn\text{ and }|\alpha|+|\beta|< \tau.  
  \end{equation*}
   Moreover, let $1<q<\infty$. Then
  \begin{equation*}
    \op(a(x,y,\xi))\colon L_q(\R^n)\to H^s_q(\R^n)
  \end{equation*}
  is a bounded linear operator provided that $0\leq s < \min(\tau-m,\tau)$.
\end{thm}

The proof of Theorem \ref{thm:Bddness2} is given in the Appendix.

\begin{remark}\label{rem:cutoff}
  We note that, if $a\in C^\tau S^m_{1,0}(\R^{2n}\times \Rn)$ with $0\leq m <\tau$ satisfies $a(x,y,\xi)=0$ for all $x,y,\xi\in\Rn$ with $|x-y|<\delta$ for some $\delta>0$, then the conditions of Theorem~\ref{thm:Bddness2} are satisfied. In particular, if $p\in C^\tau S^m_{1,0}(\Rn\times \Rn)$ with $0\leq m<\tau$ and $\varphi,\psi\in C^\infty_b(\Rn)$ are such that $\supp \varphi\cap \supp \psi =\emptyset$, then
  \begin{equation*}
    \psi \op(p(x,\xi))(\varphi u) \in H^s_q(\Rn)\qquad \text{for all }u\in L_q(\Rn), 0\leq s<\tau.
  \end{equation*}
\end{remark}

The third task is to establish rules for coordinate changes, by use of the
results on $(x,y)$-form operatos.


In the following, let  $F\colon \R^n\to \R^n$ be a $C^1$-diffeomorphism with $DF \in C^{\tau}(\R^n)^{n\times n}$, and $p\in C^\tau S^m_{1,0}(\Rn\times \Rn)$ for some $m<\tau$ and $\tau >0$.  Moreover, let
\begin{equation}\label{eq:PgammaGeneral}
  (\underline P u)(x) = (P (u\circ F^{-1}))(F(x))= (F^{\ast}P F^{\ast,-1} u)(x)\quad \text{for all }u\in \SD(\Rn),
\end{equation}
where $(F^*v)(y):=v(F(y))$ and $(F^{*,-1}u)(x):=u(F^{-1}(x))$.
We will first consider the case that 
\begin{equation}\label{eq:CondF}
  \sup_{x\in\R^n}|\nabla F(x)-I|\leq \tfrac12.
\end{equation}
Then one obtains for all $u\in \SD(\R^n)$ and $x\in\R^n$:
\begin{alignat}{1}\nonumber
  \underline P u(x) &=\Os\int_{\R^{2n}} e^{i(F(x)-z)\cdot \eta } p(F(x),\eta) u(F^{-1} (z)) \sd z\dd \eta\\\label{eq:CoordTrafo1General}
  &=\Os\int_{\R^{2n}} e^{i(x-y)\cdot \xi } q(x,y,\xi) u(y) \sd y\dd \xi,
\end{alignat}
where a well-known formula for coordinate changes (explained
e.g.\ in  \cite[Proof of Theorem~3.48]{PsDOBuch}) gives 
\begin{alignat}{1} \label{eq:CoordTrafo2General}
  q(x,y,\xi)&= p(F(x),A(x,y)^{-1,T}\xi)|\det A(x,y)|^{-1}|\det \nabla_y F(y)|,\\\nonumber
  A(x,y) &= \int_0^1 \nabla_x F (x+t(y-x))\sd t,
\end{alignat}
for all $x,y,\xi\in\R^n$.  
Here $q(x,y,\xi )\in  C^\tau S^m_{1,0}(\R^{2n}\times \Rn)$ and $A(x,y)^{-1,T}=(A(x,y)^{-1})^T$. We note that \eqref{eq:CondF} ensures $\det A(x,y)\neq 0$ for every $x,y\in\Rn$.
In our nonsmooth case, the existence of the limit in the oscillatory integral in
\eqref{eq:CoordTrafo1General} follows from
Theorem~\ref{thm:OscIntegrals}. The existence of this limit also
implies the existence of the limit in the oscillatory integral
preceding it. 
Moreover, using that for every $\tau'\in (0,\tau]$ there is some $C>0$ such that 
\begin{equation*}
  \|a\circ F-a\|_{C^{\tau'}(\Rn)}\leq C\|a\|_{C^\tau(\Rn)}\|F-\operatorname{id}\|_{C^{1+\tau}(\Rn)}^{\min (1,\tau-\tau')}\qquad \text{for all }a\in C^\tau(\Rn)
\end{equation*}
one verifies in a straightforward manner that for every $r_0>0$ and $k\in \N_0$ there is some $C_k$ independent of $F$ such that
\begin{alignat}{1}\label{eq:qEstim}
  |q -p|_{k, C^{\tau'} S^m_{1,0}(\R^{2n}\times\Rn)} &\leq C_k \|F-\operatorname{id}\|_{C^{1+\tau}(\Rn)}^{\min (1,\tau-\tau')}|p|_{k+1, C^\tau S^m_{1,0}(\R^{n}\times\Rn)}
\end{alignat}
for all $p\in C^\tau S^m_{1,0}(\R^{n}\times\Rn)$ and $C^1$-diffeomorphisms $F$ such that $DF\in C^{\tau}(\Rn)^{n\times n}$, $\|F-\operatorname{id}\|_{C^{1+\tau}}\leq r_0$ and \eqref{eq:CondF} holds
since
\begin{alignat*}{1}
  \|A-I\|_{C^{\tau}(\Rn\times \Rn)}\leq C\|F-\operatorname{id}\|_{C^{1+\tau}(\Rn)}.
\end{alignat*}
We shall now apply the Taylor expansion in \eqref{eq:Taylor} and 
Corollary~\ref{cor:xyFormReduction} to $q$. This gives that for any $N\in\N_0$ with $N<\tau$,
\begin{alignat*}{1}
  \ul P u(x) &= \sum_{|\alpha|\leq N} \op(q_\alpha (x,\xi))u(x) + \op (\tilde{r}_N(x,y,\xi)) u(x),
\end{alignat*}
where
\begin{alignat}{1}\label{eq:CoordTrafo3General}
  q_\alpha (x,\xi) =& \frac1{\alpha!} \left.\partial_y^\alpha
      D_\xi^\alpha q(x,y,\xi )\right|_{y=x},\\ \nonumber
  \tilde{r}_N(x ,y,\xi) =& \sum_{|\alpha|=N} \frac N{\alpha!}
  \int_0^1(1-t)^{N-1} \partial_y^\alpha D_\xi^\alpha q(x,x+t(y-x),\xi)\sd t  -\sum_{|\alpha|=N} q_\alpha (x,\xi ).
\end{alignat}

\begin{theorem}\label{thm:CoordTrafo1}
    Let $p\in C^\tau S^{m}_{1,0}(\Rn\times \Rn)$ with $m<\tau$ and $\tau>0$,  let
  $F$ be a $C^1$-diffeomorphism with $DF\in C^{\tau}(\Rn)^{n\times n}$ such that \eqref{eq:CondF} holds true, $|\alpha|\leq N<\tau$ and let  $q$
  be defined as in \eqref{eq:CoordTrafo3General}. Then
  $q_\alpha\in C^{\tau-|\alpha|}S^{m-|\alpha|}_{1,0}(\R^n\times \R^n)$ and $\tilde{r}_N\in C^{\tau -N}S^{m-N}_{1,0}(\R^{2n}\times \Rn)$ with $\tilde{r}_N(x,x,\xi)=0$ for all $x,\xi\in\R^n$.
  Moreover, for every $r_0>0$, $\tau'\in (N,\tau]$, and $k\in \N_0$ there is some $C_k$ independent of $F$ such that
\begin{alignat*}{1}
  |q_0 -p|_{k, C^{\tau'} S^{m}_{1,0}(\Rn\times\Rn)} &\leq C_k \|F-\operatorname{id}\|_{C^{1+\tau}(\Rn)}^{\min (1,\tau-\tau')}|p|_{k+1, C^\tau S^m_{1,0}(\R^{n}\times\Rn)},\\
  |q_\alpha|_{k, C^{\tau'-|\alpha|} S^{m-|\alpha|}_{1,0}(\Rn\times\Rn)}&\leq C_k \|F-\operatorname{id}\|_{C^{1+\tau}(\Rn)}^{\min (1,\tau-\tau')}|p|_{k+1, C^\tau S^m_{1,0}(\R^{n}\times\Rn)},\\
   |\tilde r_N|_{k, C^{\tau'-N} S^{2a-N}_{1,0}(\R^{2n}\times\Rn)}&\leq C_k \|F-\operatorname{id}\|_{C^{1+\tau}(\Rn)}^{\min (1,\tau-\tau')}|p|_{k+1, C^\tau S^m_{1,0}(\R^{n}\times\Rn)}
\end{alignat*}
for all $1\leq |\alpha|\leq N$, provided that $\|F-\operatorname{id}\|_{C^{1+\tau}}\leq r_0$.
\end{theorem}
\begin{proof}
  The first statements follow easily from the definitions. Morever, 
as for \eqref{eq:qEstim} one can verify the stated estimates in a straightforward manner.
\end{proof}

For general $C^{1+\tau}$-diffeomorphisms we obtain:
\begin{theorem}\label{thm:CoordTrafo2}
    Let $p\in C^\tau S^{m}_{1,0}(\Rn\times \Rn)$ with $m<\tau$ and $\tau>0$,  let
    $F\colon \R^n\to \R^n$ be a $C^1$-diffeomorphism with $DF\in C^\tau(\R^n)^{n\times n}$ such that
    \begin{equation}\label{eq:NonDeg}
      c_0 \leq |\det (\nabla F(x))|\leq C_0 \qquad \text{for all } x\in \R^n
    \end{equation}
    for some $c_0,C_0>0$. Then there is some $q\in C^\tau S^{m}_{1,0}(\R^{2n}\times \Rn)$ such that
    \begin{equation}\label{eq:Pgamma'}
      \ul P u= F^{\ast}P F^{\ast,-1} u= \op(q(x,y,\xi)) u + R u\quad \text{for all }u\in \SD(\Rn),
    \end{equation}
    where
    \begin{equation*}
      R \colon L_q(\Rn)\to H^s_q(\Rn)\qquad \text{for all }s<\min(\tau-m,\tau).
    \end{equation*}
    Moreover, for any $N\in\N_0$ with $N<\tau$ 
    \begin{alignat*}{1}
      \op(q(x,y,\xi)) &= \sum_{|\alpha|\leq N} \op(q_\alpha (x,\xi)) + \op (\tilde{r}_N(x,y,\xi)),
    \end{alignat*}
    where the entries are defined by \eqref{eq:CoordTrafo3General},
with
$
  A(x,y)= \int_0^1 \nabla_x F(x+t(y-x))\, dt
$
for every $y\in\Rn$ sufficiently close to $x$. 
\end{theorem}
\begin{proof}
  Since $\nabla F\in C^{\tau}(\R^n)$ and satisfies \eqref{eq:NonDeg}, there is some $\delta>0$ such that $A(x,y)$ is invertible for every $x,y\in\R^n$ with $|x-y|<\delta$.
  Now choose $\psi\in C^\infty_0(\Rn)$ with $\psi\equiv 1$ on $B_{\delta/2}(0)$ and $\supp \psi \subset B_\delta(0)$. Then
  \begin{alignat*}{1}
    Pu= & \Os\int_{\R^{2n}} e^{i(x-y)\cdot \xi } \psi(x-y) p(x,\xi) u(y) \sd y\dd \xi\\
    &+ \Os\int_{\R^{2n}} e^{i(x-y)\cdot \xi } (1-\psi(x-y)) p(x,\xi) u(y) \sd y\dd \xi\equiv P'u + \op (a(x,y,\xi)) u
  \end{alignat*}
  where
  \begin{equation*}
    \op (a(x,y,\xi))\colon L_q(\Rn)\to H^s_q(\Rn)\qquad \text{for all }s<\min(\tau-m,\tau)
  \end{equation*}
  by Theorem~\ref{thm:Bddness2} since $a(x,y,\xi)= 0$ if $|x-y|<\delta/2$. Moreover,
\begin{alignat*}{1}
  \underline P' u(x) &=\Os\int_{\R^{2n}} e^{i(F(x)-z)\cdot \eta } \psi(F(x)-z) p(F(x),\eta) u(F^{-1} (z)) \sd z\dd \eta\\
  &=\Os\int_{\R^{2n}} e^{i(x-y)\cdot \xi } q(x,y,\xi) u(y) \sd y\dd \xi,
\end{alignat*}
where
\begin{alignat}{1} 
  q(x,y,\xi)&= \psi(F(x)-F(y))p(F(x),A(x,y)^{-1,T}\xi)|\det A(x,y)|^{-1}|\det \nabla_y F(y)|
\end{alignat}
for all $x,y,\xi\in\R^n$. The rest follows in the same way as in the proof of Theorem~\ref{thm:CoordTrafo1}, since $\eta(F(x)-F(y))|_{x=y}=1$ and $\partial_x^\alpha \left(\eta(F(x)-F(y))\right)|_{x=y}=0$ for every $|\alpha|\leq N$.  
\end{proof}

\noindent
\begin{proof*}{of Theorem~\ref{thm:main2}}
  The first part is a consequence of Corollary~\ref{cor:xyFormReduction}, Theorem~\ref{thm:bd-compos} applied to $p_\alpha$ and Corollary~\ref{cor:Bddnessr} applied to $r_\alpha$. Theorem~\ref{thm:CoordTrafo1} implies the second part.
\end{proof*}

\section{Preliminaries for Operators on Domains}\label{sec:FurtherPrelim}

\subsection{The ${\mu}$-transmission spaces for the halfspace and other smooth domains}\label{subsec:muTransmissionSpaces}

For $\Omega $ equal to $\rnp$ or a bounded smooth domain, the special {\it ${\mu} $-transmission spaces} were 
introduced for all $\mu \in{\mathbb C}$ by
H\"ormander \cite{H66} for $q=2$, cf.\ the account in \cite{G15} where
the spaces are redefined and extended to general $q\in (1,\infty)$. In the present paper we take $\mu $ real with $\mu >-1$. 
The spaces are defined by use of the order-reducing operators recalled
in Section \ref{subsec:PsDOs}: 
\begin{equation}
  \label{eq:2.9}
\begin{split}
H_q^{{\mu} (s)}(\crnp)&=\begin{cases}\Xi _+^{-{\mu} }e^+\ol H_q^{s- {\mu}
}(\rnp)=\Lambda  _+^{-{\mu} }e^+\ol H_q^{s- {\mu}
}(\rnp)
,&\text{ for }  s> {\mu} -\tfrac1{q'},\\
\dot H_q^{s}(\crnp),&\text{ for }  s\le {\mu} -\tfrac1{q'},
\end{cases}\\
H_q^{{\mu} (s)}(\comega)&=\begin{cases} \Lambda  _+^{(-{\mu} )}e^+\ol H_q^{s- {\mu}
}(\Omega ),&\text{ for }  s> {\mu} -\tfrac1{q'},\\
\dot H_q^{s}(\comega ),&\text{ for }  s\le {\mu} -\tfrac1{q'}.\end{cases}
\end{split}  
\end{equation}
Here $\frac1{q'}=1-\frac1q  $; for convenience of notation we have included the cases $s\le \mu
-\tfrac1{q'}$ (as mentioned in \cite[Definition 1.5]{G15}), although they
play a very small role in regularity studies. The spaces decrease with
growing $s$: $H_q^{{\mu} (s)}(\comega)\subset H_q^{{\mu}
  (s')}(\comega)$ when $s>s'$.

\begin{remark}\label{rem:Interpolation}
  From the definition and the interpolation properties of Bessel
  potential spaces it follows that $H_q^{{\mu}(s)}(\ol{\Omega })$
  is preserved under complex interpolation in $s$ when $s>\mu-\tfrac1{q'}$.
\end{remark}
We now list some further properties, formulated for $\Omega$ with the
convention that $\Lambda _+^{(t)}$ is replaced by $\Xi _+^t$ or
$\Lambda _+^t$ when $\rnp$ is considered. There holds (cf.\
\cite[Definition 1.8]{G15})
\begin{equation} \label{eq:2.10a}
  \|u\|_{ H_q^{{\mu} (s)}(\comega)}\simeq
  \| r^+\Lambda _+^{(\mu )}u\|_{\ol H_q^{s- {\mu}}(\Omega )}, \text{ when }s>\mu -\tfrac1{q'}.
\end{equation}
Observe in particular that
\begin{equation}
  \label{eq:2.10}
\begin{split}
H_q^{{\mu}(s)}(\comega)&=\dot
H_q^s(\comega) \text{ for }s-{\mu}\in (-\tfrac1{q'},\tfrac1q ) ,\\
\dot
H_q^s(\comega)&\subset H_q^{{\mu}(s)}(\comega)\subset H_{q,\operatorname{loc}}^{s}(\Omega) \text{ for all }s\in{\mathbb R},
\end{split} 
\end{equation}
since $e^+\ol H_q^{s- {\mu}}(\Omega )\supset \dot H_q^{s-\mu }(\comega) $ for all
$s>{\mu}-\tfrac1{q'}$, with equality if $s-{\mu}\in
(-\tfrac1{q'},\tfrac1q )$, and since $\Lambda _+^{(-\mu )}$ is elliptic.
We have moreover, for $s\ge {\mu}$,
\begin{equation}
  \label{eq:2.11}
H_q^{{\mu}(s)}(\comega)\subset H_q^{{\mu}({\mu})}(\comega)=\dot H_q^{\mu}(\comega).  
\end{equation}

The great interest of the spaces $H_q^{{\mu}(s)}(\comega)$ lies in the
following facts:
\begin{itemize} 
\item Pseudodifferential operators $P$ of order $m$ satisfying the $\mu $-transmission
condition map these spaces 
into standard spaces $\ol H_q^{s-m}(\Omega )$, by \cite[Theorem~4.2]{G15},
\begin{equation}
\label{eq:2.14a}
\|r^+Pu\|_{ \ol H_q^{s-m}(\Omega )}\le C\|u\|_{ H_q^{{\mu}(s)}(\comega)}\quad \text{ for } s>\mu -\tfrac1{q'}.
\end{equation}
\item When  $P$ furthermore is strongly elliptic, the spaces are
the solution spaces for the homogeneous Dirichlet problem, cf.\ \cite[Theorem~4.4]{G15}.
\end{itemize}

\begin{example}\label{ex:LaxMilgram} 
To demonstrate how \eqref{eq:2.9} enters into the
picture, we give a simple example: Let $P=\operatorname{OP}(p(\xi ))$,
where the $C^\infty $-function $p(\xi )$ is homogeneous of degree $2a$ and
even for
$|\xi |\ge 1$, and strongly elliptic satisfying $\operatorname{Re}p(\xi )\ge c>0$ for
$\xi \in \rn$. The Dirichlet problem on $\rnp$,
$$
r^+Pu=f \text{ on }\rnp, \quad \operatorname{supp}u\subset \crnp,
$$
is uniquely solvable for $f\in L_2(\rnp)$, when $u$ is sought in $\dot
H^a(\crnp)$; this follows by applying the Lax-Milgram lemma (cf.\
e.g.\ \cite[Section 12.4]{G09}) to the sesquilinear form
$s(u,v)=\int_{\rnp}r^+Pu\, \bar v\, dx$ on $\dot H^a(\crnp)$. A precise description of the domain $D(P_D)=\{u\in \dot H^a(\crnp)\mid
r^+Pu\in L_2(\rnp)\}$ can be found by letting $Q=\Lambda _-^{-a}P\Lambda _+^{-a}$; it
defines a bijection $Q_+=r^+Qe^+$ in $L_2(\rnp)$. Here $Q$ is of order
0 and satisfies Boutet de Monvel's $0$-transmission condition at
$\partial\rnp$, hence $Q_+$ is also a bijection in $\ol H^s_q(\rnp)$
for all $1<q<\infty $, $s>-1/q'$. When this is carried back to $P$, we
see that $r^+P$ maps the $a$-transmission space $\Lambda _+^{-a}e^+\ol H_q^{s+a}(\rnp)=H_q^{a(s+2a)}(\crnp)$ bijectively
to $\ol H_q^s(\rnp)$. In particular, $D(P_D)=H^{a(2a)}(\crnp)$.
\end{example}

We note that the spaces do not depend on $P$.
As a very special case of \eqref{eq:2.14a}, since $\Lambda _+^{(\mu )}$ satisfies the $\mu
$-transmission condition, so does the composition $P=\Lambda _+^{(\mu
  )}\circ\varphi $ for any $\varphi \in C^\infty _b(\rn)$, hence for any $s>\mu -\tfrac1{q'}$ there is some $C>0$ such that
\begin{equation*}
\|\varphi u\|_{ H_q^{\mu (s)}(\comega)}\simeq\|r^+\Lambda _+^{(\mu )}\varphi u\|_{ \ol H_q^{s-\mu }(\Omega )}
\le C\|u\|_{ H_q^{\mu(s)}(\comega)} \text{ for all }u\in H_q^{\mu(s)}(\comega),
\end{equation*}
cf.\ also \eqref{eq:2.10a}. Thus the multiplication by $\varphi $ maps
$ H_q^{\mu(s)}(\comega)$ into itself (also for lower $s$, since the
  property is well-known for the spaces $\dot H_q^s(\comega)$). 

With $d$ defined as in \eqref{eq:2.3a} (in particular, $d$ can equal $d_0$), there are local weighted boundary operators  
\begin{equation}
  \label{eq:2.12b}
\gamma _j^\mu u=\Gamma (\mu
+1+j)\gamma _j(u/d^\mu )\colon H_q^{\mu(s)}(\comega)\to B_q^{s-\mu -j-\frac1q}(\partial\Omega)
\end{equation}
defined for $s>\mu +j+\frac1q$, here $\gamma _j$ denotes the
$j$'th normal derivative $\gamma _jv=(\partial_n^jv)|_{\partial\Omega
}$, and $\Gamma $ denotes the Gamma function. 
There is a hierarchy (cf.\ \cite[Section~5]{G15}),
$H_q^{{\mu}(s)}(\comega)\supset H_q^{({\mu}+1)(s)}(\comega)\supset\dots
\supset  H_q^{({\mu}+j)(s)}(\comega)$ for $s>{\mu}+j-\tfrac1{q'}$, and
\begin{equation}
  \label{eq:2.12}
u\in H_q^{({\mu}+j)(s)}(\comega) \iff u\in H_q^{{\mu}(s)}(\comega)\text{ with }\gamma ^{\mu}_0u=\dots=\gamma ^{\mu}_{j-1}u=0;  
\end{equation}
  this is of importance for the study of
nonhomogeneous boundary conditions. 

Defining $\E_{\mu}(\comega)\equiv e^+d^{\mu}\ol C^\infty (\Omega )$,
we have for bounded smooth domains  that $ {\bigcap}_s
H_q^{{\mu}(s)}(\comega)=\E_{\mu}(\comega)$, which is dense in
$H_q^{\mu (s)}(\comega)$. For $\rnp$, $ {\bigcap}_s
H_q^{{\mu}(s)}(\crnp)=\E_{\mu}(\crnp)\cap {\bigcap}_s
H_q^{{\mu}(s)}(\crnp)$, and $\E_{\mu}(\crnp)\cap \E'(\rn)$ is dense in $H_q^{{\mu}(s)}(\crnp)$.

It was shown in \cite{G15} that
\begin{equation}
  \label{eq:2.14}
 H_q^{{\mu}(s)}(\comega)\subset \dot H_q^{s} (\comega)+e^+d^\mu \ol
 H_q^{s-\mu }(\Omega ),\text{ for }s>\mu +\tfrac1{q} , s-\mu -\tfrac1{q} \notin{\mathbb N},
\end{equation}
and the inclusion holds with $\dot H_q^{s}(\comega)$ replaced by $\dot
H_q^{s-\varepsilon }(\comega)$ if $s-\mu -\tfrac1{q}\in{\mathbb
  N}$.  There is a similar statement on $\rnp$ with $d$ replaced by
$x_n$.

In a recent paper \cite{G19}, the representations \eqref{eq:2.14} were
sharpened by an identification of which functions in $e^+d^\mu \ol
 H_q^{s-\mu }(\Omega )$ actually enter in $
 H_q^{{\mu}(s)}(\comega)$. 

For the H\"older-Zygmund scale $C^s_*$, the $\mu $-transmission spaces
are defined analogously by
\begin{alignat*}{2}
C_*^{\mu  (s)}(\comega)&=\Lambda  _+^{(-{\mu} )}e^+\ol C_*^{s- {\mu}
                         }(\Omega ),&\quad &\text{for }  s> {\mu} -1,\\
  C_*^{\mu  (s)}(\comega)&=\dot C_*^{s}(\comega ),&\quad &\text{for }  s\le
                           {\mu} -1,
\end{alignat*}
and all the properties listed above for $H^s_q$-spaces carry over to
the $C^s_*$-spaces. The formulas hold with $H^s_q$ replaced by $C^s_*$
and $\tfrac1{q},\tfrac1{q'}$ replaced by $0,1$. In particular, there
is the analogue of  \eqref{eq:2.14}: 
\begin{equation}
  \label{eq:2.14b}
 C_*^{{\mu}(s)}(\comega)\subset \dot C_*^{s} (\comega)+e^+d^\mu \ol
 C_*^{s-\mu }(\Omega ),\text{ for }s>\mu  , s-\mu  \notin{\mathbb N},
\end{equation}
and the inclusion holds with $\dot C_*^{s}(\comega)$ replaced by $\dot
C_*^{s-\varepsilon }(\comega)$ if $s-\mu \in{\mathbb
  N}$.  There is a similar statement on $\rnp$ with $d$ replaced by
$x_n$.

\subsection{Definitions in nonsmooth cases}\label{subsec: NonsmoothSets}

Recall that for any $C^{1}$-diffeomorphism $F\colon \R^n\to \R^n$ with $DF\in C^\tau(\Rn)^{n\times n}$ and $0\leq s< 1+\tau$ we have that
\begin{equation}\label{eq:change}
  F^\ast \colon H^s_q(\Rn)\to H^s_q(\Rn)\colon u\mapsto u\circ F
\end{equation}
is bounded. In the case $s\leq 1+[\tau]$, this follows by interpolation from the corresponding statement for $H^m_q(\Rn)$, $m=0,\ldots, 1+[\tau]$. In the case $1+[\tau]<s<1+\tau$ one uses that
\begin{equation*}
  \partial_{x_j} (u\circ F)(x)= (\nabla u)(F(x))\cdot \partial_{x_j} F(x) \quad \text{for }u\in H_q^s(\Rn),
\end{equation*}
where $\nabla u \circ F\in H^{s-1}_q(\Rn)$ since $s-1\leq 1+[\tau]$,
$\partial_{x_j} F \in C^{\tau}(\Rn)^n$, and one can apply well-known multiplication
results for Bessel potential spaces, namely that $fg\in H^t_q(\Rn)$ for any $f\in C^\tau(\Rn)$, $g\in H^t_q(\Rn)$ if $|t|<\tau$ and $1<q<\infty$, cf.\ e.g.\ the book by Runst and
Sickel~\cite[Section~4.7.1]{RunstSickel}. This result also follows from Theorem~\ref{thm:bd-compos} for $p(x,\xi)= f(x)$. The mapping is a bijection
if also $D(F^{-1})\in C^{\tau }(\Rn)^{n\times n}$.

In the following, let  $\Rn_\gamma= \{x\in\R^n\mid x_n>\gamma(x')\}$ for
some $\gamma\in C^{1+\tau}(\R^{n-1})$ with $\tau >0$, and let
$F_\gamma\colon \R^n\to \R^n$ be such that 
\begin{equation}\label{eq:change1} 
F_\gamma(x)=
(x',x_n-\gamma(x'))\quad\text{for all }x\in\R^n, \text{ where }x'=(x_1,\ldots,
x_{n-1});
\end{equation} 
note that for both $F_\gamma $ and
$F_\gamma^{-1} $ (defined by $F_\gamma(x)^{-1}=
(x',x_n+\gamma(x'))$) the differential is in $C^{\tau}(\R^n)^{n\times n}$.

Moreover, we consider a bounded open set $\Omega\subset \R^n$ with
$C^{1+\tau }$-boundary.

\begin{definition}\label{def:roughTransmSpaces} Let $\mu >-1$, $\tau
  >0$, $1<q<\infty $, and let
  $\mu -\frac1{q'}< s<1+\tau $.

  $1^\circ$ For the set $\Rn_\gamma$ with $\gamma \in C^{1+\tau }(\R^{n-1})$, we define
\begin{equation*}
u\in  H^{\mu(s)}_q(\ol{\R}^n_\gamma ) \iff u\circ F_\gamma ^{-1}\in H^{\mu(s)}_q(\crnp), 
\end{equation*}
and provide $ H^{\mu(s)}_q(\ol{\R}^n_\gamma ) $ with the inherited
norm. In other words,
\begin{equation}\label{eq:change2}
H^{\mu(s)}_q(\ol{\R}^n_\gamma ) = F^*_\gamma (H^{\mu(s)}_q(\crnp)).
\end{equation}

$2^\circ$ When $\Omega$ is a bounded $C^{1+\tau}$-domain, each point $x_0\in
\partial\Omega $
has a bounded open neighborhood $U\subset \R^n$ and
a $\gamma \in C^{1+\tau }(\R^{n-1})$,
such that (after a suitable rotation) $\Omega \cap U=\R^n_\gamma \cap
U$. We denote by
$H^{\mu(s)}_q(\overline{\Omega})$ the set of all $u\in
H^s_{q,loc}(\Omega)$ such that for each $x_0$, with a $\varphi\in
C_0^\infty(U)$ with $\varphi\equiv 1$ in a neighborhood of $x_0$, we have
\begin{equation*}
  F_\gamma ^{\ast,-1} (\varphi u):= (\varphi u)\circ F_\gamma ^{-1} \in H^{\mu(s)}_q(\ol{\R}^n_+)
\end{equation*}
in the rotated situation, with $F_\gamma $ defined by \eqref{eq:change1}.

$3^\circ$ There are similar spaces defined with $H^s_q$ replaced by $C^s_*$,
$q,q'$ replaced by $\infty ,1$.
\end{definition}

From  the inclusions
\eqref{eq:2.14}, \eqref{eq:2.14b} in the halfspace case we obtain:

\begin{proposition}\label{prop:roughTransmSpaces}  Let $\mu >-1$ and $\mu -\frac1{q'}< s<1+\tau $ with
  $s-\mu<1+\tau$, 
  and let $\gamma \in C^{1+\tau}(\R^{n-1})$. 
  There is a function $d$ as in {\rm \eqref{eq:2.3a}}ff.\ such that (with $\eps>0$): 
\begin{equation*}
H_q^{{\mu} (s)}(\ol{\R}^n_\gamma )\begin{cases} 
 = \dot H_q^{s}(\ol{\R}^n_\gamma  ), &\text{for }  s< {\mu} +\tfrac1{q},\\
 \subset  \dot H_q ^{s(-\varepsilon )}(\overline{\R}^n_\gamma )+d^\mu e^+\ol
  H_q^{s-\mu}({\ol{\R}^n_\gamma }),& \text{for }
    s> {\mu} +\tfrac1{q},
\end{cases}
\end{equation*}
where $(-\varepsilon )$ is active if $s-\mu -\tfrac1q\in\N$.

Moreover, the mapping $\gamma _0^\mu \colon u\mapsto \Gamma (\mu +1)(u/d^\mu
)|_{\partial \rn_\gamma }$ is continuous:
\begin{equation*}
\gamma _0^\mu \colon H_q^{{\mu} (s)}(\ol{\R}^n_\gamma )\to B_{q}^{s-\mu
  -\frac1q}(\partial \rn_\gamma), \text{ for } s-\mu -\tfrac1q>0.
\end{equation*}

If $\tau \geq 1$, then one can replace $d$ by $d_0$, chosen as a
$C^{1+\tau}$-function coinciding with the distance to
$\partial\Rn_\gamma$ near $\partial\rn_\gamma $, as indicated after
{\rm \eqref{eq:2.3a}}.

There are similar results in $C_*^s$-spaces with $\frac1q$ replaced by $0$.
\end{proposition}

\begin{proof}
The properties of $H_q^{{\mu} (s)}(\R^n_\gamma )$ follow from
\eqref{eq:2.9} and \eqref{eq:2.14}ff., since the function $x_n^\mu $
carries over to the function $d^\mu$ where $d(x)=x_n-\gamma (x')$ has the
mentioned properties, and since $F^{\ast,-1}$ maps $H^s_q(\Rn)$ and
$H^{s-\mu}_q(\Rn)$ to themselves. Similarly, the properties of
$C_*^{{\mu} (s)}(\R^n_\gamma )$ are carried over from \eqref{eq:2.14b}. The
definition of the trace follows from \eqref{eq:2.12b} for $\Omega=\Rn_+$.

If $\tau \ge 1$, we can replace $d(x)=x_n-\gamma (x')$ by $d_0(x)$,
since $d(x)=f(x)d_0(x)$, where $f\in C^\tau
(\overline{\R}^n_\gamma)$
is strictly positive, and multiplication with a $C^\tau (\overline{\R}^n_\gamma)$-functions maps $H_q^{s-\mu}({\ol{\R}^n_\gamma })$ into itself.
\end{proof}

For the use of  local coordinates, we need to know how the
multiplication by regular
functions acts in the transmission-spaces.

 \begin{proposition}\label{prop:Commutator1}
   Let $\mu >-1$ and $\sigma >0$.
   If $\mu \ge 0$, assume that $\sigma
   >\mu $ and $\mu -\frac1{q'}<s<\sigma$.
   If $-1<\mu < 0$, assume that  $\sigma >\mu +1$ and
   $\mu -\frac1{q'}<s<\sigma -1$.
   Then multiplication by a function $\varphi \in C^\sigma(\rn)$ maps $ H_q^{\mu  (s)}(\crnp)$ into itself.

  Consequently, if $\gamma\in C^\sigma (\R^{n-1})$, then
  multiplication by a function $\varphi \in C^\sigma(\rn)$ maps $
  H_q^{\mu  (s)}(\ol{\R}^n_\gamma )$ into itself.
  
\end{proposition}

\begin{proof} 
For $s-\mu \in (-\frac1{q'},\frac1q )$,  $H_q^{\mu
  (s)}(\crnp)=\dot H_q^s(\crnp)$, where the property is well-known since
$|s|<\sigma $, so
we can assume $s\ge \mu $.
There holds in general:
\begin{equation}\label{eq:2.15b}
v\in  H_q^{\mu  (s)}(\crnp)\iff \Lambda _+^\mu v\in e^+\ol H_q^{s-\mu }(\rnp).
\end{equation}
Let $u\in  H_q^{\mu  (s)}(\crnp) $, and let  $\varphi \in C^{\sigma
  }_b(\rn)$. To show that  $\varphi u\in  H_q^{\mu  (s)}(\crnp)$, we
  must show that $
\Lambda _+^\mu (\varphi u)\in e^+\ol H_q^{s-\mu }(\rnp)$.
Here 
\begin{equation*}
\Lambda _+^\mu (\varphi u)=\varphi \Lambda _+^\mu  u+[\Lambda _+^\mu, \varphi] u.
\end{equation*}
For the first term, $\Lambda _+^\mu u\in e^+\ol H_q^{s-\mu }(\rnp)$ by
hypothesis, and multiplication by $\varphi $ is known to preserve this space
since $|s-\mu |<\sigma $.
It remains to show that
\begin{equation}\label{eq:2.15a}
[\Lambda _+^\mu ,\varphi] u\in e^+\ol H_q^{s-\mu }(\rnp).
\end{equation}
 Here we shall borrow some
continuity properties shown later in Theorem~\ref{thm:BddnessTruncatedPsDOs} and Corollary~\ref{cor:BddnessHalfspace}.
The operator $[\Lambda _+^\mu, \varphi]$ may be written as a $\psi $do
in $(x,y)$-form, 
\begin{equation*}
[\Lambda _+^\mu, \varphi]=\op(a(x,y,\xi ));\quad a(x,y,\xi )=\lambda _+^\mu (\xi )(\varphi (x)-\varphi (y)),
\end{equation*}
satisfying the global $\mu $-transmission condition and with symbol in
$C^{\sigma }S^{\mu }(\R^{2n}\times \rn)$.

If $\mu \ge 0$, we apply
Theorem~\ref{thm:BddnessTruncatedPsDOs} with $\tau =\sigma $, $m=\mu $, and $s$ replaced by $s'$, to conclude that 
\begin{equation}\label{eq:2.15c}
r^+[\Lambda _+^\mu, \varphi]\colon  H_q^{\mu  (\mu +s')}(\crnp)\to \ol
H_q^{s' }(\rnp), 
\end{equation}
when $\sigma >\mu $, $ 0\le s'<\sigma  -\mu $.
The operator preserves support in $\crnp$ since $\Lambda _+^\mu$ does so. With $s'=s-\mu $ we conclude that when $\sigma >\mu $, $ \mu \le s<\sigma $,
\begin{equation}\label{eq:2.15d}
[\Lambda _+^\mu, \varphi]\colon  H_q^{\mu  (s)}(\crnp)\to e^+\ol
H_q^{s-\mu  }(\rnp).
\end{equation}

If $\mu \in (-1,0)$, we apply Corollary~\ref{cor:BddnessHalfspace} below with $\tau =\sigma $, $m=\mu $, and $s$
replaced by $s'$, to conclude that \eqref{eq:2.15c} holds when $\sigma
>\mu +1$, $0\le s'<\sigma -\mu -1$.  With $s'=s-\mu $ we conclude that 
\eqref{eq:2.15d} holds when $\sigma >\mu +1$, $\mu \le s<\sigma -1$.

For the last statement, we note that when $u\in  H_q^{\mu  (s)}(\ol{\R}^n_\gamma )$, then $\varphi u\in  H_q^{\mu
  (s)}(\ol{\R}^n_\gamma )$ holds if  $(\varphi u)\circ F_\gamma ^{-1}\in  H_q^{\mu
  (s)}(\crnp)$, by Definition \ref{def:roughTransmSpaces}. Here
$(\varphi u)\circ F_\gamma ^{-1}=(\varphi \circ F_\gamma ^{-1})(
u\circ F_\gamma ^{-1})$, where $\varphi \circ F_\gamma ^{-1}\in
C_b^\sigma (\rn)$, so the statement follows from what was proved in the
case of $\crnp$.
\end{proof}

The inclusions \eqref{eq:2.14}ff.\ are shown for bounded $C^{1+\tau
}$-domains as follows:

\begin{thm}\label{thm:TransmissionSpace} Let $\mu >-1$ and $\mu -\frac1{q'}< s<\tau $ with
  $s-\mu<\tau$, and  let $\Omega\subset\R^n$ be a bounded  $C^{1+\tau}$-domain with $\tau \geq 1$.
Then (with $\varepsilon >0$)
  \begin{equation}\label{eq:2.18}
H_q^{{\mu} (s)}(\comega)\begin{cases}   =\dot H_q^{s}(\comega ),&\text{for }  s< {\mu} +\tfrac1{q},\\
  \subset \dot H_q^{s-\varepsilon }(\comega ), &\text{for }  s= {\mu} +\tfrac1{q},\\
 \subset  \dot H_q^{s}(\comega)+ d_0^\mu e^+\ol
  H_q^{s-\mu}({\Omega}),& \text{for }
  s- {\mu} -\tfrac1{q}\in \rp\setminus \N,\\
   \subset  \dot H_q^{s-\varepsilon }(\comega)+ d_0^\mu e^+\ol
  H_q^{s-\mu}({\Omega})& \text{for }
  s- {\mu} -\tfrac1{q}\in\N.
\end{cases}
\end{equation}
Moreover, the mapping $\gamma _0^\mu \colon u\mapsto \Gamma (\mu +1)(u/d_0^\mu
)|_{\partial \Omega }$ is continuous:
\begin{equation*}
\gamma _0^\mu \colon H_q^{{\mu} (s)}(\comega )\to B_q^{s-\mu
  -\frac1q}(\partial \Omega ),\quad \text{ for } s-\mu -\tfrac1q>0.
\end{equation*}
There are similar results in $C_*^s$-spaces, with $q,q'$ replaced by
$\infty ,1$. 
\end{thm}
\begin{proof} We formulate the proof for $H_q^s$-spaces; the proof
  for $C_*^s$-spaces is similar. We can assume $s> \mu +\frac1q$,
  noting that the identity is known when $s\in (\mu -\frac1{q'}, \mu
  +\frac1q)$, and the spaces decrease with increasing $s$.

Let $u\in H_q^{{\mu} (s)}(\comega )$ and let $x_0,U,\gamma ,\varphi $ be as in Definition~\ref{def:roughTransmSpaces} $2^\circ$. Let $\psi \in C_0^\infty
(U)$ satisfy $\psi \varphi =\varphi $. Let $U'$ be the interior of the
set where $\varphi =1$. By Proposition
\ref{prop:roughTransmSpaces} with  $d=d_0$,
\begin{equation}\label{eq:2.18a}
\varphi u=w+d_0^\mu e^+v=\psi w+d_0^\mu e^+\psi v\text{ with }\psi
w\in \dot H_q^s(\comega), \psi v\in \ol H_q^{s-\mu }(\Omega ),
\end{equation}
using that multiplication by $\psi $ preserves the space, by
Proposition \ref{prop:Commutator1}.

There is a finite set of points $\{x_{0,i}\}_{i=1,\dots,I}$ such that
$\bigcup_iU'_i\supset \partial\Omega $ holds for the associated data $\{U_i,\gamma _i,\varphi _i,U'_i,\psi _i\}$.
Supply these sets with an open set
$U'_0\supset \comega\setminus \bigcup_iU'_i $ with $\overline U'_0\subset \Omega $, and let $\{\varrho
_i\}_{i=0,\dots,I}$ be an associated partition of unity, $\varrho
_i\in C_0^\infty (U'_i)$. Then $u=\sum_i\varrho _iu$. Moreover,
$\varrho _iu=\varrho _i\varphi _iu$ for $i\ge 1$, where the $\varphi
_iu$ satisfy \eqref{eq:2.18a}. 

Summation over $i$ gives the statement in \eqref{eq:2.18} with $d$
replaced by $d_0$. (Note that the functions $d_0$ in the different charts
are consistent near $\partial\Omega $, and their extensions further away
play no role since $u$ is in $H_{q,\operatorname{loc}}^s(\Omega )$
  anyway).

  The statement on the trace operator follows from Proposition \ref{prop:roughTransmSpaces}  in a similar way.
\end{proof}
\begin{remark}\label{rem:TransmissionSpace}
 If we replace the assumption $\tau\geq 1$ in
 Theorem~\ref{thm:TransmissionSpace} by $\tau>0$, we still obtain the
 following local inclusion: If $u\in H_q^{{\mu} (s)}(\comega )$ and
 $x_0,U,\gamma ,\varphi $ are as in
 Definition~\ref{def:roughTransmSpaces} $2^\circ$, then by
 Proposition~\ref{prop:roughTransmSpaces}:  
\begin{equation*}
\varphi u=w+d^\mu e^+ v\qquad \text{ with }w\in \dot H_q^s(\ol{\R}^n_\gamma),\; v\in \ol H_q^{s-\mu }(\ol{\R}^n_\gamma ),
\end{equation*}
where $d$ may depend on $\gamma$.
\end{remark}

Also higher traces as in  \eqref{eq:2.12b} can be defined; we intend
to take up their applications in later works. 

The concepts are applied to the fractional Laplacian $(-\Delta )^a$
and its generalizations  primarily for $a\in (0,1)$, but also higher
values of $a$ are
of interest. The $a$-transmission spaces enter as solution spaces for
the homogeneous Dirichlet problem. The $(a-1)$-transmission spaces allow
the definition of nonzero Dirichlet and Neumann traces; this is the reason that
we have made an effort to include cases $\mu \in (-1,0)$ in the treatment.

The following commutation result will be needed later:
\begin{proposition}\label{prop:Commutator}
  Let $p\in C^\tau S^m_{1,0}(\Rn\times \Rn)$ for some $\tau>0$, $\tau\not\in\N$, $m\in\R$ and $\varphi \in C^\infty_b(\Rn)$. Then there is some $q\in C^\tau S^{m-1}_{1,0}(\Rn\times \Rn)$ such that
  \begin{equation*} 
    [\op(p),\varphi]u= \op (q)u \qquad \text{for all }u\in \SD(\Rn).
  \end{equation*}
  Moreover, if $p\in C^\tau S^m(\Rn\times\Rn)$, then $q \in C^\tau S^{m-1}(\Rn\times\Rn)$ and, if $p$ is even, then $q$ is odd.
\end{proposition}
\begin{proof}
  First of all we have
  \begin{align*}
    [P,\varphi] u(x)&= \Os\int_{\R^{2n}} e^{i(x-y)\cdot \xi} p(x,\xi) (\varphi(y)-\varphi(x)) u(y)\,dy\, \dd\xi\\
                    &= \Os\int_{\R^{2n}} e^{i(x-y)\cdot \xi} p(x,\xi) (y-x)\cdot \Phi(x,y) u(y)\,dy\, \dd\xi\\
                    &= \Os\int_{\R^{2n}} e^{i(x-y)\cdot \xi} D_\xi p(x,\xi)\cdot \Phi(x,y) u(y)\,dy\, \dd\xi
  \end{align*}
for all $u\in\SD(\Rn)$ and $x\in\Rn$,  where $\Phi\in C^\infty_b(\Rn\times \Rn)^n$ is defined by
\begin{equation*}
  \Phi(x,y)=\int_0^1 (\nabla \varphi)(x+t(y-x))\, dt\qquad \text{for all }x,y\in \Rn.
\end{equation*}
Hence
\begin{equation}
  a(x,y,\xi)= D_\xi p(x,\xi)\cdot \Phi(x,y)\qquad \text{for all }x,y,\xi\in\Rn
\end{equation}
defines a symbol in $C^\tau S^{m-1}_{1,0}(\Rn\times \Rn\times\Rn;\infty)$ as defined in \cite{AbelsPfeufferCharacterization}. Therefore, because of \cite[Theorem~4.15]{AbelsPfeufferCharacterization}, there is some $q=a_L\in C^\tau S^{m-1}_{1,0}(\Rn\times \Rn)$ such that
\begin{equation*}
  [P,\varphi] u(x)                    = \Os\int_{\R^{2n}} e^{i(x-y)\cdot \xi} a(x,y,\xi) u(y)\,dy \dd\xi = \op (q)u(x)
\end{equation*}
for all $u\in \SD(\Rn)$, $x\in\Rn$. More precisely,
\begin{equation*}
  a_L(x,\xi)= \Os\int_{\R^{2n}} e^{-iy\cdot \eta} a(x,x+y,\eta+\xi)dy\dd \eta\quad \text{for all }x,\xi\in\Rn
\end{equation*}
and it follows first follows from \cite[Theorem~4.15]{AbelsPfeufferCharacterization} that $a_L\in C^\tau S^{m-1}_{0,0}(\Rn\times \Rn)$. In order to see that $a_L\in C^\tau S^{m-1}_{1,0}(\Rn\times \Rn)$ one uses that
\begin{equation*}
  \partial_\xi^\alpha a_L(x,\xi)= \Os\int_{\R^{2n}} e^{-iy\cdot \eta} (\partial_\xi^\alpha a)(x,x+y,\eta+\xi)dy\dd \eta\quad \text{for all }x,\xi\in\Rn
\end{equation*}
due to \cite[Theorem~2.11]{AbelsPfeufferCharacterization}, 
where $\partial_\xi^\alpha a \in C^\tau S^{m-|\alpha|-1}_{1,0}(\Rn\times \Rn\times \Rn;\infty)$. Applying  \cite[Theorem~4.15]{AbelsPfeufferCharacterization} again, we obtain $\partial_\xi^\alpha a_L\in C^\tau S^{m-|\alpha|-1}_{0,0}(\Rn\times \Rn)$ for all $\alpha\in\N_0$, i.e, $q=a_L\in C^\tau S^{m-1}_{1,0}(\Rn\times \Rn)$.

Finally, we note that, using a Taylor expansion of $a(x,x+y,\eta+\xi)$ with respect to $\eta$ (around $0$) in a standard manner, it is easy to show that we have the asymptotic expansion
\begin{equation*}
  a_L(x,\xi)\sim {\sum}_{\alpha\in\N_0^n}\tfrac1{\alpha!} \partial_\xi^\alpha D_y^\alpha a(x,y,\xi)|_{y=x}.
\end{equation*}
Moreover, $a\in C^\tau S^{m-1}(\R^{2n}\times \Rn)$ if $p\in C^\tau
S^{m}(\Rn\times \Rn)$. Using the asymptotic expansion one easily
observes that $q=a_L\in C^\tau S^{m-1}(\Rn\times \Rn)$. Furthermore, one
verifies in a straightforward manner that $a$ and $q=a_L$ are odd if
$p$ is even.
\end{proof}

There is a corollary to Proposition~\ref{prop:Commutator} showing that an
operator sandwiched between smooth functions with disjoint supports
acts like an operator of arbitrarily low order and the same H\"older smoothness:

\begin{corollary}\label{cor:Commutator} Let $p\in C^\tau S^m_{1,0}(\Rn\times \Rn)$
  for some $\tau>0$, $\tau\not\in\N$, $m\in\R$, and let  $\varphi,\psi
\in C^\infty_b(\Rn)$ with disjoint supports. For any $N\in\N$ there is a $q_N\in C^\tau S^{m-N}_{1,0}(\Rn\times \Rn)$ such that
$$ 
   \varphi  \op(p)\psi u= \varphi \op (q_N)u \qquad \text{for all }u\in \SD(\Rn).
$$  
(If $p\in C^\tau S^m(\Rn\times\Rn)$, then $q_N \in C^\tau
S^{m-N}(\Rn\times\Rn)$ and, if $p$ is even, then $q_N$ is even for
even $N$, odd for odd $N$.) 
\end{corollary}

\begin{proof} Setting $\psi _0=\psi $, we can for $N=1,2,\dots$ choose
a nested sequence of
$C_b^\infty $-functions $\psi _N$ with supports disjoint from
$\supp\varphi $,
such that $\psi _N$ is 1 on $\supp \psi _{N-1}$ for all
$N$. 
By Proposition~\ref{prop:Commutator}, $
\varphi  \op(p)\psi _0 =\varphi  [\op(p),\psi _0 ]= \varphi \op(q_1)
$
 with $q_1\in C^\tau S^{m-1}_{1,0}(\Rn\times \Rn) $. Since $\psi
 _1\psi _0=\psi _0$, we can repeat the argument with
$$
\varphi  \op(p)\psi _0 =\varphi  \op(p)\psi _0 \psi _1=\varphi  \op(q_1)\psi _1= \varphi  [\op(q_1),\psi _1] =\varphi  \op(q_2),$$
where  $q_2\in C^\tau S^{m-2}_{1,0}(\Rn\times \Rn) $. Continuing in
this way, with the $N$-th step
being
$$
\varphi  \op(p)\psi _0 =\varphi  \op(p)\psi _0 \psi _N=\varphi  \op(q_{N})\psi _{N}= \varphi  [\op(q_{N}),\psi _{N}] =\varphi  \op(q_{N+1}),
$$
shows the main assertion. The last statement follows from the
corresponding statement in Proposition~\ref{prop:Commutator}.
\end{proof}

\section{Nonsmooth Transmission Conditions}\label{sec:TransmissionCond}

\subsection{Transmission conditions for nonsmooth symbols.}

For the consideration of $\psi $do's $P$ on open subsets of $\rn$ one needs
conditions that govern their behavior at a boundary. There have been
many contributions through the times, mainly for $\psi $do's with
smooth $x$-dependence. The {\it transmission property} in case of a
smooth open set $\Omega $ is the property that $r^+Pe^+$ maps  $C^\infty
(\comega)\cap \E'(\rn)$ into $C^\infty (\comega)$. Necessary and sufficient conditions for this
property have been established in several works,
and sufficient conditions have been
introduced under additional requirements (e.g.,
parameter-dependence) --- called {\it transmission
  conditions}. (References are given in Example \ref{ex:transexamples} below.)

For some $\psi $do's $P$, $r^+Pe^+$ does not preserve  $C^\infty
(\comega)$, but maps another space $d^\mu C^\infty (\comega)$ into
$C^\infty (\comega)$; they satisfy the so-called {\it $\mu $-transmission
condition}, where the abovementioned case is the case $\mu =0$.

We now consider nonsmooth situations:  $P=\op(p(x,\xi ))$ is a  $\psi
$do with symbol $p$ in $  C^\tau 
S^m(\Rn\stimes\Rn)$, and it is considered relative to an open subset $\Omega \subset \rn$ with
$C^{1+\tau }$-boundary. The definition of the $\mu
$-transmission condition from \cite{H66} and \cite{G15}  (and
\cite[Section~18.2]{H85}  with a different notation) will here be
generalized to be the
requirement that the difference between $p $ and a certain symbol
obtained by  twisted reflections of the homogeneous symbol terms
vanishes to the order $\tau $ at $\partial\Omega $. In details:

\begin{definition}\label{def:roughmu-transm}
Let $m,\mu \in \R$ and $\tau \in \crp$.  
Let $p(X,\xi )\in C^\tau S^m({\mathbb R}^{n'}\times
{\mathbb R}^n)$ with the expansion in homogeneous terms  $p(X,\xi
)\sim \sum_{j\in{\mathbb N}_0}p_j(X,\xi )$.

$1^\circ$ $p(X,\xi )$ will be said to satisfy the 
$\mu $-transmission condition with respect to $\crnp$ at $X$, when
there holds for all $j\in \mathbb N_0$, $ \alpha  \in{\mathbb N}_0^n$, 
\begin{equation}\label{eq:exroughmu-transm}
 \partial_\xi ^\alpha {p_j}(X,0,-1)=e^{i \pi (m-2\mu -j-|\alpha | )
} \partial_\xi ^\alpha{p_j}(X,0,1).
  \end{equation} 
For $n'=n$, $X=x$,
$p(x,\xi )$ will be said to satisfy the 
$\mu $-transmission condition with respect to $\crnp$, when for all
$x'\in \mathbb R^{n-1}$, $j\in \mathbb N_0$, $ \alpha  \in{\mathbb N}_0^n$, 
\begin{align}\label{eq:roughmu-transm}
\partial_x^\beta \partial_\xi ^\alpha {p_j}(x',0,0,-1)&=e^{i \pi (m-2\mu -j-|\alpha | )
}\partial_x^\beta \partial_\xi ^\alpha{p_j}(x',0,0,1),
\text{ for } |\beta |\le \tau .
\end{align}

$2^\circ$ For $n'=n$ or $2n$, $X=x$ or $(x,y)$,
$p(X,\xi )$ will be said to satisfy the {\bf extended}
$\mu $-transmission condition with respect to $\crnp$, when there is
an $\varepsilon >0$ such that for the points
$X$ in the region $\{0\le x_n<\varepsilon \}$ resp.\ $\{0\le
x_n,y_n<\varepsilon \}$, {\rm \eqref{eq:exroughmu-transm}} holds for  $j\in \mathbb N_0$, $ \alpha  \in{\mathbb N}_0^n$. 
The condition
is said to be {\bf global} if all $\varepsilon >0$ can be used.

$3^\circ$ For an open set $\Omega $ with $C^{1+\tau 
}$-boundary, analogous
conditions are formulated  with $(x',0)$ replaced by
$x_0\in\partial\Omega $, $(0,1)$ replaced
by $\nu (x_0)$, and $x=(x',x_n)$  replaced by
$x=x_0+t\nu (x_0)$, $x_0\in \partial\Omega $ and $0\le
t<\varepsilon $ (with $t$ playing the role of $x_n$).

\end{definition}

Note that in $2^\circ$, equalities for $X$-derivatives as in $1^\circ$
follow simply by differentiating the identities
\eqref{eq:exroughmu-transm} up to order $[\tau ]$. 
Note also
that addition of an integer to $\mu $ does not change the
formulas; the conditions depend only on $\mu (\operatorname{mod}\Z)$. 

We need the extended 0-transmission condition in order to apply the
results of Abels \cite{A05}, where the mapping properties for truncated
integer-order operators depend on an extended definition of Poisson
operators.

\begin{example}\label{ex:transexamples}
  \begin{enumerate}
  \item
  {\it The case $\mu =0$.}  The operators considered by Boutet de
Monvel \cite{B71}, Grubb 
\cite{G84,G96,G09}, Rempel and Schulze \cite{RS82} are of integer
order and satisfy the 0-transmission condition with $\tau =\infty
$. In \cite{B66} also noninteger-order classical $\psi $do's were included. Grubb and H\"ormander \cite{GH90} treated
operators of arbitrary orders with symbols in $S^m_{\varrho
  ,\delta }$-spaces,  giving general conditions that are necessary
and sufficient for the transmission property.
Abels
\cite{A05} introduced a generalization to operators of integer order
with finite $C^\tau
$-smoothness (defining a slightly different version of the global
0-transmission condition).
 
 \item
 {\it General $\mu $.} Simple examples of  symbols satisfying
the global $\mu $-transmission condition with respect to $\crnp$ are
$\chi ^\mu _+(\xi )=(\ang{\xi '}+i\xi _n)^\mu $ and its truly $\psi
$do variant $\lambda ^\mu _+$ (see (2.10)ff.); here $\chi _{+,0}^\mu
(0,\pm 1)=\lambda
_{+,0}^{\mu }(0,\pm 1)=(\pm i)^{\mu }$. Note that $\lambda ^\mu _-$
is of 0-transmission type.

 \item
 {\it Even symbols.}  When $P$ is of order $m=2a$ and
 even, cf.\ \eqref{eq:even}, it satisfies the $a$-transmission condition with respect to any
halfspace, so it in fact fulfills the global $a$-transmission
condition relative to any $C^{1+\tau  }$-domain.
 Examples of operators in this category are: {\it  fractional powers of
elliptic differential operators}, including $(-\Delta )^a$. (Cf.\ 
 Lemma 2.9 and Example 3.2 in \cite{G15}.)
Note that the symbol $p'=p-p_0$, considered as a symbol of order
$m'=m-1$, is {\it odd}. If $m=2a$, hence $m'=2a-1$, $p'$ satisfies
the $a$-transmission condition.
  \end{enumerate}
\end{example}

The conditions are preserved under multiplication in the following way:

\begin{prop}\label{thm:multipl}
When $p(X,\xi )\in C^{\tau }S^{m}({\mathbb R}^{n'}\times
{\mathbb R}^n)$ and
$p'(X,\xi )\in C^{\tau }S^{m'}({\mathbb R}^{n'}\times
{\mathbb R}^n)$, satisfying the (extended) $\mu $-transmission resp.\ $\mu '$-transmission condition,
 then $p(X,\xi )p'(X,\xi )\in C^{\tau }S^{m+m'}({\mathbb R}^{n'}\times
{\mathbb R}^n)$ satisfying the (extended) $(\mu +\mu ')$-transmission condition.
\end{prop}

\begin{proof}
This follows straightforwardly from the definition, using that $C^\tau
$ is  preserved under multiplication. The $j$'th term in $pp'$ is 
$(pp')_j={\sum}_{k+l=j}p_kp'_l$, homogeneous of order $m+m'-j$,
where 
$$
p_k(X,0,-\xi _n)p'_l(X,0,-\xi _n)=e^{i\pi (m+m'-2(\mu +\mu ')-j)
}p_k(X,0,\xi _n)p'_l(X,0,\xi _n).
$$
\end{proof}

The $0$-transmission condition for $p(X,\xi )$ can also be expressed by formulations in terms of
$\tilde p(X,\xi ',z_n)=\F^{-1}_{\xi _n\to z_n}p$, where it means
smoothness from the right:

\begin{theorem}\label{thm:0transestimates} Let $m \in{\mathbb R}$ and $\tau \in
  \crp$, and let $p(X,\xi )\in C^\tau S^m({\mathbb R}^{n'}\times
{\mathbb R}^n)$. Then
$p$ satisfies the $0$-transmission condition with respect to $\crnp$ at $X$ 
({\rm \eqref{eq:exroughmu-transm}} with $\mu =0$) if and only if 
$r^+\tilde p(X,\xi ',z_n)=r^+\mathcal F^{-1}_{\xi _n\to z_n}p(X,\xi )$
satisfies: 
\begin{equation} \label{eq:right-smooth}
r^+z_n^{\alpha _n}\partial_{\xi '}^{\alpha '}\tilde
p_j(X,0,z_n)\in C^\infty (\crp),\text{ all }j\in\N_0,\alpha \in \N_0^n.
\end{equation}

When $p(X,\xi )$ satisfies the extended $0$-transmission condition
with respect to $\crnp$,
$\tilde p$ moreover satisfies
the estimates, for all $j,k,l\in \mathbb N_0,\alpha \in \mathbb
N_0^{n-1}$:
\begin{align}\label{eq:0transestimates}
\|z_n^k\partial_{z_n}^l\partial_{\xi '}^\alpha & r^+\tilde p_j(.,\xi
',z_n)\|_{C^\tau (U;L_{2,z_n}(\rp))}\le C_{k,l,\alpha
}\ang{\xi '}^{m+\frac12-k+l-j-|\alpha |},\\
\|z_n^k\partial_{z_n}^l\partial_{\xi '}^\alpha & r^+\tilde p(.,\xi
',z_n)\|_{C^\tau (U;L_{2,z_n}(\rp))}\le C_{k,l,\alpha
}\ang{\xi '}^{m+\frac12-k+l-|\alpha |},\label{eq:full0transestimates}
\end{align}
where $U = \{x\in\R^n\mid x_n\in [0,\eps)\}$ if $n'=n$ and $U =
\{(x,y)\in \R^{2n}\mid  x_n,y_n\in [0,\eps)\}$ if $n'=2n$. In the case
that $p$ satisfies the global $0$-transmission condition with respect
to $\ol{\R}^n_+$, the estimates
\eqref{eq:0transestimates}--\eqref{eq:full0transestimates} hold true
with $U=\crnp$ resp.\ $\crnp\times\crnp$.
\end{theorem}

\begin{proof}
Many of the ingredients in the proof were already present in 
\cite{B66}. The details we give below make use of later
developments.

First let  $m\in\Z$. Here \eqref{eq:roughmu-transm} takes the form
\begin{equation}\label{eq:int0transm}
\partial_\xi ^\alpha {p_j}(X,0,-\xi _n)=(-1)^{m-j-|\alpha | 
}\partial_\xi ^\alpha{p_j}(X,0,\xi _n),
\end{equation}
and clearly holds also with $\xi _n$ and $-\xi _n$ interchanged (is in
some texts in fact written as such), so it is {\it two-sided}, valid also with
respect to $\crnm$. In \cite{B71}, and in many later works, e.g.\
\cite{G96,G09,S01,A05}, it is replaced by a condition where $p$ as a
function of $\xi _n$ takes
values in the space $\HS=\F_{z_n\to \xi _n}(e^+\SD(\crp)\oplus e^-\SD(\crm))\oplus\C[\xi  _n]$ with certain estimates; they
imply  \eqref{eq:right-smooth}--\eqref{eq:full0transestimates}, and
vice versa. The equivalence of \eqref{eq:int0transm}
with the $\HS$-estimates is shown in \cite{B71}, and in \cite[Theorem~2.2.5]{G96} in a situation with a parameter. 
We shall
not take up further space  here with details.

Now let $m\in \R\setminus \Z$. First consider the case where $m<-1$.
We study each homogeneous term in the symbol individually; take for
example $p_0$. 
Consider  $p_0(X,0,\xi _n)$ at a fixed $X$. It is
homogeneous in $\xi _n$ of degree $m $ for $|\xi _n|\ge 1$. By the
rules of Fourier transformation of
homogeneous functions of one variable we have, as shown in detail  e.g.\ in  Lemma
2.7 of \cite{G15} (which takes the behavior for $|\xi _n|\le 1$ into
account), that the twisted parity property 
\begin{equation*}
p_0(X,0,-\xi
_n)=e^{i\pi m }p_0(X,0,\xi _n)\text{ for }|\xi _n|\ge 1,
\end{equation*} 
holds if and only if $r^+\tilde
p_0(X,0,z_n)$ is zero on $\rp$ modulo $C^\infty (\crp) $. Likewise,
the $\xi $-derivatives have the corresponding twisted parity if and
only if $r^+z_n^{\alpha _n}\partial_{\xi '}^{\alpha '}\tilde
p_0(X,0,z_n)$ is zero on $\rp$ modulo $C^\infty (\crp) $.
 This shows
the equivalence with  \eqref{eq:right-smooth} for $j=0$.

The consequences can be further analyzed:
Since $p_0(X,\xi )$ is a symbol of order  $m $, one has that 
\begin{equation*}
\|\partial_{\xi _n}^k \xi _n^l\partial_{\xi '}^\alpha 
p_0\|_{L_{2,\xi_n}({\mathbb R})} \le C_{k,l,\alpha }\ang{\xi
  '}^{m+\frac12-k+l-|\alpha |},\text{ for }k>m+l-|\alpha |+2, 
\end{equation*}
and hence
$z_n^k\partial_{z_n}^l\partial_{\xi '}^\alpha \tilde p_0$ is in
$L_{2,z_n}({\mathbb R})$ and
\begin{equation}\label{eq:plus-estimate}
\|z_n^k\partial_{z_n}^l\partial_{\xi '}^\alpha r^+\tilde
p_0\|_{L_{2,z_n}(\rp)} \le \|z_n^k\partial_{z_n}^l\partial_{\xi '}^\alpha
\tilde p_0\|_{L_{2,z_n}({\mathbb R})}\le C_{k,l,\alpha }\ang{\xi '}^{m+\frac12-k+l-|\alpha |}, 
\end{equation}
for such indices. When $\xi '=0$, the left entry is moreover, a fortiori,
bounded for all lower values of $k\in\mathbb N_0$,
 in view of the smoothness for $z_n\to 0+$ shown above. 
In particular, 
\begin{equation}\label{eq:plus-limit}
r^+\partial_{\xi '}^{\alpha }\tilde p_0(X,0,z_n)\in \mathcal S(\crp)
\end{equation}
 as a function of
$z_n$.

We can extend
 the estimates to $\xi '\ne 0$ by a Taylor expansion in $\xi '$, using
 the estimates for $\xi '=0$ (and handling remainders by use of symbol estimates), as in
 the detailed proof of Theorem 2.6 in \cite{G15}.

The estimates moreover hold with $\partial_{X}^\beta $ inserted, for
$|\beta |\le [\tau ]$. 
This shows that estimates  \eqref{eq:0transestimates}  hold for $p_0$ with $\tau $
replaced by $[\tau ]$. When $\tau =[\tau ]+\sigma $, $\sigma \in
(0,1)$, we also apply the considerations to $(\partial_{}^\beta p_0(X,\xi
)-\partial_{x}^\beta p_0(Y,\xi ))/|X-Y|^\sigma $ for $|\beta |=[\tau ]$, $X\ne Y$;
these functions likewise have the twisted parity property allowing to conclude
the smoothness for $z_n\to 0+$ of the inverse Fourier transformed
function when $\xi '=0$, with
estimates as above. This can, as above, be extended to estimates of the type
\eqref{eq:0transestimates} with $\tau $ replaced by 0. The validity
with uniform bounds in $X,Y$ then implies that \eqref{eq:0transestimates}
 holds for $p_0$.

There are similar proofs for the other terms $p_j$.

Larger values of $m$ are included as follows: When a positive
integer $r$ is chosen so large that $m-2r<-1$, then the above analysis
applies to $q(X,\xi )=p(X,\xi
)(1+\sum_{l=1}^n\xi _l^{2r})^{-1}$. Now $p=q+\sum_{l=1}^{n-1}q\xi
_l^{2r}+q\xi _n^{2r}$. The above analysis carries over directly to the
first two terms. The third term leads to the function
$r^+D_{z_n}^{2r}\tilde q$, which is also included in the analysis.

For the last estimate \eqref{eq:full0transestimates}, we use that
$p-\sum_{j<J}p_j$ satisfies estimates like \eqref{eq:plus-estimate}
with $m\to -\infty $ when $J\to\infty $, so that when the term is
added to the finite sum for $j<J$, the estimates cover more of the
desired indices, the larger $J$ is taken.
\end{proof}

\begin{corollary}\label{cor:mutransestimates} Let $m,\mu  \in{\mathbb R}$ and $\tau \in
  \crp$, and let $p(X,\xi )\in C^\tau S^m({\mathbb R}^{n'}\times
{\mathbb R}^n)$. Then $p$ satisfies the (extended)  $\mu $-transmission
condition with respect to $\crnp$ if and only if
 $b(X,\xi )=p(X,\xi )\lambda _+^{-\mu }(\xi )$ satisfies the equivalent
conditions in Theorem {\rm \ref{thm:0transestimates}} with $m$
replaced by $m-\mu $.

One can here replace $\lambda _+^{-\mu }(\xi )$ by any other invertible
symbol $l(X,\xi )\in C^\tau S^{-\mu }({\mathbb R}^n\times
{\mathbb R}^{n'}) $ satisfying the (extended) $(-\mu )$-transmission condition, for
which $1/l(X,\xi )$ is in $ C^\tau S^{\mu }({\mathbb R}^{n'}\times
{\mathbb R}^n) $ satisfying the (extended) $\mu $-transmission condition.
 
\end{corollary}

\begin{proof}
Note that  $b=p\lambda _+^{-\mu}$ is of order $m'=m-\mu $ with
homogeneous terms  $b_j=p_j\lambda _+^{-\mu}$ of degree $m'-j$.
The system of identities  \eqref{eq:roughmu-transm} for $p$ is equivalent with
the analogous system of identities for $b$ with $m,\mu $ replaced by
$m-\mu ,0$.
To check this,
it suffices in view of the homogeneity to evaluate the symbols for $\xi _n=\pm
1$. Since $\lambda
_+^{-\mu }(0,\pm 1)=(\pm i)^{-\mu }=e^{\mp i\mu \pi /2}$, $b_0$ satisfies
\begin{align*}
b_0(X,0,-1)&-e^{i\pi (m-\mu )}b_0(X,0,1 )=p_0(X,0,-1)e^{i\mu \pi  /2}
-e^{i\pi (m-\mu )}p_0(X,0,1)e^{-i\mu \pi/2}\\
&= e^{i\mu \pi
  /2}[p_0(X,0,-1)-e^{i\pi (m-2\mu )}p_0(X,0,1)] =0,
\end{align*}
by the hypothesis on $p_0$. The  lower-order terms and derivatives  are
checked  similarly, and then Theorem \ref{thm:0transestimates}
readily applies.

The last statement is likewise straightforward, in view of Proposition \ref{thm:multipl}.
\end{proof}

As a special case, the even $2a$-order symbols fit into the setup as follows:

\begin{example}\label{ex:evensymbols}
Let $a\in{\mathbb R}$, and let $p(X,\xi )\in C^\tau S^{2a}({\mathbb
  R}^{n'}\stimes{\mathbb R}^n)$.
 Assume that $p$ is {\it even}, cf.\ Example  \ref{ex:transexamples}.  Then $p$
satisfies the global  $a$-transmission condition with respect to $\crnp$
with $m=2a$,  and $b=p\lambda
_+^{-a}$ is in $C^\tau S^{a}({\mathbb
  R}^{n'}\stimes{\mathbb R}^n)$ satisfying the global
$0$-transmission condition with respect to $\crnp$. 
\end{example}

\subsection{Mapping properties over the
  halfspace and smooth sets}

In preparation for showing mapping properties of $\psi $do's $P=\op(p(x,\xi ))$
truncated to the halfspace $\rnp$, we shall consider the 
Poisson-type operators that arise in connection with the
truncation. Recall that  there holds
(as a version of Green's formula): 
$$
D_n^ke^+u=e^+D_n^ku-i{\sum}_{l=0}^{k-1}\gamma ^c_lu(x')\otimes D_n^{k-1-l}\delta (x_n),
$$
cf.\ e.g.\ \cite[(2.2.39)]{G96}. (We are here using the complex trace
operator $\gamma ^c_ju(x')=D_n^ju(x',0)=i^{-j}(\partial_n^ju)|_{x_n=0}$.)  Then when $r^+P$ is applied
to the extension by zero of a normal derivative of a function $u\in
\mathcal S(\crnp)$ --- which will usually have a
jump at $x_n=0$ --- one finds
\begin{align}\nonumber
r^+PD_n^ke^+u&-r^+Pe^+D_n^ku=-ir^+P{\sum}_{l=0}^{k-1}\gamma ^c_lu\otimes
  D_n^{k-1-l}\delta (x_n)=  -i{\sum}_{l=0}^{k-1}K_{p,{k-1-l}}\gamma ^c_lu,\\\text{ with }
  &K_{p,r}v=r^+P(v(x')\otimes D_n^r\delta (x_n)).  \label{eq:inducedPoisson}                   
\end{align}

The application of the $\psi $do $P=\op(p(x,\xi ))$ is understood as
an application by Fourier transformation for each fixed $x$ (considered as a
parameter).
If $p$ is independent of $x_n$, satisfying the 0-transmission
condition with respect to $\crnp$, $K_{p,0}$ is the Poisson operator with
symbol-kernel $\tilde k(x',z_n,\xi ')=r^+\tilde p(x',\xi ',z_n)$, as
defined in \cite{B71,G96,G09}. If $p$ depends on
$x_n$, one can for smooth symbols use an
expansion derived from the Taylor expansion of $p$ in $x_n$ to define the
operator, but in case of limited smoothness in $x$, this is
unsatisfactory. In \cite{A05}, this point is solved for nonsmooth
symbols  by allowing a more
general definition of Poisson operators incorporating the
$x_n$-dependence, and requiring the {\it global} 0-transmission condition for
the involved $\psi $do's.

The estimates in  Theorem \ref{thm:0transestimates} assure precisely
that when $p(x,\xi )\in C^\tau S^m(\rn\times\rn)$ satisfies the global
$0$-transmission condition introduced in  Definition
\ref{def:roughmu-transm}, 
the function $r^+\tilde p$ 
is a Poisson symbol-kernel as in
\cite[Definition~4.1]{A05} with $d=m$, lying in the space  
$C^\tau
S^{m}_{1,0}(\mathbb R^n\stimes \mathbb R^{n-1},\mathcal S(\crp))$
defined there. The general definition of a Poisson operator from a
symbol-kernel $\tilde k(x,\xi ',z_n)\in C^\tau
S^{m}_{1,0}(\mathbb R^n\stimes \mathbb R^{n-1},\mathcal S(\crp))$ is
\begin{equation}\label{eq:PoissonOP}
\OPK(\tilde k(x,\xi ',z_n))v=\int_{\R^{n-1}} e^{ix'\cdot \xi
  '}\tilde k(x,\xi ',z_n)\hat v(\xi ')\,\dd\xi ' \Big|_{z_n=x_n}.
\end{equation}

For the operator in \eqref{eq:inducedPoisson}, the calculation is, when  $v\in\mathcal S(\mathbb R^{n-1})$ and rules for Fourier
transformation of distributions are applied:
\begin{align*}
  K_{p,r}v&=
                          r^+P(v(x')\otimes D_n^r\delta (x_n))
  =r^+\mathcal F^{-1}_{\xi \to x}[p(x,\xi ',\xi _n)\hat v(\xi ')\xi
    _n^r]\\
   & =r^+\mathcal F^{-1}_{\xi '\to x'}[D_{z_n}^r\tilde p(x,\xi ',z_n)\hat v(\xi ')] \big|_{z_n=x_n}\\
&=\mathcal F^{-1}_{\xi '\to x'}[\tilde k_{p,r}(x,\xi ',z_n)\hat v(\xi ')] \big|_{z_n=x_n}=\OPK(\tilde k_{p,r}(x,\xi ',z_n))v,
\end{align*}
where
\begin{equation}\label{eq:inducedPoissonsym}                   
  \tilde k_{p,r}(x,\xi
  ',z_n)=r^+\mathcal F^{-1}_{\xi _n\to z_n}[p(x,\xi ',\xi _n)\xi
  _n^r]=r^+D_{z_n}^r\tilde p(x,\xi ',z_n),
\end{equation}
a Poisson symbol-kernel in $C^\tau
S^{m+r}_{1,0}(\mathbb R^n\stimes \mathbb R^{n-1},\mathcal
S(\crp))$. We have shown:

\begin{lemma}\label{lem:inducedPoisson} Let $p(x,\xi)\in C^\tau 
  S^m(\Rn\stimes\Rn)$ 
  satisfy the global  $0$-transmission condition, let $r\in\mathbb N_0$, and
  define $K_{p,r}$ by {\rm \eqref{eq:inducedPoisson}}.  Then
  $K_{p,r}$ is the Poisson operator $\operatorname{OPK}(\tilde
  k_{p,r})$, where $\tilde k_{p,r}(x,\xi ',z_n)\in              
C^\tau S^{m+r}_{1,0}(\mathbb R^n\stimes \mathbb R^{n-1},\mathcal
  S(\crp))$ is defined by {\rm  \eqref{eq:inducedPoissonsym}}.                   
\end{lemma}

For integer-order nonsmooth $\psi $do's there is a deduction of such 
Poisson operators in \cite[Lemma 5.4]{A05}.

 We shall now show that when $P=\op(p(x,\xi ))$ is a $C^\tau $-smooth  pseudodifferential operator satisfying the extended  $0$-transmission condition
w.r.t.\ $\crnp$, then the truncated version $P_+=r^+Pe^+$ preserves
regularity in $\crnp$ up to orders dominated by $\tau $. There holds:

\begin{theorem}\label{thm:P_+bound} Let  $\tau >0$, $1<q<\infty $ and
  $m\in\mathbb R$. 
When $P=\op(p(x,\xi ))$ with $p\in C^\tau S^m({\mathbb
R}^{n}\stimes {\mathbb R}^n)$ satisfying
the extended  $0$-transmission condition according to Definition {\rm
 \ref{def:roughmu-transm}}, then the truncated version $P_+=r^+Pe^+$ satisfies
\begin{align}\label{eq:P_+bound}
 P_+\colon \ol H^{s+m}_q(\rnp)&\to \ol H^{s}_q(\rnp),\text{ for
                               $|s|<\tau $ with $s+m>-\tfrac1{q'}$. }\\
  \nonumber
  P_+\colon \dot H^{s+m}_q(\crnp)&\to \ol H^{s}_q(\rnp),\text{ for $|s|<\tau $. }
\end{align}
\end{theorem}

\begin{proof}
  Assume in the  following that $|s|<\tau $.
  The second statement is an immediate consequence of Theorem
  \ref{thm:bd-compos}  since $ \dot H^{s+m}_q(\crnp)$ is a closed 
  subspace of $  H^{s+m}_q(\rn)$. 
  When $-\tfrac1{q'}<s+m <{\tfrac1q}$, the first statement also follows immediately, since
  $ \ol H^ {s+m}_q(\rnp)$ identifies with $ \dot H^{s+m}_q(\crnp)$ then. 

To show the first statement for higher $s$, we use a method similar to that of \cite[Lemma 5.6]{A05}.
(The iterative proof in \cite{GH90} does not adapt well, since
commutation of $P$ with $D_j$ introduces a decrease in the H\"older smoothness.)

Assume to begin with that $P$ satisfies the global 0-transmission condition. 
Let us show that the estimate holds for $s+m\in (k-\tfrac1{q'},k+{\tfrac1q} )$,
$k=1,2,\dots$ with $|s|<\tau
$. Fix $k$ and write $s=s_0+k$, $s_0+m\in (-\tfrac1{q'},{\tfrac1q} )$. Setting
$r (\xi )=(\sum_{j=1}^n\xi _j^{2k}+1)^{-1}$, we write $p$ as a sum of
three terms:
\begin{align*}
p(x,\xi )&=p_1(x,\xi )+p_2(x,\xi )+p_3(x,\xi ),\\ p_1(x,\xi )&={\sum}_{j=1}^{n-1}p(x,\xi )r(\xi ) \xi _j^{2k} ,\;
p_2(x,\xi )=p(x,\xi )r(\xi ) \xi _n^{2k},\;p_3(x,\xi )=p(x,\xi )r(\xi ) , 
\end{align*}
defining operators $P_i=\op(p_i(x,\xi ))$ satisfying the  global
$0$-transmission condition.

We can write
$$
r^+P_1e^+={\sum}_{j=1}^{n-1}r^+\op(pr \xi _j^k)D_j^ke^+
={\sum}_{j=1}^{n-1}r^+\op(pr \xi _j^k)e^+D_j^k,
$$
since the tangential derivatives $D_j$ commute with $e^+$. When $u\in
\ol H_q^{m+s_0+k}(\rnp)$, then  $D_j^ku\in
\ol H_q^{m+s_0}(\rnp)\simeq \dot
H_q^{m+s_0}(\crnp)$, so since  $\op(pr \xi _j^k)$ is of
order $m-k$, $r^+\op(pr \xi _j^k)e^+$ maps $ \ol
H_q^{m+s_0}(\rnp)$ to  $\ol H_q^{s_0+k}(\rnp)$. Summing over $j$ we see that $r^+P_1e^+$ has the
desired mapping property.

For $P_2$, we have that
\begin{equation*}
r^+P_2e^+u=r^+\op(pr \xi _n^k)D_n^ke^+u
=r^+\op(pr
\xi _n^k)e^+D_n^ku+ r^+\op(pr \xi _n^k)[D_n^k,e^+]u.
\end{equation*}
The term $r^+\op(pr
\xi _n^k)e^+D_n^ku$ is treated like the terms in $P_1$, defining
an operator with the desired mapping property. The other term
satisfies, by \eqref{eq:inducedPoisson}  applied to $\op(pr\xi _n^k)$,                 \begin{align*}
r^+\op(pr \xi _n^k)[D_n^k,e^+]u&=-ir^+\op(pr \xi
_n^k){\sum}_{l=0}^{k-1}\gamma ^c_lu\otimes D_n^{k-1-l}\delta (x_n)\\
&=
-i{\sum}_{l=0}^{k-1}K_{pr \xi _n^k,k-1-l}\gamma ^c_lu,
\end{align*}
with Poisson operators defined by Lemma \ref{lem:inducedPoisson}. 
Here $K_{pr \xi _n^k,k-1-l}$ is a Poisson operator with 
symbol-kernel in  $C^\tau
S^{m-1-l}_{1,0}({\mathbb R}^{n}\stimes{\mathbb R}^{n-1},\mathcal
S(\crp))$, 
hence
continuous from  $B_q^{s+m-l-{\frac1q}}({\mathbb R}^{n-1})$ to 
$\ol H^s_q(\rnp)$
for $|s|<\tau $, by \cite[Theorem~4.8]{A05}. 
 The trace operator
$\gamma ^c_l$ goes from $\ol H_q^{s+m}(\rnp)$ to $B_q^{s+m-l-{\frac1q}}({\mathbb
R}^{n-1})$ for any $s>-m+{\tfrac1q}$, so 
 $K_{pr \xi _n^k,k-1-1}\gamma ^c_lu \in\ol H^s_q(\rnp)$. Altogether,
 $P_2$ has the asserted mapping property.

The term $P_3$ is easily treated:
Since $P_3$ is of order $m-2k$, it maps $H_q^{m-2k+s_0+k}({\mathbb
  R}^n)=H_q^{m-k+s_0}({\mathbb R}^n)$ to $H_q^{s_0+k}({\mathbb R}^n)$. Here
$\ol H_q^{m-k+s_0}(\rnp)\supset $  $\ol H_q^{m+s_0}(\rnp)
\simeq \dot
H_q^{m+s_0}(\crnp)$, 
so $r^+P_3e^+$ maps $ \ol H_q^{m+s_0}(\rnp)$ to  $\ol H_q^{s_0+k}(\rnp)$,
and a fortiori $ \ol H_q^{m+k+s_0}(\rnp)$ to  $\ol H_q^{s_0+k}(\rnp)$.

We have then obtained \eqref{eq:P_+bound} for all $|s|<\tau $ with $s+m+\tfrac1{q'}\in
\rp\setminus \mathbb N$. The intermediate integer values are included
by interpolation. This ends the proof in the case where $P$ satisfies
the global 0-transmission condition with respect to $\crnp$.

Finally, consider the case where the 0-transmission condition is only
satisfied for $x$ with $0\le x_n<\varepsilon $, some $\varepsilon >0$.
Let $\eta (x_n),\zeta _0(x_n)\in C_0^\infty (\R )$, supported in
$(-\varepsilon ,\varepsilon )$, equal to 1 on a neighborhood of 0, and
such that $\eta =1 $ on a neighborhood of $\supp\zeta _0$. Then
\begin{equation*}
  P=P_1+P_2+P_3, \text{ where } P_1=\eta P,\; P_2= (1-\eta)P\zeta _0,\; P_3=
  (1-\eta)P(1-\zeta _0).
  \end{equation*}
The term $P_1$ satisfies the global 0-transmission condition, hence
has the asserted mapping properties. For the term  $P_3$, $r^+P_3e^+$
acts on $\ol H_q^{s+m}(\rnp)$ as on $\dot H_q^{s+m}(\crnp)$ so it
likewise has them. For the middle term $P_2$, we note that since
$1-\eta $ and $\zeta _0$ have disjoint supports, we can  by
Corollary~\ref{cor:Commutator} for any large $N$ write
\begin{equation*}
  P_2= (1-\eta)P\zeta _0
  =(1-\eta )Q_N,
\end{equation*}
where $Q_N$ has  symbol in $C^\tau S^{m-N}(\rn\times\rn)$.
Now $e^+\ol H_q^{s+m}(\rnp)\subset H_q^{-M}(\rn)$
for some large $M$ for the considered values of $s $, and this will be mapped into
$H_q^{\tau -\delta }(\rn)$ (any $\delta >0$) by $P_2$ when $N$ is chosen large enough in the above
representation, by Theorem \ref{thm:bd-compos}.
\end{proof}

As a corollary we get the mapping property for operators satisfying the extended  $\mu
$-transmission condition:

\begin{corollary}\label{cor:mu-P_+bound} 
Let  $\tau >0$, $1<q<\infty $ and $\mu >-1$. 
Let $P $ 
have symbol $ C^\tau S^m({\mathbb
R}^n\stimes {\mathbb R}^n)$ satisfying the extended  $\mu
$-transmission condition with respect to $\crnp$.
Then 
\begin{equation}\label{eq:mu-P_+bound}
r^+P\colon H_q^{\mu (m+s)}(\crnp)\to \ol H^{s}_q
(\rnp),
\end{equation}
 holds for $|s|<\tau
$. 
\end{corollary}

\begin{proof}
By Corollary \ref{cor:mutransestimates}, the $\psi $do $B=P\Lambda _+^{-\mu }$, of order $m-\mu $, satisfies
the extended  $0$-transmission condition. By Theorem \ref{thm:P_+bound},
\begin{align*}
r^+B\colon e^+ \ol H^{s+m-\mu }_q(\rnp)&\to \ol H^{s}_q(\rnp),\text{ for $|s|<\tau $  with $s+m-\mu >-\tfrac1{q'}$,}\\
r^+B\colon  \dot H^{s+m-\mu }_q(\crnp)&\to \ol H^{s}_q(\rnp),\text{ for $|s|<\tau $.}\\
\end{align*}
Let $|s|<\tau $. Recalling that $H_q^{\mu (m+s)}(\crnp)= \Lambda _+^{-\mu }e^+ \ol
H^{s+m-\mu }_q(\rnp)$ for $s+m-\mu >-\tfrac1{q'}$,
we infer that
\begin{equation*}
r^+P =r^+B\Lambda _+^\mu \colon \Lambda _+^{-\mu }e^+\ol
H_q^{m-\mu +s}(\rnp)=H_q^{\mu (m+s)}(\crnp)\to \ol H^{s}_q
(\rnp),
\end{equation*}
for $|s|<\tau $ with $s+m-\mu >-\tfrac1{q'}$. For $s+m-\mu < {\tfrac1q}$, we use
that $H_q^{\mu (m+s)}(\crnp)=\dot H_q^{m+s}(\crnp)$ by definition.
\end{proof}

There is in particular a consequence for operators as in Example \ref{ex:evensymbols}:

\begin{corollary}\label{cor:evenP_+bound}
Let  $\tau >0$ and $1<q<\infty $. 
When $P $ is even of order $2a>0$ as in Example {\rm
 \ref{ex:evensymbols}}, 
then 
\begin{equation}\label{eq:evenP_+bound}
r^+P\colon H_q^{a(2a+s)}(\crnp)\to \ol H^{s}_q
(\rnp),\text{ for $|s|<\tau
$.} 
\end{equation}
\end{corollary}

The result can be generalized to bounded smooth domains by tools that are
already available in the literature, namely the result of Abels and
Jim\'enez  \cite{AJ18} that $C^\tau S^m(\rn\times\rn)$ is preserved
under $C^\infty $-transformations, and the localization
explained e.g.\ in \cite{G18}.

\begin{theorem}\label{thm:P_+boundSmoothSets}
Let  $\tau >0$, $1<q<\infty $ and $\mu >-1$. 
Let $\Omega \subset\rn$ be a bounded $C^\infty $-domain,
and let $P=\op(p(x,\xi ))$ with $p\in C^\tau S^{m}({\mathbb
R}^{n}\stimes {\mathbb R}^n)$ satisfying the extended $\mu
$-transmission condition with respect to $\comega$.
Then the restricted operator $r^+P$
has the mapping property:
\begin{equation}\label{eq:P_+bound}
r^+ P\colon  H^{\mu  (m+s)}_q(\comega )\to \ol H^{s}_q(\Omega),\text{
  for }|s|<\tau .
\end{equation}
\end{theorem} 

\begin{proof}
  For $s+m-\mu <\tfrac1q$, the statement follows immediately from Theorem \ref{thm:bd-compos}.
  For $s+m-\mu > -\tfrac1{q'}$, we use local coordinates and a subordinated
  partition of unity, as in \cite[Remark~4.3ff.]{G18}   It is described
  there how $\comega$ has a cover by bounded open sets $U_i$ with
  diffeomorphisms $\kappa _i\colon U_i\to V_i$ such that
  $U_i\cap\Omega  $ is mapped to $V_i\cap \rnp$, $i=0,\dots,I_1$, and
  there is a {\it subordinated} partition of unity $\{\varrho
  _j\}_{j=0,\dots, J}$ where for each pair $j,k$ there is an index
  $i=i(j,k)$ such that $\varrho _j,\varrho _k\in C_0^\infty
  (U_i)$. Choose also $\zeta _j$  in $C_0^\infty (U_i)$
  satisfying $\zeta _j\varrho _j=\varrho _j$. Let $u\in H^{\mu 
    (m+s)}_q(\comega )$, and let $u_k=\varrho _ku=\zeta _ku_k$, then 
 $Pu=\sum_{j,k}\varrho _jP\zeta _ku_k$. Here
 the operators $P_{jk}=\varrho _jP\zeta
  _k$ carry over via $\kappa _i$ to operators $\underline P_{jk}$ acting over $V_i$ with
   symbols in $C^\tau S^{2a}(\rn\times\rn)$ in view of Proposition
 \ref{prop:Commutator} and \cite{AJ18}, satisfying the $\mu
 $-transmission condition with respect to $\crnp$, and $u_k$ carries over to $\underline
  u_k\in  H^{\mu  (m+s)}_q(\crnp ) $. Now we apply Corollary
  \ref{cor:mu-P_+bound} to each $\underline P_{jk}\underline u_k$,
  carry the contributions back to $\comega$, and sum over $j$ and $k$
  to end the proof.
  \end{proof}

  A similar result holds for $\Omega =\R^n_\gamma $ when $\gamma \in
  C^\infty _b(\R^{n-1})$.

\subsection{Mapping properties of $(x,y)$-form operators over the halfspace}

As a preparation for the treatment of operators on nonsmooth sets
we consider operators with symbols in $(x,y)$-form on $\Rn_+$.
We begin with an observation on remainders:

\begin{cor}\label{cor:BoundednessRemainderNew} Let $r_\alpha$ and $m<\tau$ be as in
   Corollary~{\rm\ref{cor:xyFormReduction}}, with
  $l<\tau$, and let $\mu\geq 0$. Then   
  
  \begin{equation*}
    r^+ \op(D_\xi^\alpha r_\alpha(x,y,\xi ))\colon  H^{\mu ((m-l+s)_+)}_q(\ol{\R}_+^n)\to \ol{H}_q^s(\Rn_+) 
  \end{equation*}
  is bounded if $0\leq s<\tau $ and ${s+m
    <\tau}$, and $s+m<\mu + l+\frac1q$. 
  Moreover, there is some $k\in\N$ and $C_{s,m,\mu}>0$ independent of $a$ such that
  \begin{equation*}
    \|\op(D_\xi^\alpha r_\alpha(x,y,\xi ))\|_{\mathcal{L}( H^{\mu ((m-l+s)_+)}_q(\ol{\R}_+^n), \ol{H}_q^s(\Rn_+))}\leq C_{s,m,\mu }|a|_{k, C^{\tau} S^{m}_{1,0}}.
  \end{equation*}
\end{cor}
\begin{proof}
  \noindent
  We use that
  \begin{equation*}
    H^{\mu({(m-l+s)_+})}_q(\ol{\R}_+^n) =\dot{H}^{(s+m-l)_+}_q(\ol{\R}^n_+)
  \end{equation*}
  since ${(m-l+s)_+} < \mu+ \frac1q$. Hence
  \begin{equation*}
    r^+ \op(D_\xi^\alpha r_\alpha(x,y,\xi ))  \colon  H^{\mu((m-l+s)_+)}_q(\ol{\R}_+^n)  \to \ol{H}_q^s(\Rn_+),
  \end{equation*}
 with the mentioned estimates,  because of Corollary~\ref{cor:Bddnessr}.
\end{proof}

We  now show a mapping property for restricted $(x,y)$-form operators,
with limitations on both $\mu ,m$ and $s$:

\begin{thm}\label{thm:BddnessTruncatedPsDOs}
  Let $1<q<\infty $,  $\tau >0$, $0\leq \mu \leq m<\tau$ and $\mu-m\leq s< \tau-m$, 
and let $a\in C^\tau S^m(\R^{2n}\times \Rn)$, satisfying the global
$\mu$-transmission condition with respect to $\ol{\R}_+^n$.
Then
  \begin{equation*}
    r^+ \op(a(x,y,\xi ))\colon H^{\mu(m+s)}_q(\ol{\R}_+^n)\to \ol{H}_q^s(\Rn_+) 
  \end{equation*}
  is bounded. 
  Moreover, there is some $N\in\N$ and $C_{s,m,\mu,q}>0$ such that
  \begin{equation}\label{eq:OpaEstim}
    \|r^+\op(a(x,y,\xi ))\|_{\mathcal{L}(H^{\mu(m+s)}_q(\ol{\R}_+^n), \ol{H}_q^s(\Rn_+))}\leq C_{s,m,\mu ,q}|a|_{N, C^{\tau} S^{m}_{1,0}}.
  \end{equation}
\end{thm}

\begin{proof}
  We will prove the statement in the cases $s=\mu-m$ and $s=\tau-m-\eps$ for $\eps>0$ sufficiently small. Then the general statement follows by complex interpolation since $s+m\geq \mu>\mu-\tfrac1{q'}$, cf.\ Remark~\ref{rem:Interpolation}.

  \noindent
  {\bf Case $s=\mu-m$:} In this case we have
    \begin{equation*}
    H^{\mu(m+s)}_q(\ol{\R}_+^n)= H^{\mu(\mu)}_q(\ol{\R}_+^n) =\dot{H}^{\mu}_q(\ol{\R}^n_+)\subset H_q^\mu(\Rn)
  \end{equation*}
  and
  \begin{equation*}
   r^+ \op(p_\alpha (x,\xi))\colon \dot{H}^{\mu}_q(\ol{\R}^n_+)  \to \ol{H}_q^{\mu-m}(\Rn_+)
 \end{equation*}
 by Theorem~\ref{thm:Bddness}, using that $|\mu-m|=m-\mu\leq m<\tau$ and $\mu< \tau$.\\[1ex]
  \noindent
  {\bf Case $s=\tau-m-\eps$, $\eps>0$ sufficiently small:} We can assume $\tau-m\not\in\N$ without loss of generality. (Otherwise replace $\tau$ by some $\tau'\in (s+m,\tau)$.) Then $k:=[s]=[\tau-m]$, if $\eps>0$ is sufficiently small.

  First we consider the case $k=0$. Then $0\leq s=\tau-m-\eps <\tau
  -m<1$. We use again the expansion in
  Corollary~\ref{cor:xyFormReduction} with $l=[m]$.
  Here Theorem~\ref{thm:Bddness2} yields:
\begin{equation*}
    r^+ \op( D_\xi^\alpha r_\alpha (x,y,\xi))\colon H^{\mu(s+m)}_q(\ol{\R}^n_+)\hookrightarrow e^+L_q(\Rn_+)\to \ol{H}^s_q(\Rn_+)
  \end{equation*}
  because of $D_\xi^\alpha r_\alpha \in C^{\tau-[m]}S^{m-[m]}_{1,0}(\R^{2n}\times \Rn)$ and $s<\tau-[m]-(m-[m])=\tau-m$. Moreover,
\begin{equation*}
    r^+ \op( p_\alpha (x,\xi))\colon H^{\mu(s+m-|\alpha|)}_q(\ol{\R}^n_+)\to \ol{H}^s_q(\Rn_+)
  \end{equation*}
  by Corollary~\ref{cor:mu-P_+bound}. This shows the case $k=0$.

  Next we consider the case $k\geq 1$.
  We shall use that $s=s'+k$ with $s'\in [0,1)$ and
  \begin{equation*}
    v \in \ol{H}^s_q(\Rn_+) \iff  \partial_x^\beta v \in
    \ol{H}^{s'}_q(\Rn_+)\quad \text{for all }|\beta|= k\text{ and }\beta=0, 
  \end{equation*}
  with corresponding norm equivalences. Let $|\beta |=k$. The
  composition of the differential operator $\partial_x^\beta $ with
  $r^+\OP(a(x,y,\xi ))$ is a finite sum
  \begin{equation*}
    \partial_x^\beta r^+ \op(a(x,y,\xi )) = \sum_{0\leq \gamma \leq \beta} \tbinom{\beta}{\gamma}r^+ \op(\partial_x^{\gamma}a(x,y,\xi )(i\xi)^{\beta-\gamma} )=\sum_{j=0}^kr^+\OP(a_{\beta ,j}(x,y,\xi )), 
  \end{equation*}
  where $a_{\beta ,j}(x,y,\xi )\in C^{\tau
    -j}S^{m+k-j}(\R^{2n}\times\Rn)$. Here the result for the case $k=0$ yields that
  \begin{equation*}
    r^+ \op(a_{\beta ,j} )\colon H^{\mu(m+k+s'-j)}_q(\ol{\R}_+^n)\to \ol{H}_q^{s'}(\Rn_+) 
  \end{equation*}
  since $s'=s-k < \tau-j-(m+k-j)=\tau-m-k$ and $m+k-j<\tau-j$ due to $k=[\tau-m]<\tau-m$.
This treats the case $|\beta |=k$. 
For the case $|\beta |=0$ we apply
the first case directly in a similar way.

  Finally, the last statement is a consequence of the closed graph
  theorem. More precisely, let
  \begin{equation*}
    C^\tau S^m_{\mu,tr}(\R^{2n}\times \Rn):=\{a\in C^\tau
    S^m(\R^{2n}\times \Rn)\mid a \text{ satisfies the global }\mu\text{-transmission condition}\}
  \end{equation*}
  and 
  consider the mapping 
  \begin{alignat*}{1}
    \op_+\colon &C^\tau S^m_{\mu,tr}(\R^{2n}\times \Rn)\to \mathcal{L}( H^{\mu(m+s)}_q(\ol{\R}_+^n), \ol{H}^s_q(\Rn_+))\colon a\mapsto  r^+ \op(a(x,y,\xi )).
  \end{alignat*}
 Note that $C^\tau S^m_{\mu,tr}(\R^{2n}\times \Rn)$ is a closed subspace of the Fr\'echet space $C^\tau S^m_{1,0}(\R^{2n}\times\Rn)$ and therefore a Fr\'echet space.
  If $(a_k)_{k\in\N}\subset C^\tau S^m_{\mu,tr}(\R^{2n}\times \Rn)$ is such that
  \begin{alignat*}{2}
    a_k&\to_{k\to\infty} a&\qquad& \text{in } C^\tau S^m_{1,0;\mu,tr}(\R^{2n}\times \Rn), \\
    r^+\op(a_k(x,y,\xi))&\to_{k\to \infty} A&\qquad& \text{in } \mathcal{L}( H^{\mu(m+s)}_q(\ol{\R}_+^n), \ol{H}^s_q(\Rn_+)),
  \end{alignat*}
  then for any $u\in \mathcal E_\mu\cap \mathcal E'$
  and a suitable subsequence
  \begin{equation*}
    r^+\op(a_k(x,y,\xi))u(x) \to_{k\to \infty} Au(x)\qquad \text{for almost every }x\in\Rn_+. 
  \end{equation*}
  Moreover, using the representation in Theorem~\ref{thm:OscIntegrals}
     of $\op(a(x,y,\xi))u(x)$ with $a$ replaced by $a_k$ it is easy to observe that
  \begin{equation*}
    r^+\op(a_k(x,y,\xi))u(x)\to_{k\to\infty}r^+\op(a(x,y,\xi))u(x)\qquad \text{for all }x\in \R^n_+.
  \end{equation*}
  Hence $Au(x)=r^+\op(a(x,y,\xi))u(x)$ for almost all $x\in\R^n_+$ and
  all $u\in \mathcal E_\mu\cap \mathcal E'$. This shows the closedness of $\op_+$ since $\mathcal E_\mu\cap \mathcal E'$ is dense in $H^{\mu(m+s)}_q(\ol{\R}_+^n)$. 
Hence $\op_+$ is continuous and therefore bounded, which yields the last statement.
\end{proof}

Also cases where $\mu $ and $m$ are in $(-1,0)$ can be included, with a
loss of H\"older-regularity by one step:

\begin{cor}\label{cor:BddnessHalfspace}
    Let  $\tau >0$, and $m\ge \mu >-1$, 
and let $a\in C^\tau S^m(\R^{2n}\times \Rn)$ satisfy the global
$\mu$-transmission condition with respect to $\ol{\R}_+^n$. If  $ m<\tau-1$ and $\mu -m \leq s< \tau-m-1$, then
  \begin{equation*}
    r^+ \op(a(x,y,\xi ))\colon H^{\mu(m+s)}_q(\ol{\R}_+^n)\to \ol{H}_q^s(\Rn_+) 
  \end{equation*}
  is bounded.
   Moreover, there is some $N\in\N$ and $C_{s,m,\mu}>0$ such that
  \begin{equation}\label{eq:OpaEstim}
    \|r^+\op(a(x,y,\xi ))\|_{\mathcal{L}(H^{\mu(m+s)}_q(\ol{\R}_+^n), \ol{H}_q^s(\Rn_+))}\leq C_{s,m,\mu }|a|_{N, C^{\tau} S^{m}_{1,0}}.
  \end{equation}
\end{cor}
\begin{proof}
  We use that by definition,
  \begin{equation*}
   H_q^{\mu(m+s)}(\ol{\R}_+^n) = \Xi_+^1 \Xi_+^{-(\mu+1)}e^+ \ol{H}_q^{m+s-\mu}(\Rn_+)=\Xi_+^1 H_q^{(\mu+1)(m+s+1)}(\ol{\R}_+^n),
 \end{equation*}
 where $\Xi_+^1= \partial_{x_n}+ \op(\weight{\xi'})$. Here $\op(\weight{\xi'})\colon H_q^{(\mu+1)(m+s+1)}(\ol{\R}_+^n)\to H_q^{(\mu+1)(m+s)}(\ol{\R}_+^n)$ since the operator commutes with $e^+$. Thus for every $u\in  H_q^{\mu(m+s)}(\ol{\R}_+^n)$ there are $v\in H_q^{(\mu+1)(m+s+1)}(\ol{\R}_+^n)$ and $w\in H_q^{(\mu+1)(m+s)}(\ol{\R}_+^n)$ (depending continuously on $u$) such that $u= \partial_{x_n} v+w$. Moreover,
 \begin{equation*}
   r^+ \op(a(x,y,\xi ))\partial_{x_n} v=    r^+ \op(a(x,y,\xi )i\xi_n)v - r^+ \op(\partial_{y_n}a(x,y,\xi ))v, 
 \end{equation*}
 where $a(x,y,\xi )i\xi_n\in C^\tau S^{m+1}(\R^{2n}\times \Rn)$ and
 $\partial_{y_n} a(x,y,\xi)\in C^{\tau-1} S^{m}(\R^{2n}\times \Rn)$
 satisfy the global $(\mu+1)$-transmission condition. Considering
 $\partial_{y_n}a(x,y,\xi )$ and $a(x,y,\xi )$ as symbols of order
 $m+1$, we get from 
 Theorem~\ref{thm:BddnessTruncatedPsDOs} that all three maps
 $  r^+ \op(a(x,y,\xi )i\xi_n)$, $ r^+ \op(\partial_{y_n}a(x,y,\xi ))$
 and $   r^+ \op(a(x,y,\xi ))$ are  bounded
from $H^{(\mu+1)(m+1+s)}_q(\ol{\R}_+^n)$ to $\ol{H}_q^s(\Rn_+)$, 
when $\tau >0$, $0\le \mu +1\le m+1< \tau $ and $(\mu +1)-(m+1)\le
s<\tau -(m+1)$. In view of the assumption $-1<\mu \le m$, the latter conditions
reduce to $  m<\tau -1$, $\mu -m\le s<\tau -m-1$. Then when they hold,
 \begin{equation*}
r^+ \op(a(x,y,\xi ))u= r^+ \op(a(x,y,\xi )i\xi_n)v - r^+ \op(\partial_{y_n}a(x,y,\xi ))v + r^+ \op(a(x,y,\xi ))w
\end{equation*}
belongs to $ \ol{H}_q^s(\Rn_+)$, and the corresponding map is bounded.

The last statement follows likewise from Theorem~\ref{thm:BddnessTruncatedPsDOs}.
\end{proof}

The results can be generalized to symbols satisfying the extended $\mu
$-transmission condition.

\section{The Homogeneous Dirichlet Problem on Nonsmooth Domains}\label{sec:Regularity}

We shall now apply the analysis to the homogeneous Dirichlet problem
for those $\psi $do's that are close generalizations of the fractional
Laplacian, namely operators $P$ of order $2a$ with an {\it even}
symbol. As already noted, they satisfy the global $a$-transmission condition
with respect to any choice of normal coordinate. The homogeneous
Dirichlet problem is, for strongly elliptic operators,
\begin{equation}\label{eq:Dirpb1}
Pu=f\text{ on }\Omega ,\quad \supp u\subset \comega,
\end{equation}
where the solution $u$ is sought in $H^a(\rn)$, and it is known in the
smooth case \cite{G15} that it is Fredholm solvable for $f\in \ol
H^s_q(\Omega )$, with $ u\in H_q^{a(s+2a)}(\comega)$, when
$s>-a-\frac1{q'}$. Our present aim is to
extend the regularity result to symbols $p$ and open sets $\Omega $  with $C^{1+\tau}$-boundary, for
$s$ as large as possible relative to the H\"older exponents.

\subsection{Coordinate changes at a boundary, boundedness over nonsmooth domains}\label{subsec: CoordinateTrafo}

As in Section \ref{subsec: NonsmoothSets},   $\Rn_\gamma= \{x\in\R^n\mid
x_n>\gamma(x')\}$ for some $\gamma\in C^{1+\tau}(\R^{n-1})$ with
$\tau >0$, and $F_\gamma\colon \R^n\to \R^n$ is a
$C^{1+\tau}$-diffeomorphism such that $F_\gamma(\Rn_\gamma)=
\Rn_+$. We take $F_\gamma(x)= (x',x_n-\gamma(x'))$ for all $x\in\R^n$,
where $x'=(x_1,\ldots, x_{n-1})$. Moreover, let $p\in C^\tau S^{2a}(\R^n\times \R^n)$ be even and $\tau> 2a$ and let $P_\gamma $ be the
transformed operator:
\begin{equation}\label{eq:Pgamma}
  (P_\gamma u)(x) = (P (u\circ F^{-1}_\gamma))(F_\gamma(x))= (F^{\ast}_\gamma P F^{\ast,-1}_\gamma u)(x)\quad \text{for all }u\in \SD(\Rn),
\end{equation}
and let $\|\gamma \|_{C^1(\R^n)}\leq r_0 $ for some $r_0\in (0,1]$. 
We assume for simplicity that $r_0$ is so small that
\begin{equation}\label{eq:CondFgamma}
  \sup_{x\in\R^n}|\nabla F_\gamma(x)-I|\leq \tfrac12.
\end{equation}
Then one obtains by the results of Section~\ref{sec:CoordinateChange} that
for all $u\in \SD(\R^n)$ and $x\in\R^n$
\begin{alignat}{1}
  P_\gamma u(x) \label{eq:CoordTrafo1}
  &=\Os\int_{\R^n}\int_{\R^n} e^{i(x-y)\cdot \xi } q_\gamma(x,y,\xi) u(y) \sd y\dd \xi,
\end{alignat}
where
\begin{alignat}{1} \label{eq:CoordTrafo2}
  q_\gamma(x,y,\xi)&= p(F_\gamma(x),A_\gamma(x,y)^{-1,T}\xi)|\det A_\gamma(x,y)|^{-1}|\det \nabla_y F_\gamma(y)|,\\\nonumber
  A_\gamma(x,y) &= \int_0^1 \nabla_x F_\gamma (x+t(y-x))\sd t,
\end{alignat}
for all $x,y,\xi\in\R^n$.  
Here $q_\gamma (x,y,\xi )\in  C^\tau S^m_{1,0}(\R^{2n}\times \Rn)$. 
Moreover, for every $0<\tau'<\tau$ and $k\in \N_0$ there is some $C_k$ independent of
$\gamma$ and $p$ such that 
\begin{alignat}{1}\label{eq:qgammaEstim}
  |q_\gamma -p|_{k, C^{\tau'} S^{2a}_{1,0}(\R^{2n}\times\Rn)} &\leq C_k \|\gamma\|_{C^{1+\tau}(\R^{n-1})}^{\min(\tau-\tau',1)}|p|_{k+1, C^\tau S^{2a}_{1,0}(\R^{2n}\times\Rn)}
\end{alignat}
for all $\gamma\in C^{1+\tau}(\R^{n-1}),\|\gamma\|_{C^{1+\tau}(\R^{n-1})}\leq r_0$,
since
\begin{alignat*}{1}
  \|F_\gamma-\operatorname{id}\|_{C^{\tau}(\R^n)}&\leq C\|\gamma\|_{C^{1+\tau}(\R^{n-1})},\quad 
  \|A_\gamma-I\|_{C^{\tau}(\Rn\times \Rn)}\leq
  C\|\gamma\|_{C^{1+\tau}(\R^{n-1})}.
\end{alignat*}

In order to apply the results to the nonlocal equations on
$\Rn_\gamma$, we have to extend \eqref{eq:Pgamma} and \eqref{eq:CoordTrafo1} to $u\in H^{a(s+2a)}_q(\ol{\R}_+^n)$. First of all, $\op(q_\gamma)$ 
extends to a bounded linear operator from $H_q^a(\Rn)$ to $H_q^{-a}(\Rn)$ because of Theorem~\ref{thm:Bddness}, 
due to $0<a<2a<\tau$. Moreover, $F_\gamma^\ast$ and $F_\gamma^{\ast,-1}$ map $H^a_q(\Rn)$ to itself since $0<a<\tau$. Because of $\det DF_\gamma(x)\equiv 1$,
\begin{equation*}
  \int_{\Rn} (F_\gamma^\ast u)(x) v(x)\,dx =   \int_{\Rn} u(x) (F_\gamma^{\ast,-1}v)(x)\,dx  \qquad \text{for all }u,v\in \SD(\Rn).
\end{equation*}
 Hence $F_\gamma^\ast$ and $F_\gamma^{\ast,-1}$ map $H^{-a}_q(\Rn)$ to itself as well. Therefore we obtain
\begin{align}\label{eq:DefnPGamma}
  P_\gamma u &= F_\gamma^{\ast} P F_\gamma^{\ast,-1} u = \op(q_\gamma(x,y,\xi))u
\end{align}
for all $u\in H^a_q(\Rn)$. In particular, we obtain the identity for all $u\in H^{a(a)}_q(\ol{\R}^n_+)= \dot{H}^{a}_q(\ol{\R}^n_+)$, and conclude
\begin{align*}
 r^+F^{\ast}_\gamma P F^{\ast,-1}_\gamma u &= r^+\op(q_\gamma(x,y,\xi))u
\end{align*}
for all $u\in H^{a(a)}_q(\ol{\R}^n_+)$. Note that $H^{a(s+2a)}_q(\ol{\R}^n_+)\subset H^{a(a)}_q(\ol{\R}^n_+)$ for any $s\geq -a$.

In the case of a classical even symbol $p$ this leads to:
\begin{theorem}\label{thm:TransformedPsDOs}
  Let $p\in C^\tau S^{2a}(\Rn\times \Rn)$ be even, where $0<a<1$, let
  $\gamma\in C^{1+\tau}(\R^{n-1})$ and $N<\tau$, let $q_\gamma $ be
  the transformed symbol \eqref{eq:CoordTrafo2}, and let $P_\gamma =  \op(q_\gamma(x,y,\xi)) $. 
  Then
  $q_{\gamma }\in C^{\tau}S^{2a}(\R^{2n}\times \R^n)$
  satisfies the global $a$-transmission condition. 
  Moreover, with $1<q<\infty $,
\begin{equation*}
 r^+P_\gamma \equiv  r^+\op(q_\gamma(x,y,\xi))\colon H^{a(2a+s)}_q(\ol{\R}_+^n)\to \ol{H}_q^s(\Rn_+)
 \end{equation*}
 is bounded for any $s\in\R$ such that $-a\leq s<\tau-2a$. Furthermore, for any $r_0>0$ there is some $C_{s,r_0,q}>0$, $\theta>0$ and $k\in\N_0$ independent of $p$ and $\gamma$ such that
\begin{equation}\label{eq:TransformEstim}
  \|r^+P_\gamma - r^+ \op(p(x,\xi))\|_{\mathcal{L}(H^{\mu(2a+s)}_p(\ol{\R}_+^n), \ol{H}_p^s(\Rn_+))}\leq C_{s,r_0,q}\|\gamma\|^\theta_{C^{1+\tau}(\R^{n-1})}|p|_{k,C^\tau S^{2a}(\Rn\times\Rn)}
\end{equation}
provided that $\|\gamma\|_{C^{1+\tau}}\leq r_0$.

\end{theorem}
\begin{proof}
  Since $p$ is even, it
  is easy to observe that $q_\gamma$ is even as well. Therefore
  $q_\gamma\in C^\tau S^{2a}(\R^{2n}\times \Rn)$ satisfies the global
  $a$-transmission condition
  and we can apply Theorem~\ref{thm:BddnessTruncatedPsDOs} to
  $q_\gamma$. This implies the statement for the mapping properties
  of $r^+\op(q_\gamma(x,y,\xi))$. To show \eqref{eq:TransformEstim} one chooses some $\tau'\in (0,\tau)$ sufficiently close to $\tau$, $\theta=\min (\tau-\tau',1)$ and applies \eqref{eq:OpaEstim} for
  $r^+P_\gamma - r^+ \op(p(x,\xi))=
  r^+\op(q_\gamma(x,y,\xi)-p(x,\xi))$ and with $\tau'$ instead of $\tau$. Moreover, one uses \eqref{eq:qgammaEstim}.   
\end{proof}

\begin{cor}\label{cor:NonsmoothMap2} Let $0<a<1$ and $\tau>2a$, and let $p\in C^\tau
S^{2a}(\Rn\times \Rn)$ be even. Then $P=\op(p)$ maps
\begin{equation*} r^+P\colon 
  H^{a(s+2a)}_q(\ol{\R}^n_\gamma ) \to \ol H^s_q(\Rn_\gamma ),
\end{equation*} continuously for $-a\le s<\tau -2a$.
\end{cor}
\begin{proof}
  Follows directly from Theorem~\ref{thm:TransformedPsDOs}, since $ F_\gamma^\ast \big(
  H^{a(s+2a)}_q(\ol{\R}^n_+ )\big)=  H^{a(s+2a)}_q(\ol{\R}^n_\gamma )$,
 using that $F_\gamma^{-1,\ast}= F_{-\gamma}^\ast$
  maps $H^s_q(\Rn)$ to itself (since $|s|<1+\tau$).
\end{proof}

\begin{cor}\label{cor:NonsmoothMap3} Let $0<a<1$ and $\tau>2a$, and
  let $p\in C^\tau
S^{2a-1}(\Rn\times \Rn)$ be odd. Then $P=\op(p)$ maps
\begin{equation*}
  r^+P\colon 
  H^{a(\max(a,s+2a-1))}_q(\ol{\R}^n_\gamma )
  \to \ol H^s_q(\Rn_\gamma ),
\end{equation*}
continuously for $-a\le s<\tau -2a$.
\end{cor}
\begin{proof}
  We first consider the case that $\max(a,s+2a-1)=a$, i.e., $s\leq 1-a$. Then
  \begin{equation*}
    H^{a(\max(a,s+2a-1))}_q(\ol{\R}^n_\gamma )= \dot{H}^a_q(\ol{\R}^n_\gamma)
\end{equation*}
and
\begin{equation*}
  r^+ P\colon  \dot{H}^a_q(\ol{\R}^n_\gamma)\to \ol{H}^s_q(\Rn_\gamma)
\end{equation*}
because of Theorem~\ref{thm:Bddness}, $-\tau <-a\leq s<\tau-2a<\tau$, $p\in C^\tau S^{2a-1}(\Rn\times \Rn)\subset C^\tau S^{a-s}_{1,0}(\Rn\times \Rn)$, and $|a-s|<\tau$.

Next we consider the case $s>1-a$. Then $\max(a,s+2a-1)=s+2a-1$ and $\tau>s+2a>1+a$. We use that
\begin{equation*}
  f\in \ol{H}^s_q(\Rn_\gamma) \quad \iff \quad \partial_x^\alpha f \in \ol{H}^{s-1}_q(\Rn_\gamma)\quad \text{for all }|\alpha|\leq 1 ,
\end{equation*}
and write $\partial_x^\alpha P= \op (a_\alpha (x,y,\xi))$ for some even $a_\alpha \in C^{\tau-1} S^{2a}(\Rn\times \Rn)$. In the case $\tau>1+2a$, Corollary~\ref{cor:NonsmoothMap2} implies that
\begin{equation*}
  r^+ \op (a_\alpha (x,y,\xi))\colon 
  H^{a(s+2a-1)}_q(\ol{\R}^n_\gamma ) \to \ol H^{s-1}_q(\Rn_\gamma ),
\end{equation*}
due to $2a<\tau-1$ and $-a<s-1<\tau-1-2a$.
In the case $\tau\leq 1+2a$ we use that $a_\alpha(x,y,\xi)= q_\gamma(x,y,\xi)(i\xi)^\alpha+ \partial_x^\alpha q_\gamma(x,y,\xi)$ if $|\alpha|=1$, where $q_\gamma(x,y,\xi)(i\xi)^\alpha\in C^\tau S^{2a}(\Rn\times \Rn)$ is even and $\partial_x^\alpha q_\gamma(x,y,\xi)\in C^{\tau-1} S^{2a-1}(\Rn\times \Rn)$. Again using
Corollary~\ref{cor:NonsmoothMap2} implies that
\begin{equation*}
  r^+ \op (q_\gamma (x,y,\xi)(i\xi)^\alpha )\colon
  H^{a(s+2a-1)}_q(\ol{\R}^n_+) \to \ol H^{s-1}_q(\Rn_\gamma ),
\end{equation*}
due to $2a<\tau$ and $-a<s-1<\tau-2a$. Moreover,
\begin{equation*}
  r^+ \op (\partial_x^\alpha q_\gamma (x,y,\xi) )\colon
H^{a(a)}_q(\ol{\R}^n_+)=\dot{H}^a_q(\Rn_\gamma) \to \ol H^{s-1}_q(\Rn_\gamma )
\end{equation*}
because of $-a<s-1<\tau-1$, $a<\tau-1 $ and $a-s+1\geq 2a-1$ due to $s<\tau-2a\leq 1$. Altogether this yields the desired mapping properties.
\end{proof}

  For general domains we obtain:
 \begin{theorem}\label{thm:MappingGeneralDomain} Let $0<a<1$ and $\tau >2a$, let $\Omega $ be a bounded $C^{1+\tau
 }$-domain, and let $P=\op(p(x,\xi ))$ where $p\in C^\tau
 S^{2a}(\rn\stimes\rn)$ is even. Then 
 \begin{equation}
   \label{eq:6.1}
   r^+ P\colon H_q^{a (s+2a)}(\comega)\to \ol H_q^{s}(\Omega )    
 \end{equation}
holds for $-a\le s<\tau -2a$, $1<q<\infty $.
\end{theorem} 

\begin{proof}
Let $u\in H_q^{{a} (s+2a)}(\comega )$ and let $x_0,U,\gamma ,\varphi $
be as in Definition~\ref{def:roughTransmSpaces} $2^\circ$. Let $\psi \in C_0^\infty
(U)$ satisfy $\psi \varphi =\varphi $. Let $U'$ be the interior of the
set where $\varphi =1$. In the translated and
rotated situation, where the objects will be marked with an underline,
we then have that $\underline\varphi \underline u\in  H_q^{{a}
(s+2a)}(\ol{\R}^n_\gamma  )$; and then by Corollary~\ref{cor:NonsmoothMap2},
$r^+\underline P(\underline\varphi \underline u)\in
\ol H_q^s(\R^n_\gamma )$, and also $r^+\underline \psi \underline
P(\underline \varphi \underline u)$ lies there. Thus in the original
position, $r^+\psi P\varphi u\in \ol H_q^s(\Omega ) $.

By Corollary~\ref{cor:Commutator}, $(1-\psi )P\varphi u=(1-\psi )\op(q_N)u$ with
$q_N\in C^\tau S^{2a-N}_{1,0}(\Rn\times \Rn)$ for arbitrarily large
$N$. Take $N\ge\tau +2a$. By  Theorem~\ref{thm:bd-compos}, $\op(q_N)$ maps $H_q^{s+2a-N}(\rn)$ into
$H_q^s(\rn)$ for $|s|<\tau $. When $s\in [-a,\tau -2a)$,  $H_q^{{a}
(s+2a)}(\comega )\subset H_q^{{a}
(a)}(\comega )= \dot H_q^a(\comega)$ is thus mapped by $r^+(1-\psi
)\op(q_N)$ into $\ol H_q^s(\Omega )$.

Altogether, we see that $r^+P\varphi u\in \ol H_q^s(\Omega ) $.

There is a finite set of points $\{x_{0,i}\}_{i=1,\dots,I}$ such that
$\bigcup_iU'_i\supset \partial\Omega $ holds for the associated data $\{U_i,\gamma _i,\varphi _i,U'_i,\psi _i\}$.
Supply these sets with an open set
$U'_0\supset \comega\setminus \bigcup_iU'_i $ with $\overline U'_0\subset \Omega $, and let $\{\varrho
_i\}_{i=0,\dots,I}$ be an associated partition of unity, $\varrho
_i\in C_0^\infty (U'_i)$. Then $u=\sum_i\varrho _iu$, where also
$\varrho _iu$ belongs to $H_q^{{a}
(s+2a)}(\comega )$ by Proposition~\ref{prop:Commutator1}. Moreover,
$\varrho _iu=\varphi _i\varrho _iu$ for $i\ge 1$, where the initial
considerations apply to $\varrho
_iu$ to show that $r^+P\varphi _i\varrho _iu\in \ol H_q^s(\Omega ) $.

Summation over $i$ gives the mapping property for $u$.
\end{proof}

\subsection{Elliptic regularity in an almost flat curved halfspace }

We now turn to the study of regularity properties of solutions of
elliptic problems in this context. Also here, we restrict the
attention in the present paper to even operators; this suffices for the treatment of $(-\Delta
)^a$ and  its pseudodifferential generalizations.

In the following let  $\ol{p}\in S^{2a}(\Rn\times
\Rn)$ and $p\in C^\tau S^{2a}(\Rn\times
\Rn)$ be strongly elliptic and even, and assume that for some $1<q<\infty$ and  $s\geq 0$ 
\begin{equation}
  \label{eq:Invertibility}
  r^+ \OP(\ol{p})\colon H^{a(t+2a)}_q(\overline{\R}^n_+)\to \overline{H}_q^t(\Rn_+)\quad \text{is invertible for }t=s,s',
\end{equation}
where $s':=\min(s-1,-a)$.
Moreover, let $\gamma \in C^{1+\tau}(\R^{n-1})$, $\Rn_\gamma= \{x\in\R^n\mid x_n>\gamma(x')\}$, and
let $F_\gamma\colon \R^n\to \R^n$  be as in the preceding section.
For the following it is assumed that $s+2a<\tau$. 
Finally, let $P_\gamma$ be defined as in \eqref{eq:DefnPGamma}.

\begin{prop}\label{thm:localRegularity1} Let $\ol{p}\in S^{2a}(\Rn\times
\Rn)$ and $p\in C^\tau S^{2a}(\Rn\times
\Rn)$ be strongly elliptic and even, with $\ol p$ invertible as in {\rm\eqref{eq:Invertibility}}, and let $0\leq s <\tau -2a$.
  There are some $k\in\N$ and $\delta=\delta(\ol{p},s,q)>0$ such that 
  \begin{equation*}
    r^+ P_\gamma\colon H^{a(t+2a)}_q(\overline{\R}^n_+)\to \overline{H}_q^t(\Rn_+)
  \end{equation*}
  is invertible for $t=s$ and $t=s':=\max (s-1,-a)$ if 
  \begin{equation*}
    |\ol{p}-p|_{k,C^\tau S^{2a}_{1,0}(\Rn\times \Rn)}\leq \delta\quad \text{and} \quad \|\gamma\|_{C^{1+\tau}(\R^{n-1})}\leq \delta. 
  \end{equation*}
\end{prop}
\begin{proof}
  Because of \eqref{eq:Invertibility}, there is some $\eps>0$ such that $r^+ P_\gamma\colon H^{a(t+2a)}_q(\overline{\R}^n_+)\to \overline{H}_q^t(\Rn_+)$ is invertible for $t=s$ and $t=s'$ provided
  \begin{equation*}
    \|r^+\op(\ol{p})-r^+ P_\gamma\|_{\mathcal{L}(H^{a(s+2a)}_q(\overline{\R}^n_+), \overline{H}_q^s(\Rn_+))} <\eps.
  \end{equation*}
  Moreover, because of Theorem~\ref{thm:BddnessTruncatedPsDOs} and Theorem~\ref{thm:TransformedPsDOs}, there are some $k\in\N$, $\theta>0$ and some $C>0$ independent of $p$ and $\gamma$ with $|p|_{k,C^\tau S^{2a}(\Rn \times \Rn)}\leq 1$, $\|\gamma\|_{C^{1+\tau}(\Rn)}\leq 1$ such that
  \begin{alignat*}{1}
    &\|r^+\op(\ol{p})-r^+ P_\gamma\|_{\mathcal{L}(H^{a(t+2a)}_q(\overline{\R}^n_+), \overline{H}_q^t(\Rn_+))}\\
    &\leq
    \|r^+\op(\ol{p})-r^+\op(p)\|_{\mathcal{L}(H^{a(t+2a)}_q(\overline{\R}^n_+), \overline{H}_q^t(\Rn_+))}
    + \|r^+\op(p)-r^+ P_\gamma\|_{\mathcal{L}(H^{a(t+2a)}_q(\overline{\R}^n_+), \overline{H}_q^t(\Rn_+))} \\
    &\leq  C|\ol{p}-p|_{k,C^\tau S^{2a}_{1,0}(\Rn\times \Rn)} +C\|\gamma\|^\theta_{C^{1+\tau}(\Rn)}
  \end{alignat*}
for $t=s$ and $t=s'$.
  Hence there is some $\delta>0$ such that the right-hand side is smaller than $\eps$ provided $    |\ol{p}-p|_{k,C^\tau S^{2a}_{1,0}(\Rn\times \Rn)}\leq \delta$ and $\|\gamma\|_{C^{1+\tau}(\R^{n-1})}\leq \delta$.
\end{proof}

Next, we apply the preceding result to obtain a local regularity
result in $\R^n_\gamma$. Denote $\{x\in \R^n\mid |x-x_0|<r\}=B_r(x_0)$.

\begin{theorem}\label{thm:localRegularity2} Let $0\le s<\tau-2a$.
Assume that $\gamma\in C^{1+\tau}(\R^{n-1})$ satisfies $\gamma(0)=0,
\nabla \gamma(0)=0$, and that $u \in 
H^{a(s'+2a)}_q(\ol{\R}^n_\gamma )$ is a solution of
\begin{equation*}
  r^+ Pu = f \qquad \text{in }\R^n_\gamma
\end{equation*}
for some $f\in L_q(\Rn_\gamma)$ with $f|_{B_{R_1}(0)}\in \ol{H}^s_q(\Rn_\gamma\cap B_{R_1}(0))$ for some $R_1>0$, where $s'=\max (s-1,-a)$. Then there is some $R>0$ such that
\begin{equation*}
  u=v \qquad \text{ in } \Rn_\gamma\cap B_R(0)
\end{equation*}
for some $v\in
H^{a(s+2a)}_q(\ol{\R}_\gamma ^n)$. 
\end{theorem}

We assume for simplicity that $R_1=1$. By a suitable scaling in space, one can always reduce to this case.
The idea of the proof is to rescale in space and localize in order to apply the results for operators close to a constant coefficient pseudodifferential operator, i.e., Proposition~\ref{thm:localRegularity1}. For the rescaling we define for $R>0$
\begin{alignat*}{1}
  \gamma_R(x')&= R^{-1}\eta ((x',0)) \gamma(Rx'), \\
  p_R(x,\xi) &= \eta (x) p_0(Rx,\xi)+ (1-\eta(x)) p_0(0,\xi),\\
  \ol{p}(x,\xi) &= p_0(0,\xi),
\end{alignat*}
for all $x,\xi\in\R^n$, $x'=(x_1,\ldots,x_{n-1})$,
where $\eta \in C^\infty_0(\Rn)$ with $\eta\equiv 1$ on $B_1(0)$ and
$\supp \eta \subset B_2(0)$. 
To assure that $\op(\ol p)$ is invertible, we can in view of the strong
ellipticity assume that $p_0(0,\xi )$ has been modified for $|\xi |\le
1$ such that 
$\operatorname{Re}p_0(0,\xi)\ge c>0$ for all $\xi\in\R^n$  (also
done for $p_0(x,\xi )$ with $x$ in a small neighborhood of $0$). Then
 $\ol P=\operatorname{OP}(p_0(0,\xi ))$ satisfies
 \eqref{eq:Invertibility}, cf.\ Example \ref{ex:LaxMilgram}. 
 Furthermore, for $v\colon \Rn\to \C$ and $R>0$ we define $\sigma_R v\colon \Rn\to\C$ by
\begin{equation*}
  (\sigma_R v)(x)= v(Rx)\qquad \text{for all }x\in\Rn.
\end{equation*}
Define moreover
\begin{equation*}
  q_R(x,\xi)= p_0(Rx,R^{-1}\xi)- R^{-2a}p_0(Rx,\xi),  
\end{equation*}
and note that since $p_0(Rx,R^{-1}\xi)=R^{-2a} p_0(Rx,\xi)$ for all
$|\xi|\geq 1$ and $R\in (0,1]$ by the homogeneity,
$q_R(x,\xi)=0$ for all $|\xi|\geq 1$, $R\in (0,1]$. Hence $q_R\in C^\tau S^{-\infty}_{1,0}(\Rn\times \Rn)$.
Now $P_0=\op (p_0)$ satisfies for all $x\in\R^n$ and suitable $v\colon \R^n \to \C$:
\begin{equation*}
  \sigma_R(P_0v)(x)= \int_{\R^n} e^{iRx\cdot \xi} p_0(Rx,\xi) \hat{v}(\xi)\dd \xi
          = \int_{\R^n} e^{ix\cdot \xi} p_0(Rx,R^{-1}\xi)
          \widehat{\sigma_R(v)}(\xi)\dd \xi.
        \end{equation*}
Since $p_R(x,\xi)= p_0(Rx,\xi)$ if $|x|\leq 1$, this may by use of $q_R$
be written: 
\begin{equation}        \label{eq:IdPR}
  \sigma_R(P_0v)(x)= R^{-2a}(\op (p_R) \sigma_R(v))(x) +
  (\op(q_R)\sigma_R(v))(x)\text{ for }|x|\le 1.
\end{equation}


Moreover, we 
show a technical lemma in order to control remainder terms:

\begin{lemma}\label{lem:Scaling}
For any $k\in\N$ and $R\in (0,1]$,
\begin{equation*}
  \|\gamma_R\|_{C^{1+\tau}(\R^{n-1})}\leq CR^{\min (1,\tau) },\qquad
  |p_R-\ol{p}|_{k,C^{\tau}S^{2a}_{1,0}(\Rn\times \Rn)}\leq CR^{\min (1,\tau) }.
\end{equation*}
\end{lemma}
\begin{proof}
  Using $\gamma (0)=0,\nabla \gamma(0)=0$  we have
  \begin{equation*}
    \gamma(Rx')= \int_0^1 (\nabla \gamma(sRx')-\nabla \gamma(0)) \, ds \cdot Rx'
  \end{equation*}
  and therefore
  \begin{equation*}
    \sup_{|x'|\leq 2} |\gamma (Rx')|\leq C |R|^{1+\min(\tau,1)}.
  \end{equation*}
  since $\nabla \gamma \in C^\tau(\R^{n-1})$.
  Now let $\alpha\in \N_0^{n-1}$ with $|\alpha|< 1+\tau$. Using
  \begin{equation*}
    \partial_{x'}^\alpha (\gamma(Rx'))= R^{|\alpha|} (\partial_{x'}^\alpha \gamma)(Rx')
  \end{equation*}
  one obtains in the same way
  \begin{equation*}
    \sup_{|x'|\leq 2} |\partial_{x'}^\alpha (\gamma(Rx'))|\leq
    \begin{cases}
      C |R|^{1+\min(\tau,1)}&\quad \text{if }|\alpha|=1,\\
      C |R|^{|\alpha|}\leq C |R|^{1+\min(\tau,1)}  &\quad \text{if }|\alpha|\geq 2,
    \end{cases}
  \end{equation*}
  since $\partial_{x'}^\alpha \gamma  \in C^{\tau+1-|\alpha|}(\R^{n-1})$, $\partial_{x'}^\alpha
  \gamma(0)=0$ if $|\alpha|=1$, and
  $R\leq 1$. 

  Now let $|\alpha|=1+[\tau]$. Then
  one obtains
  \begin{align*}
    &\sup_{|x'|,|y'|\leq 2, x'\neq y'} \frac{|\partial_{x'}^\alpha(\gamma (Rx'))-\partial_{y'}^\alpha(\gamma(Ry'))|}{|x'-y'|^{\tau}}\\
    &\leq  \sup_{x',y'\in\R^{n-1}, x'\neq y'} \frac{|(\partial_{x'}^\alpha\gamma) (Rx'))-(\partial_{x'}^\alpha\gamma)(Ry'))|}{|Rx'-Ry'|^{\tau}}R^{|\alpha|+\tau}\leq C |R|^{1+\tau}.
  \end{align*}
  Therefore
  $
  \|R^{-1}\gamma(R\cdot)\|_{C^{1+\tau}(\ol{B_2(0)})}\leq C R^{\min(1,\tau)}.
  $
  This implies 
  \begin{equation*}
    \|\gamma_R\|_{C^{1+\tau}(\R^n)}=\|\eta R^{-1}\gamma(R\cdot)\|_{C^{1+\tau}(\ol{B_2(0)})}\leq C'\|R^{-1}\gamma(R\cdot)\|_{C^{1+\tau}(\ol{B_2(0)})}\leq C R^\tau.
  \end{equation*}

  In a similar manner one shows for every  $\alpha, \beta\in\N_0$ with $|\beta|<\tau$
  \begin{equation*}
    \sup_{|x|\leq 2,} |\partial_x^\beta \partial_\xi^\alpha (p_0(Rx,\xi)-p_0(0,\xi))|\leq
    \begin{cases}
      C_{\alpha,\beta} R^{\min (1,\tau)}\weight{\xi}^{m-|\alpha|}&\text{if }\beta=0,\\ 
    C_{\alpha,\beta} R^{|\beta|}\weight{\xi}^{m-|\alpha|} \leq C_{\alpha,\beta} R^{\min (1,\tau)}\weight{\xi}^{m-|\alpha|}  &\text{if }\beta\neq 0
    \end{cases}
  \end{equation*}
  and, if $|\beta|=[\tau]$,
  \begin{equation*}
    \sup_{|x|,|y|\leq 2} \frac{|\partial_x^\beta \partial_\xi^\alpha (p_0(Rx,\xi)-p_0(Ry,\xi))|}{|x-y|^{\tau-[\tau]}}\leq 
      C_{\alpha,\beta} R^{\tau}\weight{\xi}^{m-|\alpha|}
    \end{equation*}
  for all $\xi\in\R^n$ and $R\in (0,1]$. From this one derives the second statement in a straightforward manner with the aid of the product rule.
\end{proof}

\noindent
\begin{proof*}{of Theorem~\ref{thm:localRegularity2}}
  First of all, because of Proposition~\ref{thm:localRegularity1} and Lemma~\ref{lem:Scaling}, there is some $R\in (0,1]$ such that
\begin{equation*}
  r^+ P_{\gamma_R} \colon H^{a(t+2a)}_q(\overline{\R}^n_+)\to \overline{H}_q^t(\Rn_+)
\end{equation*}
is invertible for $t=s$ and $t=s'$, where
$ 
  P_{\gamma_R} u = F^{\ast}_{\gamma_R} P_R F^{\ast,-1}_{\gamma_R}
  u\text{ for all }u\in H^{a}_q(\Rn)$. 
For the following we fix such an $R$.
Then we have that
\begin{equation*}
  r^+ P_R\colon H^ {a(s+2a)}_q(\overline{\R}^n_{\gamma_R})\to \overline{H}_q^s(\overline{\R}^n_{\gamma_R})
\end{equation*}
is invertible.

Now we localize the given solution $u\in H^{a(s'+2a)}_q(\ol{\R}^n_{\gamma })$. To this end let $\psi\in C^\infty_0(B_1(0))$ with $\psi\equiv 1$ on $B_{1/2}(0)$ and set $\psi_R(x)= \psi(x/R)=(\sigma_{1/R}\psi)(x)$ for all $x\in\Rn$.
Then
\begin{alignat}{1}\label{eq:localEquation}
  r^+ P_0 (\psi_R u) = r^+ \psi_R {P} u + g =  \psi_R r^+ P u + g = \psi_R f+ g=:\tilde{f}\qquad\text{in }\Rn_\gamma,
\end{alignat}
where $P_0=\op(p_0)$,  $g:= \psi_R r^+(P_0-P) u+
r^+[P_0,\psi_R]u$, and $[P_0,\psi_R]=\op (q)$ for some
odd $q\in C^\tau S^{2a-1}(\Rn\times \Rn)$ due to
Proposition~\ref{prop:Commutator} and  $P_0-P=\op(p_0-p)$, where
$p_0-p\in C^\tau S^{2a-1}(\Rn\times \Rn)$ is odd. Therefore $q$ and
$p_0-p$ satisfy the $a$-transmission condition and we conclude that $g
\in \ol{H}^{s}_q(\Rn_\gamma)$ because of $u\in
H^{a(2a+s')}_q(\ol{\R}^n_\gamma )$, $2a+s'=\max (a,s+2a-1)$ and Corollary~\ref{cor:NonsmoothMap3}. Hence $\tilde{f}\in \ol{H}^{s}_q(\Rn_\gamma)$ since $f|_{B_1(0)}\in \ol{H}^{s}_q(\Rn_\gamma\cap B_1(0))$.

Moreover, by the definition of $\gamma_R$
\begin{equation*}
  Rx\in \Rn_\gamma\cap B_R(0)\quad \text{if and only if}\quad x\in \Rn_{\gamma_R}\cap B_1(0).
\end{equation*}
Hence \eqref{eq:localEquation} in $\Rn_\gamma\cap B_R(0)$ is equivalent to
\begin{equation*}
  r^+ P_R (\psi \sigma_R (u))= R^{2a}(\sigma_R(\tilde{f}) - \op(q_R)(\psi \sigma_R(u))\qquad \text{in } \Rn_{\gamma_R}\cap B_1(0)
\end{equation*}
because of \eqref{eq:IdPR} and $\psi \sigma_R(u)= \sigma_R(\psi_R u)$.
Moreover, since $\supp (\psi\sigma_Ru)$ is compactly contained in
$B_1(0)$, thus contained in $B_\lambda (0)$ for some $\lambda <1$, we have $P_R(\psi \sigma_R u)\in \ol{H}^s_q(\Rn\setminus B_\lambda(0)))$ because of Remark~\ref{rem:cutoff}. Hence
\begin{equation*}
  r^+ P_R (\psi \sigma_R (u))= h\qquad \text{in } \Rn_{\gamma_R}
\end{equation*}
for some $h\in \ol{H}^s_q({\Rn_{\gamma_R}})$.
Since this equation has a unique solution in
$H^{a(2a+s)}_q(\ol{\R}^n_{\gamma_R})$ and in
$H^{a(2a+s')}_q(\ol{\R}^n_{\gamma_R})$, we conclude
$\psi \sigma_R(u)\in H^{a(2a+s)}_q(\ol{\R}^n_{\gamma_R})$. Scaling
back implies $u= \tilde{v}$ in $B_{R/2}(0)\cap \Rn_\gamma$ for some
$\tilde{v}
\in H^{a(2a+s)}_q(\ol{\R}^n_\gamma )$.
\end{proof*}
\begin{corollary}\label{cor:localRegularity2} Let $0\le s<\tau-2a$.
Assume that $\gamma\in C^{1+\tau}(\R^{n-1})$ satisfies $\gamma(0)=0, \nabla \gamma(0)=0$ and that $u \in \dot{H}^{a}_q(\ol{\R}^n_\gamma)$ is a solution of
\begin{equation*}
  r^+ Pu = f \qquad \text{in }\R^n_\gamma
\end{equation*}
for some $f\in L_q(\Rn_\gamma)$ with $f|_{B_{R_1}(0)}\in \ol{H}^s_q(\Rn_\gamma\cap B_{R_1}(0))$ for some $R_1>0$. Then there is some $R>0$ such that
\begin{equation*}
  u=v \qquad \text{ in } \Rn_\gamma\cap B_R(0)
\end{equation*}
for some $v\in H^{a(s+2a)}_q(\ol{\R}_\gamma ^n)$. 
\end{corollary}
\begin{proof}
  If $s\le 1-a$, then the statement follows directly from Theorem~\ref{thm:localRegularity2} since $s'=-a$ and $H^{a(s'+2a)}_q(\ol{\R}^n_\gamma )= \dot{H}^a_q(\ol{\R}^n_\gamma)$ in that case.
  Let us consider the case $s>1-a$. Then we see from the proof of Theorem~\ref{thm:localRegularity2} that
  \begin{equation*}
    r^+ P (\psi_R u)= \psi_R f+r_{\Rn_\gamma}[P,\psi_R]u=: f' \qquad \text{in } \Rn_{\gamma_R},
  \end{equation*}
  where $\psi_R u \in H^{a(a+1)}_q(\ol{\R}^n_\gamma)$. Moreover, if $\eta\in C_0^\infty(B_{R/2}(0))$ such that $\eta\equiv 1$ on $B_{R/4}(0)$, then
  \begin{align*}
    \eta [P,\psi_R]u &= \op (\eta(x) p(x,\xi)(\psi_R(x)-\psi_R(y)) u= \op (\eta(x) p(x,\xi)(1-\psi_R(y))u\\
    &= \eta P ((1-\psi_R)u)\in H^s_q(\Rn) 
  \end{align*}
  because of Remark~\ref{rem:cutoff} and $\supp \eta \cap \supp (1-\psi_R)=\emptyset$. Hence $f' \in \ol{H}^s_q(\Rn_\gamma\cap B_{R/4}(0))$. Therefore we can apply Theorem~\ref{thm:localRegularity2} to $\psi_R u$ and $f'$ again to conclude the statement of the corollary provided that $s\le 2-a$. Repeating this argument finitely many times with the help of Corollary~\ref{cor:Commutator}, we obtain the statement in the general case.
\end{proof}

\subsection{Elliptic regularity in a bounded domain}

Now we are in the position to prove:
\begin{theorem}\label{thm:ellipticRegularity}
  Let $1<q<\infty$, $a\in (0,1)$, $\tau>2a$, and $0\leq s
  <\tau-2a$. Let  $\Omega\subset \R^n$ be a bounded 
  $C^{1+\tau}$-domain, and let  $p\in C^\tau S^{2a}(\Rn\times \Rn)$
  be an even and strongly elliptic symbol, $P=\op(p(x,\xi))$. If $u\in \dot{H}^a_q(\ol{\Omega})$ solves
  \begin{equation}\label{eq:6.9}
    r^+ Pu = f\qquad \text{in }\Omega 
  \end{equation}
  for some $f\in \ol{H}^s_q(\Omega)$, then $u\in H^{a(s+2a)}_q(\ol\Omega)$.
\end{theorem}
\begin{proof}
Let $x_0\in \partial\Omega$ be arbitrary. Moreover, let $\gamma\in C^{1+\tau}(\R^{n-1})$ and $R_0>0$ be such that
\begin{equation*}
  \Omega\cap B_{R_0}(x_0)= \Rn_\gamma\cap B_{R_0}(x_0)
\end{equation*}
(after a suitable rotation). By a simple translation and rotation we can always reduce to the case $x_0=0$, $\gamma(0)=0$, and $\nabla \gamma (0)=0$.
It suffices to show that there is an $R\in (0,R_0]$ such that
\begin{equation*}
  u = v \qquad \text{in }\Rn_\gamma \cap B_R(x_0)
\end{equation*}
for some $v\in H^{a(2a+s)}_q(\ol{\R}^n_\gamma)$.
Now let $\psi\in C^\infty_0(B_{R_0}(x_0))$ with $\psi\equiv 1$ on $B_{R_0/2}(x_0)$. Then
\begin{alignat*}{1}
  r^+ {P} (\psi u) = \psi r_\Omega {P} u + g = \psi f+ g\qquad\text{in }\Rn_\gamma,
\end{alignat*}
where $g:= r^+[{P},\psi]u \in
\ol{H}^{1-a}_q(\Rn_\gamma)\subset L_q(\Rn_\gamma)$ since $u\in \dot{H}^a_q(\ol{\Omega})$, and $q\in C^\tau S^{2a-1}_{1,0}(\Rn\times\Rn)\subset C^\tau S^{a}_{1,0}(\Rn\times\Rn)$. Moreover, $g|_{B_{R_0/4}(0)}\in \ol{H}^s_q(\Rn_\gamma \cap B_{R_0/4}(0))$ because of Remark~\ref{rem:cutoff} and the same observations as in the proof of Corollary~\ref{cor:localRegularity2}.
Hence Corollary~\ref{cor:localRegularity2} implies that there is an $R>0$ such that
\begin{equation*}
  u=v \qquad \text{ in } \Rn_\gamma\cap B_R(0)
\end{equation*}
for some $v\in H^{a(s+2a)}_q(\ol{\R}_\gamma ^n)$. Hence the statement of the theorem follows.
\end{proof}

Combining Theorem \ref{thm:ellipticRegularity} with the forward
mapping property shown in Theorem~\ref{thm:MappingGeneralDomain} we
find:

\begin{corollary}\label{cor:DirDomain}
Hypotheses as in Theorem {\rm \ref{thm:ellipticRegularity}}.
A function 
  $u\in \dot H_q^a(\comega)$  solves {\rm \eqref{eq:6.9}} for some
 $f\in \ol H_q^s(\Omega )$ if and only if $u\in
 H^{a(s+2a)}_q(\comega)$.
 
 Hence the Dirichlet domain {\rm \eqref{eq:1.7a}} for $P$ with data in
 $\ol H_q^s(\Omega )$ equals $H^{a(s+2a)}_q(\comega)$. 
  \end{corollary}

\noindent
\begin{proof*}{of Theorem~\ref{thm:main}}
  The result $1^\circ$ is a consequence of
  Theorem 6.4, and $2^\circ$ is shown above in Corollary \ref{cor:DirDomain}.
\end{proof*}

The theorem applies to $(-\Delta )^a$ in the way that $(-\Delta )^a=P_1+P_2$,
where $P_1=\op((1-\psi (\xi ))|\xi |^{2a})$ satisfies the hypotheses
and $P_2$ is smoothing, so that $r^+(-\Delta )^au=f\in \ol
H_q^s(\Omega )$ is turned into $r^+P_1u=f_1$, with $f_1=f-r^+P_2u\in
\ol H_q^s(\Omega )$ since $P_2$ is smoothing. 

There is a  corollary on regularity in H\"older spaces:

\begin{corollary}\label{cor:Zygmundref}
Hypotheses as in Theorem~\ref{thm:ellipticRegularity}.
$1^\circ$ If $f\in \ol C^s_*(\Omega )$ for some $s\in(0,\tau -2a)$, then 
$u\in C_*^{a(s+2a-\varepsilon )}(\comega)$ for every  small $\varepsilon
>0$. When $\tau \ge 1$, it satisfies
\begin{equation}\label{eq:Zygmundreg}
u\in \dot C^{s+2a-\varepsilon }(\comega)+d^ae^+\ol C^{s+a-\varepsilon
}(\Omega ).
\end{equation}
If $2a<1$ and $2a<\tau <1$, it satisfies a local version of \eqref{eq:Zygmundreg}, cf.\ Remark~\ref{rem:TransmissionSpace}.

$2^\circ$ If $f\in L_\infty (\Omega )$,  then 
$u\in C_*^{a(2a-\varepsilon )}(\comega)$ for every small $\varepsilon >0$,
satisfying \eqref{eq:Zygmundreg}~ff.\ with $s$ replaced by $0$.

\end{corollary}

\begin{proof}
We use the Sobolev embedding property $H_q^t(\rn)\subset
C_*^{t-n/q-\varepsilon }(\rn)$, which implies
$H_q^{a(t)}(\comega)\subset C_*^{a(t-n/q-\varepsilon )}(\comega)$ in
view of Definition \ref{def:roughTransmSpaces} (when $a\le t-n/q-\varepsilon <t<1+\tau $). For $1^\circ$,  $f\in \ol
C^s_*(\Omega )\subset \ol H_q^{s-\varepsilon /2}(\Omega )$ implies when
$n/q<\varepsilon /2$ that $u\in H_q^{a(s+2a-\varepsilon /2)}(\comega)\subset
C_*^{a(s+2a-\varepsilon )}(\comega)$, and  \eqref{eq:Zygmundreg}ff.\
follow from Theorem \ref{thm:TransmissionSpace} and Remark
\ref{rem:TransmissionSpace}. For $2^\circ$, we conclude similarly from
$f\in L_\infty (\Omega )\subset L_q(\Omega )=\ol H_q^0(\Omega )$ that
$u\in H_q^{a(2a)}(\comega)\subset 
C_*^{a(2a-\varepsilon )}(\comega)$, when
$n/q<\varepsilon $, with the ensuing descriptions.
\end{proof}

\appendix
\section{Appendix}

\begin{proof*}{of Theorem \ref{thm:OscIntegrals}}
    We use \eqref{eq:Taylor} with $l=\max\{[m-|\gamma|],0\}$.
  Since $\partial_y^\alpha a(x,y,\xi)|_{y=x} (y-x)^\alpha u(y)= (\partial_y^\alpha a)(x,x,\xi) (y-x)^\alpha u(y)$ is smooth with respect to $(y,\xi)$, the existence of
  \begin{equation*}
    \lim_{\eps\to 0} \int_{\R^{2n}}\chi(\eps y,\eps \xi) e^{i(x-y)\cdot \xi} \left(\partial_y^\alpha a(x,y,\xi)|_{y=x}\right) (x-y)^\alpha u(y)\sd y\dd\xi = \op(p_\alpha)u(x)
  \end{equation*}
  follows from standard results on oscillatory integrals. Here
  \begin{equation*}
    p_\alpha (x,\xi)= \partial_y^\alpha D_\xi^\alpha a(x,y,\xi)|_{y=x}\qquad \text{for all }x,\xi\in\Rn,
  \end{equation*}
  because of the calculation rules for oscillatory integrals. Therefore it only remains to show the existence of the oscillatory integrals for $(y-x)^{\alpha+\gamma}r_\alpha(x,y,\xi)$.
  Now we use that
  \begin{alignat*}{1}
    &\int_{\R^{2n}} \chi(\eps y,\eps \xi) e^{i(x-y)\cdot \xi} (y-x)^{\alpha+\gamma}r_\alpha(x,y,\xi) u(y)\sd y\dd\xi\\
    &= \int_{\R^n} \int_{\R^n}  e^{i(x-y)\cdot \xi}\chi(\eps y,\eps \xi) (y-x)^{\alpha+\gamma} r_\alpha(x,y,\xi) \dd\xi\, u(y)\sd y
    = \int_{\R^n} k_\eps (x,y,x-y) u(y)\sd y,
  \end{alignat*}
  where
  \begin{alignat*}{1}
    k_\eps (x,y,z) &:= (y-x)^{\alpha+\gamma}\int_{\R^n} e^{iz\cdot \xi}\chi(\eps y,\eps \xi) r_\alpha(x,y,\xi) \dd\xi= \sum_{j\in\N_0} k_{\eps,j}(x,y,z),\\
    k_{\eps,j} (x,y,z) &:= (y-x)^{\alpha+\gamma}\int_{\R^n} e^{iz\cdot \xi}\chi(\eps y,\eps \xi) r_\alpha(x,y,\xi)\varphi_j(\xi) \dd\xi.
  \end{alignat*}
  Using \eqref{eq:ralphaEstim} one shows in the same manner as in
  the proof of \cite[Lemma~5.14]{PsDOBuch} that for every $N\in\N_0$ there is some $C_N>0$ such that
  \begin{equation*}
    |k_{\eps,j} (x,y,z)|\leq C_N |x-y|^{l+|\gamma|+\theta} |z|^{-N} 2^{j(n+m-N)} 
  \end{equation*}
  for all $z\neq 0, j\in\N_0, \eps\in (0,1),x,y\in\R^n$, where $\theta= \min \{\tau-l, 1\}$.
  Using
  \begin{equation*}
    \sum_{j\in\N_0} k_{\eps,j}(x,y,z)= \sum_{2^j\leq |z|^{-1}} k_{\eps,j}(x,y,z)+ \sum_{2^j> |z|^{-1}} k_{\eps,j}(x,y,z)
  \end{equation*}
  one can derive in the same way as in the proof of
  \cite[Theorem~5.12]{PsDOBuch} that the series $\sum_{j\in\N_0} k_{\eps,j}(x,y,z)$ converges absolutely and uniformly in $|z|\geq \delta$, $x,y\in\R^n$ for any $\delta>0$ to a function $k_\eps\colon \R^n\times \R^n \times (\R^n\setminus \{0\})\to \C$ that satisfies for any $N\in\N_0$
  \begin{equation*}
    |k_\eps(x,y,z)|\leq
    \begin{cases}
      C_N|x-y|^{l+|\gamma|+\theta} |z|^{-m-n}\weight{z}^{-N}&\text{if } m+n>0,\\
      C_N|x-y|^{l+|\gamma|+\theta}(1+ \log |z|^{-1}) \weight{z}^{-N}&\text{if } m+n=0,\\
    C_N|x-y|^{l+|\gamma|+\theta}\weight{z}^{-N}&\text{if } m+n<0,
    \end{cases}
  \end{equation*}
  for all $x,y\in\R^n$, $z\neq 0$, $\eps\in (0,1)$, $j\in\N_0$. Now, if we choose $N\in\N_0$ sufficiently large and $z=x-y$, the right-hand side is in $L_1(\Rn)$ with respect to $y$ since $\tau+|\gamma|> m$. Hence by the dominated convergence theorem the limit
  \begin{alignat*}{1}
    &\lim_{\eps\to 0}   \int_{\R^{2n}} \chi(\eps y,\eps \xi) e^{i(x-y)\cdot \xi}(y-x)^{\alpha+\gamma} r_\alpha (x,y,\xi) u(y)\sd y\dd\xi
    = \int_{\R^n} k_{\alpha,\gamma}(x,y,x-y) u(y)\sd y 
  \end{alignat*}
  exists, 
  where $k_{\alpha,\gamma}$ and $k_{\alpha,\gamma,j}$ are as in the theorem.
  This concludes the proof.
\end{proof*}

\begin{proof*}{of Lemma \ref{lem:PartIntOscInt}}
  First of all, note that both oscillatory integrals exist because of Theorem~\ref{thm:OscIntegrals}, $D_\xi^\beta a \in C^\tau S^{m-|\beta|}_{1,0}(\R^{2n}\times \Rn)$, and $m-|\beta|<\tau+ |\gamma|-|\beta|$. Moreover, it is sufficient to consider the case $\beta =e_k$ for some $k\in \{1,\ldots,n\}$. The general case follows inductively.
In view of Theorem~\ref{thm:OscIntegrals}, it only remains to show that
  \begin{equation*}
    \op ((y-x)^{\alpha+\gamma} r_{\alpha} (x,y,\xi))u(x)=   \op ((y-x)^{\alpha+\gamma-e_k} D_{\xi_k} r_{\alpha} (x,y,\xi))u(x).
  \end{equation*}
  By Theorem~\ref{thm:OscIntegrals} applied to $D_{\xi_k} a$ we have
  \begin{equation*}
    \op ((y-x)^{\alpha+\gamma-e_k} D_{\xi_k} r_{\alpha} (x,y,\xi))u(x) = \int_{\Rn} \tilde{k}_{\alpha,\gamma} (x,y,x-y)u(y) \, dy,
    \end{equation*}
    where $\tilde{k}_{\alpha,\gamma} (x,y,z)= \sum_{j\in\N_0} \tilde{k}_{\alpha,\gamma,j}(x,y,z)$ and
      \begin{alignat*}{1}
        \tilde{k}_{\alpha,\gamma,j}(x,y,z) &= \int_{\Rn} e^{iz\cdot \xi} (y-x)^{\alpha+\gamma-e_k} D_{\xi_k} r_\alpha (x,y,\xi) \varphi_j(\xi)\dd \xi
      \end{alignat*}
      For $z=x-y\neq 0$, an integration by parts yields 
      \begin{alignat*}{1}
        &\tilde{k}_{\alpha,\gamma,j}(x,y,x-y) 
        = \int_{\Rn} e^{i(x-y)\cdot \xi} (y-x)^{\alpha+\gamma} r_\alpha (x,y,\xi) \varphi_j(\xi)\dd \xi\\
        &\quad - \int_{\Rn} e^{i(x-y)\cdot \xi} (y-x)^{\alpha+\gamma-e_k} r_\alpha (x,y,\xi) D_{\xi_k}\varphi_j(\xi)\dd \xi\\
        &= k_{\alpha,\gamma,j} (x,y,x-y) -  \int_{\Rn} e^{i(x-y)\cdot \xi} (y-x)^{\alpha+\gamma-e_k} r_\alpha (x,y,\xi) D_{\xi_k}\varphi_j(\xi)\dd \xi.
      \end{alignat*}
      By the results in the proof of Theorem~\ref{thm:OscIntegrals}
      \begin{equation*}
        \sum_{j\in\N_0} \tilde{k}_{\alpha,\gamma,j}(x,y,x-y) \quad \text{and} \quad \sum_{j\in\N_0} k_{\alpha,\gamma,j}(x,y,x-y)
      \end{equation*}
      converge absolutely for every $x\neq y$.
      Hence the same is true for
      \begin{equation*}
        \sum_{j\in\N_0} \int_{\Rn} e^{i(x-y)\cdot \xi} (y-x)^{\alpha+\gamma-e_k} r_\alpha (x,y,\xi) D_{\xi_k}\varphi_j(\xi)\dd \xi.
      \end{equation*}
      Here for $x\neq y$ and $N\in\N$
      \begin{alignat*}{1}
        &\int_{\Rn} e^{i(x-y)\cdot \xi} (y-x)^{\alpha+\gamma-e_k} r_\alpha (x,y,\xi) D_{\xi_k}\varphi_j(\xi)\dd \xi\\
        &=|x-y|^{-2N}\int_{\Rn} e^{i(x-y)\cdot \xi} (y-x)^{\alpha+\gamma-e_k} (-\Delta_\xi)^N \left(r_\alpha (x,y,\xi) D_{\xi_k}\varphi_j(\xi)\right)\dd \xi,
      \end{alignat*}
      where
      \begin{equation*}
        \left|(-\Delta_\xi)^N \left(r_\alpha (x,y,\xi) D_{\xi_k}\varphi_j(\xi)\right)\right|\leq C_N \weight{\xi}^{m-2N}\quad \text{for all }\xi\in\Rn.
      \end{equation*}
      Hence choosing $N\in\N$ such that $m-2N<-n$, we can apply the dominated convergence theorem to conclude
      \begin{alignat*}{1}
        &\sum_{j\in\N_0} \int_{\Rn} e^{i(x-y)\cdot \xi} (y-x)^{\alpha+\gamma-e_k} r_\alpha (x,y,\xi) D_{\xi_k}\varphi_j(\xi)\dd \xi\\
        &=\sum_{j\in\N_0}|x-y|^{-2N}\int_{\Rn} e^{i(x-y)\cdot \xi} (y-x)^{\alpha+\gamma-e_k} (-\Delta_\xi)^N \bigl(r_\alpha (x,y,\xi) \sum_{j\in\N_0}D_{\xi_k}\varphi_j(\xi)\bigr)\dd \xi
        =0,
      \end{alignat*}
      because of $0 = \sum_{j=0}^\infty D_{\xi_k} \varphi_j(\xi)$ for all $\xi\in\R^n$.
      Thus $\tilde{k}_{\alpha,\gamma}(x,y,x-y)=
      k_{\alpha,\gamma}(x,y,x-y)$ for all $x\neq y$, and the statement
      of the lemma 
      follows. 
    \end{proof*}

\begin{proof*}{of Theorem \ref{thm:Bddness}}
   First we treat the case $s\geq 0$ by splitting it up in  several cases. Afterwards the case $s<0$ follows by duality.\\[1ex]
  {\bf Case $m\in [0,1]$ and $0\leq s<1$:} Let us define $\sigma:=m$ and $\varrho := \tau -m$ if $\tau-m<1$, and $\varrho \in (s,1)$ arbitrary if $\tau-m\geq 1$. Then $a\in C^{\varrho +\sigma}S^\sigma_{1,0}(\R^{2n}\times \R^n)$, and \cite[Proposition 9.8]{ToolsForPDE} yields the boundedness of 
  \begin{equation*}
    \op(a(x,y,\xi))\colon H^{s+m}_q(\R^n)\to H^s_q(\R^n).
  \end{equation*}
  {\bf Case $m\in (-1,0)$ and $0\leq s<1$:} We use the decomposition
  \begin{equation*}
    \op(a(x,y,\xi))= \op (a(x,x,\xi))+ \op (b(x,y,\xi)),
  \end{equation*}
  where $b(x,y,\xi)= a(x,y,\xi)-a(x,x,\xi)$. Because of Theorem~\ref{thm:bd-compos}, it is sufficient to show the corresponding mapping property for $\op (b(x,y,\xi))$.  Since $b\in C^\tau S^m_{1,0}(\R^{2n}\times \Rn)\subset C^\tau S^0_{1,0}(\R^{2n}\times \Rn) $ and $b(x,x,\xi)=0$ for all $x,\xi\in\R^n$, \cite[Proposition 9.5]{ToolsForPDE} yields that
  \begin{equation}\label{eq:mapping-b}
    \op(b(x,y,\xi))\colon H^\sigma_q(\R^n)\to H^{\sigma+t}_q(\R^n)\quad \text{for all }-\tau <\sigma\leq 0,\; 0\leq t<\tau.
  \end{equation}
  If $s+m\leq 0$, we can choose $\sigma=s+m\in (-\tau,0]$ and $t=-m\in [0,\tau)$ and obtain the desired mapping property. If $s+m>0$, we use that $H^{s+m}_q(\Rn)\hookrightarrow L_q(\Rn)$ and \eqref{eq:mapping-b} with $\sigma=0$ and $t=s$ to conclude that
 $
    \op(b(x,y,\xi))\colon H^{s+m}_q(\R^n)\to H^s_q(\R^n)
$
is bounded.
  \\[1ex]
  {\bf Case $m\in (-1,1]$ and $s\geq 1$:}
  Let $k=[s], s'= s-k\in [0,1)$. We use that
  \begin{equation*}
    u\in H^s_q(\R^n) \quad \text{if and only if}\quad \partial_x^\alpha u\in H^{s'}_q(\Rn)\text{ for all }|\alpha|\leq k
  \end{equation*}
  for $u=\op(a(x,y,\xi)) f$ for some $f\in H^{s+m}_q(\Rn)$.

  First let $m\in [0,1]$. Using that
  \begin{equation*}
    [\partial_{x_j}, \op(a(x,y,\xi)] = \op ( \partial_{x_j} a(x,y,\xi)+ \partial_{y_j} a(x,y,\xi))
  \end{equation*}
  for all $j=1,\ldots,n$ one obtains
  \begin{equation*}
    \partial_x^\alpha \op(a(x,y,\xi))f = \sum_{0\leq \beta\leq \alpha} \op(a_\beta(x,y,\xi))\partial_x^{\alpha-\beta} f
  \end{equation*}
  for some $a_\beta \in C^{\tau-|\beta|}S^m_{1,0}(\R^{2n}\times \Rn)\subset C^{\tau-k}S^m_{1,0}(\R^{2n}\times \Rn)$, where $0\leq s'=s-k <\tau':=\tau-k$ and $0\leq s'+m= s+m-k<\tau'$ because of $|s|,|s+m|<\tau$. Moreover, $\partial_x^{\alpha-\beta} f\in H^{s'+m}_q(\Rn)$. By the preceding cases,
  \begin{equation*}
    \op(a_\beta(x,y,\xi))\colon H^{s'+m}_q(\Rn)\to H^{s'}_q(\Rn)
  \end{equation*}
  is bounded. Altogether this yields the boundedness of $\op(a(x,y,\xi))$ in this case if $m\in [0,1]$.

  If $m\in (-1,0)$ and $1\leq |\alpha|\leq k$, then $\alpha=\alpha'+e_j$ for some $j\in \{1,\ldots, n\}$. Using 
  \begin{equation*}
    \partial_{x_j} \op (a(x,y,\xi))f = \op (a(x,y,\xi)i\xi_j) + \op (\partial_{x_j} a(x,y,\xi)),
  \end{equation*}
  one obtains  similarly as before
    \begin{equation*}
    \partial_x^\alpha \op(a(x,y,\xi))f = \sum_{0\leq \beta\leq \alpha'} \op(a'_\beta(x,y,\xi)+ a''_\beta(x,y,\xi))\partial_x^{\alpha'-\beta} f
  \end{equation*}
  for some $a'_\beta \in C^{\tau-k+1}S^{m+1}_{1,0}(\R^{2n}\times \Rn)$ and $a''_\beta \in C^{\tau-k}S^{m}_{1,0}(\R^{2n}\times \Rn)\subset C^{\tau-k}S^0_{1,0}(\R^{2n}\times \Rn)$. Since $m+1\in (0,1)$, one obtains by the preceding cases that
    \begin{equation*}
    \op(a'_\beta(x,y,\xi))\colon H^{s'+m+1}_q(\Rn)\to H^{s'}_q(\Rn),\quad     \op(a''_\beta(x,y,\xi))\colon H^{s'}_q(\Rn)\to H^{s'}_q(\Rn)
  \end{equation*}
  are bounded due to $0\leq s'+m+1=s+m-k+1<\tau-k+1$ and $0\leq s'<\tau-k$. This yields the statement in this case since $\partial_x^{\alpha'-\beta}f$ in $H_q^{s'+m+1}(\Rn)$ for all $f\in H_q^{s+m}(\Rn)$ and $0\leq \beta \leq \alpha'$, where $|\alpha'|\leq k-1$, and the case $\alpha=0$ is easy. 
  \\[1ex]
  \noindent
  {\bf Case $m\in [0,\tau)$ and $s\geq 0$:}
  Now let $0\leq m <\tau$ and $s\geq 0$ and set $m'= [m]$. We use that there are polynomials $p_k\colon \R^n\to \R$ of order at most $m'$ and $q_k\in S^{0}_{1,0}(\Rn\times \Rn)$, independent of $x$, $k=0,\ldots, n$ such that
  \begin{equation*}
    \weight{\xi}^{m'} = {\sum}_{k=0}^n q_k(\xi)p_k(\xi),
  \end{equation*}
  cf.\ e.g.\ \cite[Proof of Theorem~6.8]{PsDOBuch}.
  Hence
  \begin{equation*}
    a(x,y,\xi)= {\sum}_{k=0}^n a_k(x,y,\xi) p_k(\xi),
  \end{equation*}
  where $a_k\in C^\tau S^{m-m'}_{1,0}(\R^{2n}\times \R^n)$. Combining this with the general relation
  \begin{equation}\label{eq:CompRight}
    \op(b(x,y,\xi)i\xi_j)= \op (b(x,y,\xi))\partial_{x_j}+ \op (\partial_{y_j} b(x,y,\xi))
  \end{equation}
  we have the representation
  \begin{equation*}
    \op (a(x,y,\xi))= {\sum}_{|\alpha|\leq m'} \op (a_\alpha (x,y,\xi)) \partial_x^\alpha  
  \end{equation*}
  for some $a_\alpha\in C^{\tau-m'}S^{m-m'}_{1,0}(\R^{2n}\times \Rn)$. Because of the case ``$m\in [0,1]$ and $s\geq 0$'', we conclude that
  \begin{equation*}
    \op (a_\alpha (x,y,\xi))\colon H^{s+m-m'}_q(\Rn)\to H^s_q(\Rn)
  \end{equation*}
  is bounded for every $|\alpha|\leq m'$. This implies the statement in this case.\\[1ex]
  \noindent
  {\bf Case: $m\in (-\tau,0)$ and $s\geq 0$:}
   Let $m'\in \N_0$ be such that $m+m'\in (-1,0]$, i.e.,
   $m'=[-m]$. First we consider the case $s=0$. As noted above,
   $ \weight{\xi}^{m'} = {\sum}_{k=0}^n q_k(\xi)p_k(\xi)$,
  for some $q_k\in S^0_{1,0}(\Rn\times \Rn)$ independent of $x$ and polynomials $p_k$ of order at most $m'<\tau$. Hence
  \begin{equation*}
    \OP(a(x,y,\xi))= \OP(a(x,y,\xi)) \weight{D_x}^{m'} \weight{D_x}^{-m'}
= \sum_{k=0}^n \OP(a(x,y,\xi)) p_k(D_x)\tilde{q}_k(D_x),
  \end{equation*}
  where $\tilde{q}_k\in S^{-m'}_{1,0}(\Rn\times \Rn)$ is independent of $x$ and $\tilde{q}_k(D_x)\colon H^{m}_q(\Rn)\to H^{m+m'}_q(\Rn)$. Therefore it remains to show that
  \begin{equation*}
    \OP(a(x,y,\xi)) p_k(D_x)\colon H^{s+m+m'}_q(\Rn)\to H^s_q(\Rn)
  \end{equation*}
  is bounded. Using that $p_k(D_x)$ is a differential operator of order $m'$ and \eqref{eq:CompRight} one obtains the representation
  \begin{equation*}
    \OP(a(x,y,\xi)) p_k(D_x) = {\sum}_{|\alpha|\leq m'} \op (a_{\alpha,k} (x,y,\xi) ) 
  \end{equation*}
  for some $a_{\alpha,k}\in C^{\tau-m'}S^{m+m'}_{1,0}(\R^{2n}\times \Rn)$. Hence the case ``$m\in (-1,0]$ and $s\geq 0$'' implies that
  \begin{equation*}
    \op (a_{\alpha,k} (x,y,\xi))\colon H^{m+m'}_q(\Rn)\to L_q(\R^n).
  \end{equation*}
  Altogether we conclude that
  \begin{equation*}
    \op (a (x,y,\xi) )\colon H^{m}_q(\Rn)\to L_q(\R^n)
  \end{equation*}
  is bounded. Next let $s\in [-m,\tau)$ and $s'=s-m'$. Using
    \begin{equation*}
    u\in H^s_q(\R^n) \quad \text{if and only if}\quad \partial_x^\alpha u\in H^{s'}_q(\Rn)\text{ for all }|\alpha|\leq m'
  \end{equation*}
  for $u= \op (a (x,y,\xi) ) f$ for some $f\in H^{s+m}_q(\Rn)$, it is sufficient to show that
  \begin{equation*}
    \partial_x^\alpha \op (a (x,y,\xi) )\colon H_q^{s+m}(\Rn)\to H_q^{s'}(\R^n)
  \end{equation*}
  is bounded for all $|\alpha|\leq m'$. Similarly as before
   \begin{equation*}
    \partial_x^\alpha \op(a(x,y,\xi)) = \op(a_\alpha(x,y,\xi)) 
  \end{equation*}
  for some $a_\alpha \in
  C^{\tau-|\alpha|}S^{m+|\alpha|}_{1,0}(\R^{2n}\times \Rn)\subset
  C^{\tau-m'}S^{m+m'}_{1,0}(\R^{2n}\times \Rn)$, where  $0\leq s+m
  \leq s'=s-m' <\tau':=\tau-m'$. By the preceding cases, 
  \begin{equation*}
    \op(a_\alpha(x,y,\xi))\colon H^{s+m}_q(\Rn)= H^{s'+m+m'}_q(\Rn)\to H^{s'}_q(\Rn)
  \end{equation*}
  is bounded. Now the mapping properties for general $s\in [0,\tau)$ follow by interpolation between the case $s=0$ and $s\in [-m,\tau)$.

  \medskip
  
\noindent
{\bf Case $s<0$:} By the assumptions of the theorem, $|s|<\tau$ and $|s+m|<\tau$. First we consider the case that additionally $s\in (-\tau, -m)$. 
Since $|m|<\tau$, there are some $s$, which satisfy all these assumptions. Note that in the case $m<0$ this condition is trivial. By the case ``$s\geq 0$'' we obtain that
  \begin{equation*}
    \OP(a(x,y,\xi))^\ast = \OP(\ol{a(y,x,\xi)})\colon H^{-s}_{q'}(\Rn)\to H^{-s-m}_{q'}(\Rn)
  \end{equation*}
  since $s':=-s-m\in (0,\tau)$ and $-\tau<-s=s'+m<\tau$. 
  Hence we conclude by duality that
  \begin{equation*}
    \op (a(x,y,\xi))\colon H^{s+m}_q(\Rn)\to H^s_q(\Rn)
  \end{equation*}
  if additionally $s\in (-\tau,-m)$. Interpolation with the case $s\geq 0$ yields the statement for all $s\in (-\tau, \tau-m)$.
\end{proof*}

 \begin{proof*}{of Theorem  \ref{thm:Bddness2}}
  It is sufficient to consider the case $m\geq 0$ since $ C^\tau S^m_{1,0}(\R^{2n}\times \R^n)\subset  C^\tau S^0_{1,0}(\R^{2n}\times \R^n)$ and $\min(\tau-m,\tau)=\tau$ if $m<0$.
  
  Let us first consider the case $0\leq s< 1$.
  Since $0\leq s<\tau-m$ is arbitrary and $B^{s+\eps}_{q,\infty}(\Rn)\hookrightarrow H^s_q(\Rn)$ for any $\eps>0$, it is sufficient to prove that
  \begin{equation*}
    \op(a(x,y,\xi))\colon L_q(\R^n)\to B^s_{q,\infty}(\R^n)
  \end{equation*}
  is a bounded linear operator for any $0\leq s < \min (\tau-m,\tau)$.  To this end we use that 
  \begin{equation}\label{eq:KernelRepresentation}
    \op(a(x,y,\xi)) u(x) = \int_{\R^n} k(x,y,x-y) u(y) \, dy \qquad \text{for all }u\in \SD(\R^n),x\in\Rn,
  \end{equation}
  where $k\colon \Rn\times \Rn\times (\R^n\setminus \{0\})\to \C$ is smooth with respect to the third variable and satisfies for any $\alpha\in\N_0^n$ and $N\in\N$
  \begin{equation*}
    \|\partial_z^\alpha k(.,.,z)\|_{C^\tau (\R^{2n})} \leq C_{\alpha,N} |z|^{-n-m-|\alpha|}(1+|z|)^{-N}\qquad \text{for all }z\neq 0.
  \end{equation*}
  Here $k$ can be defined as
  \begin{equation*}
    {\sum}_{j\in\N_0} k_j(x,y,z) \qquad \text{for all }x,y,z\in\Rn, z\neq 0, 
  \end{equation*}
  where $k_j(x,y,z)= \F^{-1}_{\xi\mapsto
    z}\left(a(x,y,\xi)\varphi_j(\xi)\right)$ and $\varphi_j\in
  C^\infty_0(\Rn)$, $j\in\N_0$, is a smooth dyadic partition of unity
  as in \eqref{eq:dyadic}ff. The proofs in
  \cite[Section~5.4]{PsDOBuch} carry directly over to the present situation.
  
  Moreover, $\left.(\partial_y^\alpha k)(x,y,z)\right|_{y=x}=0$ since $\left.\partial_y^\alpha a(x,y,\xi)\right|_{y=x}=0$ for all $x,\xi\in\Rn$ and $|\alpha|<\tau$, and we have for any $N\in\N$:
  \begin{alignat}{1}\nonumber
    |k(x,y,z)|&=\Big|k(x,y,z)-{\sum}_{|\alpha|<\tau}\frac1{\alpha!}(\partial_y^\alpha k)(x,x,z)z^\alpha\Big|\\\label{eq:kEstim}
    &\leq C_N|x-y|^{\tau}|z|^{-n-m}(1+|z|)^{-N},
  \end{alignat}
  and in particular
 \begin{alignat}{1}\nonumber
    |k(x,y,x-y)|&\leq C_N|x-y|^{-n-m+\tau}(1+|x-y|)^{-N},
  \end{alignat}
  where $|z|^{-n-m+\tau} (1+|z|)^{-N}$ is in $L_1(\Rn)$ with respect to $z$ for sufficiently large $N\in\N$ since $m <\tau$. 
  Now let $(\Delta_h f)(x):=f(x+h)-f(x)$ for any $x,h\in\Rn$ and $f\colon \Rn \to \C$. Then
   \begin{align*}
    (\Delta_h \op(a(x,y,\xi)) u)(x)
     =&  \int_{|x-y|<2|h|} (k(x+h,y,x+h-y)-k(x,y,x-y))u(y)\,dy\\
      &+   \int_{|x-y|\geq 2|h|} (k(x+h,y,x-y)-k(x,y,x-y))u(y)\,dy     \equiv I_1+I_2.                             
  \end{align*}
  In order to estimate $I_1$ we use that
  \begin{align*}
    |k(x+h,y,x+h-y)-k(x,y,x-y)|&\leq |k(x+h,y,x+h-y)|+ |k(x,y,x-y)|\\
    &\leq C\left(|x+h-y|^{-n-m+\tau} + |x-y|^{-m-m+\tau}\right).
  \end{align*}
  Hence Young's inequality implies
  \begin{align*}
    \|I_1\|_{L_q(\Rn)} &\leq C\int_{|z|\leq 2|h|}(|z+h|^{-n-m+\tau}+ |z|^{-n-m+\tau})\|u\|_{L_q(\Rn)}\\
    &\leq C' |h|^{\tau-m}\|u\|_{L_q(\Rn)}\leq C'|h|^s\|u\|_{L_q(\Rn)}
  \end{align*}
  for any $s<\tau-m$ and $|h|\leq 1$. In order to estimate $I_2$ we use
  \begin{align*}
    I_2
    =&  \int_{|x-y|\geq 2|h|} (k(x+h,y,x+h-y)-k(x+h,y,x-y))u(y)\,dy\\
    &+   \int_{|x-y|\geq 2|h|} (k(x+h,y,x-y)-k(x,y,x-y))u(y)\,dy     \equiv J_1+J_2.                             
  \end{align*}
  For the second integral we use that for any $|x-y|\geq 2|h|$
  \begin{align*}
    |k(x+h,y,x-y)-k(x,y,x-y)|&\leq C_N|h|^{\tau}|x-y|^{-n-m}(1+|x-y|)^{-N}\\
    &\leq C'_N|h|^s |x-y|^{-n-m+\tau-s}(1+|x-y|)^{-N}.
  \end{align*}
  Therefore
  \begin{align*}
    \|J_2\|_{L_q(\Rn)} \leq C_N|h|^s\int_{\Rn}|z|^{-n-m+\tau-s}(1+|z|)^{-N}\, dz\|u\|_{L_q(\Rn)}\leq C'_s|h|^s\|u\|_{L_q(\Rn)}
  \end{align*}
  for any $s<\tau-m$ and $|h|\leq 1$ and sufficiently large $N$. Furthermore, for any $|x-y|\geq 2|h|$ 
  \begin{align*}
    &|k(x+h,y,x+h-y)-k(x+h,y,x-y)|\\
    &= \left|\int_0^1 D_zk(x+h,y,x+sh-y)\, ds\, h\right|
     \leq C_N|x+h-y|^\tau |x-y|^{-n-m-1}(1+|x-y|)^{-N}|h|\\
    &\leq C'_N|h|^s |x-y|^{-n-m+\tau-s}(1+|x-y|)^{-N}.
  \end{align*}
  by \eqref{eq:kEstim} with $k$ replaced by $D_z k$ and since
  \begin{equation*}
    \tfrac12 |x-y|\leq |x-y|-|h|\leq |x+sh-y|\leq |x-y|+|h|\leq \tfrac32 |x-y| 
  \end{equation*}
  for every $s\in [0,1]$. Thus we obtain as before
  \begin{align*}
    \|J_1\|_{L_q(\Rn)} \leq C_N|h|^s\int_{\Rn}|z|^{-n-m+\tau-s}(1+|z|)^{-N}\, dz\|u\|_{L_q(\Rn)}\leq C'_s|h|^s\|u\|_{L_q(\Rn)}
  \end{align*}
  for any $s<\tau-m$, $|h|\leq 1$, and suitable $N\in\N$.
  Altogether we obtain
  \begin{equation*}
    \|\Delta_h \op(a(x,y,\xi)) u\|_{L_q(\Rn)}\leq C|h|^s \|u\|_{L_q(\Rn)}
  \end{equation*}
  uniformly in $|h|\leq 1$ for any $s<\tau-m$. Moreover, one obtains by similar, but simpler estimates
  \begin{equation*}
    \|\op(a(x,y,\xi)) u\|_{L_q(\Rn)}\leq C\|u\|_{L_q(\Rn)}.
  \end{equation*}
  This implies the boundedness of $\op(a(x,y,\xi))\colon L_q(\Rn)\to B^s_{q,\infty}(\Rn)$ for any $0\leq s<\tau-m$.

  Finally, let $s\geq 1$. 
  Now let $k\in\N$ such that $s=k+s'$ with $s'\in [0,1)$. We use that $u\in H^s_q(\Rn)$ if and only if $\partial_x^\alpha u \in H^{s'}_q(\Rn)$ for every $|\alpha|\leq k$. For $|\alpha|\leq k$ we have
  \begin{equation*}
    \partial_x^\alpha \op(a(x,y,\xi))= \sum_{0\leq \beta \leq \alpha} \binom{\alpha}{\beta} \op (\partial_x^\beta a(x,y,\xi)(i\xi)^{\alpha-\beta}),
  \end{equation*}
  where $\partial_x^\beta a(x,y,\xi)(i\xi)^{\alpha-\beta}\in C^{\tau-|\beta|} S^{m+|\alpha|-|\beta|}_{1,0}(\R^{2n}\times \Rn)$. Because of
  $$
  0\leq s'=s-k<s-|\alpha|\leq \tau-m-|\alpha|=(\tau-|\beta|)-(m+|\alpha|-|\beta|),
  $$
  we can apply the case ``$0\leq s<1$'' to $\partial_x^\beta a(x,y,\xi)(i\xi)^{\alpha-\beta}$ (with $s'$ instead of $s$) and obtain the result.
\end{proof*}
   
\footnotesize



\def\cprime{$'$} \def\ocirc#1{\ifmmode\setbox0=\hbox{$#1$}\dimen0=\ht0
  \advance\dimen0 by1pt\rlap{\hbox to\wd0{\hss\raise\dimen0
  \hbox{\hskip.2em$\scriptscriptstyle\circ$}\hss}}#1\else {\accent"17 #1}\fi}

\end{document}